\documentclass[reqno]{amsart}
\usepackage{amsmath,latexsym,amssymb}
\usepackage[mathscr]{eucal}
\usepackage{enumitem}
\usepackage{caption}
\usepackage{subcaption}
\usepackage{color}

\theoremstyle{plain}
\newtheorem{theorem}{Theorem}[section]
\newtheorem{lemma}[theorem]{Lemma}
\newtheorem{corollary}[theorem]{Corollary}
\newtheorem{proposition}[theorem]{Proposition}
\newtheorem{NonSplit}[theorem]{Non-splitting Lemma} %bd added

\theoremstyle{definition}
\newtheorem{definition}[theorem]{Definition}
\newtheorem{remark}[theorem]{Remark}

 {\begin{proof}[Proof of Claim~#1]}%
 {\end{proof}}

\usepackage{tikz,pgf}
\usetikzlibrary{calc,arrows,shapes,backgrounds,fit,intersections,decorations.pathmorphing,decorations.pathreplacing,positioning}

\tikzset{%
 element/.style={draw,circle,fill=white,inner sep=1.5pt},
 shaded element/.style={draw,circle,fill=black!30,inner sep=1.5pt},
 potato/.style={rounded rectangle,draw,densely dotted,inner sep=1.5pt},
 operation/.style={semithick,densely dotted,shorten >=3pt,shorten <=3pt,>=latex},
 loopy operation/.style={semithick,densely dotted,shorten >=3pt,shorten <=3pt,>=latex, min distance=18pt},
 graph/.style={thin,shorten >=3pt,shorten <=3pt},
 loopy graph/.style={thin,shorten >=3pt,shorten <=3pt, min distance=18pt},
 move up/.style= {transform canvas={yshift=2pt}},
 move down/.style={transform canvas={yshift=-2pt}},
 move far down/.style={transform canvas={yshift=-4pt}},
 move left/.style= {transform canvas={xshift=-2pt}},
 move right/.style={transform canvas={xshift=2pt}},
 auto}

\DeclareMathOperator{\Var}{Var}
\DeclareMathOperator{\Mod}{Mod}

\DeclareMathOperator{\HSP}{\mathbb{HSP}}

\newcommand{\dotcup}{\mathbin{\dot\cup}}
\newcommand{\down}{{\downarrow}}
\newcommand{\up}{{\uparrow}}
\newcommand{\restrictedto}[1]{{\upharpoonright_{#1}}} 
\newcommand{\A}{{\mathbf A}}
\newcommand{\B}{{\mathbf B}}
\newcommand{\F}{{\mathbf F}}
\newcommand{\Lb}{{\mathbf L}}
\newcommand{\Pb}{{\mathbf P}}
\newcommand{\Sb}{{\mathbf S}}
\newcommand{\T}{{\mathbf T}}
\newcommand{\X}{{\mathbf X}} %bd added
\newcommand{\Y}{{\mathbf Y}} %bd added

\newcommand{\cat}[1]{\boldsymbol{\mathscr{#1}}}
\newcommand{\CV}{\cat V}
\newcommand{\CA}{\cat A}
\newcommand{\CB}{\cat B}
\newcommand{\LV}{\cat L(\CV)} 
\newcommand{\SifinV}{\mathbf{Si}_{\mathrm{fin}}(\CV)}
\renewcommand{\phi}{\varphi}
\renewcommand{\ge}{\geqslant}
\renewcommand{\le}{\leqslant}
\renewcommand{\leq}{\leqslant}
\renewcommand{\geq}{\geqslant}

\renewcommand{\nleq}{\not\leqslant}

\newcommand{\bd}{\mathbf{Dn}} %bdnew added
\newcommand{\OSi}{\bd(\SifinV)}
\newcommand{\uu}{\mathrm{Up}}
\newcommand{\ut}{\uu^\mathcal{T}\!}
\newcommand{\bu}{\mathbf{Up}}
\newcommand{\but}{\bu^{\!\mathcal{T}}\!}%bdnew kill some space  before superscript

\DeclareMathOperator{\SH}{\mathbb{SH}}
\DeclareMathOperator{\SPU}{\mathbb{SP_U}}
\DeclareMathOperator{\HS}{\mathbb{HS}}
\DeclareMathOperator{\IS}{\mathbb{IS}}

\newcommand{\CR}{\cat R}
\newcommand{\CW}{\cat W}

\def\ov#1{\overline{#1}}
\def\ra{\rightarrow}
\def\Ra{\Rightarrow}
\def\iff{\quad\Longleftrightarrow\quad}
\def\lres{\backslash}
\def\rres{/}

\def\eqv{\leftrightarrow}
\def\up{\mathord{\uparrow}}
\def\dw{\mathord{\downarrow}}

\let\lra=\eqv
\let\m=\mathbf
\newcommand{\lan}{(}
\newcommand{\ran}{)}
\let\ld=\lres
\let\rd=\rres

\newcommand{\comp}{{\setminus}}
\newcommand{\updown}{{\updownarrow}}
\newcommand{\join}{\vee}
\newcommand{\meet}{\wedge}
\newcommand{\lr}{\leftrightarrow}
\newcommand{\arrow}{\mathbin{\searrow}}

\newcommand{\codom}{\operatorname{codom}}
\newcommand{\dpc}{{\sim}}
\newcommand{\cg}{\operatorname{Cg}}
\newcommand{\con}{\operatorname{Con}}

\newcommand{\ub}[1]{{\mathbf{#1}}}

\newcommand{\h}[1][algebra]{H\textsuperscript{+}-#1} 
\newcommand{\ps}{\neg\dpc}

\newcommand{\tightoverset}[2]{
\mathop{#2}\limits^{\vbox to -.5ex{\kern-0.2ex\hbox{$#1$}\vss}}
}
\providecommand{\dotdiv}{% Don't redefine it if available
  \mathbin{% We want a binary operation
    \vphantom{+}% The same height as a plus or minus
    \text{% Change size in sub/superscripts
      \mathsurround=0pt % To be on the safe side
      \ooalign{% Superimpose the two symbols
        \noalign{\kern-.35ex}% but the dot is raised a bit
        \hidewidth$\smash{\cdot}$\hidewidth\cr % Dot
        \noalign{\kern.35ex}% Backup for vertical alignment
        $-$\cr % Minus
      }%
    }%
  }%
}
\tikzset{
fence node/.style={scale=0.4,fill=white,draw,circle}
}
\newcommand{\tikzfence}[1]{
\foreach \x[
evaluate=\x as \xi using {0.5*\x-0.5},
evaluate=\x as \yi using {mod(\x+1,2)}
] in {1,...,#1} {
\draw (\xi,\yi) node[fence node] (fence\x) {};
}
\draw \foreach \x [remember=\x as \lastx (initially 1)] in {2,...,#1}{(fence\lastx) -- (fence\x)};
}

\newcommand{\tikzfencedual}[1]{
\foreach \x[
evaluate=\x as \xi using {0.5*\x-0.5},
evaluate=\x as \yi using {mod(\x,2)}
] in {1,...,#1} {
\draw (\xi,\yi) node[fence node] (\x) {};
}
\draw \foreach \x [remember=\x as \lastx (initially 1)] in {2,...,#1}{(\lastx) -- (\x)};
}

\begin{document} 

\author[B.A. Davey, T. Kowalski and C.J. Taylor]{Brian A. Davey, Tomasz Kowalski and Christopher J. Taylor}
\title{Splittings in varieties of logic}
\address{Mathematics and Statistics, La Trobe University, Victoria 3085, Australia}
\email[Davey]{B.Davey@latrobe.edu.au}
\email[Kowalski]{T.Kowalski@latrobe.edu.au}
\email[Taylor]{Chris.Taylor@latrobe.edu.au}

\begin{abstract}
We study splittings, or lack of them, in lattices of subvarieties of some
logic-related varieties. We present a general lemma, the Non-Splitting Lemma,
which when combined with some variety-specific constructions, yields each of our
negative results: the variety of commutative integral residuated lattices
contains no splittings algebras, and in the varieties of double Heyting
algebras, dually pseudocomplemented Heyting algebras and regular double
p-algebras the only splitting algebras are the two-element and three-element
chains. 
\end{abstract}

\maketitle
%%%%%%%%%%%%%%%%%%%%%%%%%%%%%%%%%%%%%%%%%%%%%%%%%%%%%%%%%%%%%%%%%%%%%%%
\section{Introduction}\label{sec:intro}
%%%%%%%%%%%%%%%%%%%%%%%%%%%%%%%%%%%%%%%%%%%%%%%%%%%%%%%%%%%%%%%%%%%%%%%

A very natural divide-and-conquer method of studying a lattice $\mathbf{L}$ is to 
dismantle it into a disjoint pair of a principal filter and a principal
ideal. If such a \emph{splitting} is possible, then the structure of
$\mathbf{L}$ is completely determined by the filter, the ideal, and the way
they are put together to make up $\mathbf{L}$. This concept of splitting was
introduced by Whitman~\cite{Whi43}, and later used by McKenzie~\cite{McK} to
investigate the lattice of varieties of lattices. In fact, several of McKenzie's 
results apply to lattices of subvarieties of any variety $\CV$ -- we will
frequently use one of these results.

Logic-related applications of
splittings began with Jankov~\cite{Jan63} who used splittings to investigate
the lattice of superintuitionistic logics.
%extensions of the intuitionistic logic (equivalently, varieties
%of Heyting algebras).
Jankov's results were extended in various ways for other
classes of logics, notably modal and superituitionistic logics,
for which splitting methods proved to be very useful.
We refer the interested reader to Chagrov, Zakhariaschev~\cite{CZ97}
and Kracht~\cite{Kra99} for surveys and much more. Beyond superintuitionistic
and modal logics, splittings are not too common.

%Splitting methods obviously work best if there are a lot of splittings.
In the lattice of subvarieties of a variety $\CV$
every splitting is induced by a single subdirectly irreducible algebra;
such algebras are called \emph{splitting algebras}. 
McKenzie~\cite{McK} proved that if $\CV$ is congruence distributive and
generated by its finite members, then every splitting algebra in $\CV$ is
finite. Day~\cite{Day} showed that if $\CV$ is congruence distributive and
locally finite, then the converse is true as well: every finite subdirectly
irreducible algebra is a splitting algebra. 

Blok, Pigozzi~\cite{BP82} showed that in varieties with \emph{equationally
  definable principal congruences} (EDPC) every finitely presented subdirectly
irreducible algebra is a splitting algebra.
Since EDPC implies congruence distributivity but
not local finiteness, and congruence distributivity together with local
finiteness do not imply EDPC, the results of Day~\cite{Day} and of
Blok, Pigozzi~\cite{BP82} complement each other.  

Our goal in this article lies in the opposite direction. We will exhibit a
number of varieties 
(congruence distributive, related closely to logics, and generated by
their finite members) for which very few splittings exist. This will give some
empirical evidence for the claim that the assumptions of local finiteness or
EDPC are optimal, and no better general results can be expected.

%tk added the paragraph below
A class of logics to which our results immediately apply is the class of
\emph{substructural logics}, whose 
algebraic semantics is the variety of \emph{residuated lattices}.
Substructural logics form a comprehensive class: they include superitutionistic
logics, linear logic, relevant logics and many-valued logics.
Residuated lattices have a rich and complex theory, at one end connected to
classical algebra (via idempotent semirings and lattice-ordered groups),
and at the other to proof theory (via sequent systems). For more on residuated
lattices we refer the reader to Galatos \emph{et al.}~\cite{GJKO07}.

Relative pseudocomplements play a crucial role both in varieties with EDPC
(principal congruences of algebras in such varieties have relative
pseudocomplements), and in locally finite varieties (see the next section).
Thus, we will begin in Sections~\ref{sec:split-pcompl}
and~\ref{sec:split-latt-svar} with a few general results connecting relative
pseudocomplements and 
splittings, from which in particular Day's result directly follows.

In Section~\ref{sec:split-alg-var-log} we formalise what we mean by a
\emph{variety of logic} and present a basic result, the 
Non-splitting Lemma~\ref{no-splitting}, that we shall use repeatedly to prove
our non-splitting theorems.   
Indeed, all negative results on splittings in varieties of logics, known to the
authors, can be viewed as applications of the Non-splitting Lemma. 
Using the lemma we will prove a number of new negative results stating that in a
certain variety $\CV$ no algebra is splitting, except for those on a (short,
finite) list. Each of these results involves a pair 
of constructions: an expansion followed by a distortion.

In Section~\ref{sec:res-latt} we prove that the variety $\mathsf{CIRL}$ of
commutative integral residuated lattices contains no splittings algebras at all
--  in fact, if $\CR$ is a variety of residuated lattices that contains
$\mathsf{CIRL}$, then no finite algebra from $\mathsf{CIRL}$ is splitting
in~$\CR$ (Corollary~\ref{outside-CIRL}). In Sections~\ref{sec:d-Heyting}
and~\ref{sec:subvarofH+} we turn our attention to three cousins of Heyting
algebras, namely the varieties $\mathsf{DH}$ of double Heyting algebras,
$\mathsf{H}^+$ of dually pseudocomplemented Heyting algebras and  $\mathsf{RDP}$
of regular double p-algebras. Jankov~\cite{Jan63} proved that in the variety
$\mathsf{H}$ of Heyting algebras every finite subdirectly irreducible is a
splitting algebra. In stark contrast, we prove that in each of $\mathsf{DH}$,
$\mathsf{H}^+$ and $\mathsf{RDP}$ the only splitting algebras are the
two-element and three-element chains
%tk combined three into one
% (Corollaries~\ref{cor:DH}, \ref{cor:H+}, and~\ref{cor:RDP}).
(Corollarys~\ref{cor:DH-etal}).
Unlike the proof for the variety $\mathsf{CIRL}$, which is purely algebraic, the proofs for the varieties $\mathsf{DH}$, $\mathsf{H}^+$ and $\mathsf{RDP}$ use the restricted Priestley duality for each of the varieties.

%%%%%%%%%%%%%%%%%%%%%%%%%%%%%%%%%%%%%%%%%%%%%%%%%%%%%%%%%%%%%%%%%%%%%%%
\section{Splittings and relative pseudocomplements}\label{sec:split-pcompl}
%%%%%%%%%%%%%%%%%%%%%%%%%%%%%%%%%%%%%%%%%%%%%%%%%%%%%%%%%%%%%%%%%%%%%%%

\begin{definition}\label{def:relpscomp}
Let $\Lb$ be a lattice and let $a, b\in L$. The relative pseudocomplement $b \to a$ and dual relative pseudocomplement $b \dotdiv a$ are defined by
\begin{alignat*}{3}
		x \meet b \le a & \iff  x \le b \to a, &&\quad\text{or}\quad &b \to a &=\max\{\, y\in L\mid y \wedge b =a \wedge b\,\},\\
		x \join a \ge b & \iff  x \ge b \dotdiv a, &&\quad\text{or}\quad  &b \dotdiv a &=\min\{\, y\in L\mid a \vee y = a \vee b\,\}.
\end{alignat*}
(Some authors  write $x \leftarrow y$ or $x \Leftarrow y$ instead of $y \dotdiv x$.)
A \emph{Heyting algebra} is an algebra $\langle A;\join,\meet,\to,0,1\rangle$ such that $\langle A;\join,\meet,0,1\rangle$ is a bounded lattice and $\to$ is a relative pseudocomplement operation; a \emph{dual Heyting algebra} is defined analogously.
A \emph{double Heyting algebra} is an algebra $\langle A;\join,\meet,\to,\dotdiv,0,1\rangle$ such that 
${\langle A;\join,\meet,\to,0,1\rangle}$ is a Heyting algebra and ${\langle A;\join,\meet,\dotdiv,0,1\rangle}$ is a dual Heyting algebra.
Note that the underlying lattice of a (double) Heyting algebra is necessarily distributive.
\end{definition}

For each element $x$ of an ordered set $\X = \langle X; \leq\rangle$ we define $\down x :=\{y\in X \mid y\leq x\}$ and $\up x :=\{y\in X \mid y\geq x\}$. 
For basic lattice-theoretic concepts, such as join-irreducible, join-dense and algebraic lattice, we refer to Davey and Priestley~\cite{ilo}.

\begin{definition}\label{def:splitting}
A pair $(c, d)\in L^2$ is a \emph{splitting pair} (in $\Lb$) if $L = \up
c\dotcup \down d$.
\end{definition}

The following two lemmas are completely straightforward to prove.

\begin{lemma}\label{lem:easy1}
Let $\Lb$ be a complete lattice and let $c, d\in L$. The following are equivalent:

\begin{enumerate}[label={\upshape(\arabic*)},leftmargin=1.75\parindent]

\item $(c, d)$ is a splitting pair,

\item $c$ is completely join-prime and $d = \bigvee\{\, x\in L\mid x\not\ge c\,\}$, 

\item $d$ is completely meet-prime and $c = \bigwedge\{\, x\in L\mid x\not\le d\,\}$. 
\end{enumerate}
\end{lemma}

Given $x,y$ in an ordered set $\X$, we write $x \prec y$ if $x$ is covered by $y$ in~$\X$.

\begin{lemma}\label{lem:easy2}
Let $(c, d)$ be a splitting pair in a lattice $\Lb$. Then

\begin{enumerate}[label={\upshape(\arabic*)},leftmargin=1.75\parindent]

\item $c\wedge d \prec c$, and

\item $c \to (c\wedge d) = d$.

\end{enumerate}
\end{lemma}

We see from Lemma~\ref{lem:easy2} that every splitting pair gives rise to a cover and a corresponding relative pseudocomplement. Our first aim is to prove a form of converse. The following easy lemma will be useful.

\begin{lemma}\label{lem:veryeasy}
Let $a, b$ be elements of a lattice $\Lb$ such that $a \prec b$ and assume that $b\to a$ exists in~$\Lb$. For all $x\in L$ with $x\vee a \not\ge b$, we have $x\le b\to a$.
\end{lemma}

\begin{proof}
Let $x\in L$ with $x\vee a \not\ge b$. Thus $(x \vee a)\wedge b < b$. As $a \le (x \vee a)\wedge b \le b$ and $a\prec b$, it follows that $(x \vee a)\wedge b = a$, and hence
$x \le x \vee a \le b \to a$, as required.
\end{proof}

\begin{lemma}\label{lem:main}
Let $a, b$ be elements of a lattice $\Lb$ such that $a \prec b$ and assume that $d:=b\to a$ exists in~$\Lb$. If $c$ is a join-prime element of $\Lb$ with $c\le b$ and $c\not\le a$, then $(c, d)$ is a splitting pair with $b\in \up c$ and $a\in \down d$. 
\end{lemma}

\begin{proof} 
Let $c$ be a join-prime element of $\Lb$ with $c\le b$ and $c\not\le a$. 
We first prove that $\up c \cap \down d = \varnothing$. 
Suppose that $\up c \cap \down d\ne \varnothing$. 
Then we have $c \le d = b \to a$ and hence $c = c\wedge b \le a$, a contradiction. Hence $\up c \cap \down d = \varnothing$.

We now prove that $\up c \cup \down d = L$. Let $x\in L$ and assume that $x\not\ge c$. We shall prove that $x\le d$. As $a\not\ge c$ and $c$ is join-prime, we have $x\vee a\not\ge c$, and hence $x\vee a \not\ge b$, as $b\ge c$. Thus, by Lemma~\ref{lem:veryeasy}, $x \le b\to a =d$, 
as required.
\end{proof}

The following corollary is an immediate consequence of Lemmas~\ref{lem:easy1} and~\ref{lem:main}.

\begin{corollary}\label{cor:uniquejp}
Let $a, b$ be elements of a lattice $\Lb$ such that $a \prec b$ and assume that $b\to a$ exists in~$\Lb$. Then there is at most one join-prime element $c$ of $\Lb$ with $c\le b$ and $c\not\le a$ and such an element $c$ is necessarily completely join-prime.
\end{corollary}

We obtain the following result of Ne\v set\v ril, Pultr and Tardif~\cite{NPT} as an easy consequence of Lemma~\ref{lem:main}. 

\begin{theorem}\label{thm:Heyt-cov2splitting}
Assume that $\Lb$ is a Heyting algebra in which the join-irreducible elements are join-dense. Let $a, b\in L$ with $a \prec b$. Then there exists a splitting pair $(c, d)$ in $\Lb$ with $b\in \up c$ and $a\in \down d$. 

\end{theorem}

A strengthening of the assumption that the join-irreducible elements are join-dense will guarantee that a lattice forms a dual Heyting algebra.

\begin{lemma}\label{lem:dualrelpseudocomps}
Let $\Lb$ be a lattice in which each element is the join of a finite set of join-prime elements. Then all dual relative pseudocomplements exist in~$\Lb$.
\end{lemma}

\begin{proof}
It suffices to prove that $b \dotdiv a$ exists for all $a\le b$ in~$\Lb$, so let $a, b\in L$ with $a\le b$. By assumption, there are finite sets $A$ and $B$ of join-prime elements such that $a = \bigvee A$ and $b = \bigvee B$. Let
\[
F_a :=\{\, x\in A\cup B \mid x \le a\,\} \quad \text{and} \quad F_c :=\{\, x\in A\cup B \mid x \not\le a\,\};
\]
then $\bigvee F_a = a$ and $\bigvee (F_a \cup F_c) = b$. Define $c := \bigvee F_c$. We claim that $c = b \dotdiv a$. We must prove that, for all $x\in L$,  
\[
x\vee a \ge b \iff x\ge c.
\]
Let $x\in L$. We have
\[
x\ge c \implies x \vee a \ge c \vee a = \bigvee F_c \vee \bigvee F_a = \bigvee (A\cup B) = b.
\]
Now assume that $x \vee a \ge b$. Let $y\in F_c$; so $y\le b$ and $y\not\le a$. Hence $x \vee a \ge b \ge y$. As $y\not\le a$ and $y$ is join-prime, we have $y\le x$. Thus, $x$ is an upper bound of $F_c$ and so $x \ge \bigvee F_c = c$. Hence, $c = b \dotdiv a$.\end{proof}

If both $b \to a$ and $b \dotdiv a$ exist in $\Lb$, then in Lemma~\ref{lem:main} we can drop the requirement that $a$ and $b$ are separated by a join-prime element of~$\Lb$.

\begin{lemma}\label{lem:altmain}
Let $a, b$ be elements of a lattice $\Lb$ such that $a \prec b$ and assume that both $c:= b \dotdiv a$ and $d:=b\to a$ exist in~$\Lb$. Then $(c, d)$ is a splitting pair with $b\in \up c$ and $a\in \down d$. In particular, $b \dotdiv a$ is completely join-prime and $b\to a$ is completely meet-prime.
\end{lemma}

\begin{proof} A very simple calculation using the definitions of $b \dotdiv a$ and $b\to a$ shows that
\[
b \dotdiv a \le b \to a \ \iff \ a \ge b \ \iff \ (b \dotdiv a = 0 \And b\to a = 1).
\]
Since $a \prec b$, we therefore have $b \dotdiv a \not\le b \to a$, and it remains to show that $\up(b \dotdiv a)\cup \down(b\to a) = L$. 
Let $x\in L$ with $b \dotdiv a \not\le x$. As $b \dotdiv a \not\le x$, we have $b\not\le x \vee a$, and consequently, by Lemma~\ref{lem:veryeasy},  $x\le b\to a$, 
as required.
\end{proof}

The next result is an immediate corollary.

\begin{theorem}\label{thm:DoubHeyt-cov2splitting}
Let $\Lb$ be a double Heyting algebra, let $a, b\in L$ with $a \prec b$ and define $c:= b \dotdiv a$ and $d:=b\to a$. Then $(c, d)$ is a splitting pair with $b\in \up c$ and $a\in \down d$. In particular, $b \dotdiv a$ is completely join-prime and $b\to a$ is completely meet-prime.
\end{theorem}

If $\Lb$ is a Heyting algebra (or double Heyting algebra) and $u < v$ in $\Lb$, then we denote the induced Heyting algebra (or double Heyting algebra) on the interval $[u, v]$ by $\Lb_{uv}$. 

\begin{corollary}\label{cor:intervals}
Assume that $\Lb$ is a Heyting algebra in which the join-irreducible elements are join-dense or that $\Lb$ is a double Heyting algebra, and let $u < v$ in~$\Lb$. For all $a, b\in [u, v]$ with $a \prec b$, there exists a splitting pair $(c, d)$ in $\Lb_{uv}$ with $b\in \up_{\Lb_{uv}} c$ and $a\in \down_{\Lb_{uv}} d$. 
\end{corollary}

\begin{proof}
Let $a, b\in [u, v]$ with $a \prec b$. By Theorem~\ref{thm:Heyt-cov2splitting} or Theorem~\ref{thm:DoubHeyt-cov2splitting} there exists a splitting pair $(c', d')$ in $\Lb$ with $b\in \up_\Lb c'$ and $a\in \down_\Lb d'$. It is easy to check that $(c'\vee u, d'\wedge v)$ is the required splitting pair in $\Lb_{uv}$.
\end{proof}

\begin{remark}
This corollary can be used (in the contrapositive) to show that an interval in a Heyting algebra or double Heyting algebra is \emph{dense}, that is, contains no covers. For example, the homomorphism lattice $\mathcal G$ of finite symmetric graphs forms a Heyting algebra in which the join-irreducibles (the finite connected graphs) are join dense. In fact, since every finite graph is the disjoint union of its connected components, an application of Lemma~\ref{lem:dualrelpseudocomps} shows that $\mathcal G$ forms a double Heyting algebra. An easy application of a deep result of Erd\H os~\cite{Erdos} shows that the interval in $\mathcal G$ above its unique atom contains no splitting pair. Hence, by Corollary~\ref{cor:intervals}, the interval in $\mathcal G$ above the unique atom is dense -- see Ne\v set\v ril and Tardif \cite{NT} and Ne\v set\v ril, Pultr and Tardif~\cite{NPT}.
\end{remark}

%%%%%%%%%%%%%%%%%%%%%%%%%%%%%%%%%%%%%%%%%%%%%%%%%%%%%%%%%%%%%%%%%%%%%%%
\section{Splittings in a lattice of subvarieties}\label{sec:split-latt-svar}
%%%%%%%%%%%%%%%%%%%%%%%%%%%%%%%%%%%%%%%%%%%%%%%%%%%%%%%%%%%%%%%%%%%%%%%

In this section we investigate the applicability of the results of the previous
section to the lattice $\LV$ of subvarieties of a variety~$\CV$. We will use, without further comment, the fact that $\LV$ is always a dually algebraic lattice.
Splittings in subvariety lattices have been intensively studied since the foundational paper of McKenzie~\cite{McK}. The following lemma is well known and follows easily from Lemma~\ref{lem:easy1} and the facts  that every variety is generated by its finitely generated subdirectly irreducible members and is the class of all models of a set of equations.  

\begin{lemma}\label{lem:Birk}
Let $(\CA, \CB)$ be a splitting pair in the lattice $\LV$ of subvarieties of
some variety~$\CV$. Then $\CV = \Var(\A)$, for some finitely generated
subdirectly irreducible algebra $\A$ and $\CB =
\Mod(\Sigma\cup\{\varepsilon\})$, where $\Sigma$ is a set of equations such that
$\CV = \Mod(\Sigma)$ and $\varepsilon$ is a single equation. 
\end{lemma}

Resulting from this lemma (and Lemma~\ref{lem:easy1}), a finitely generated subdirectly irreducible algebra in a variety $\CV$ such that $\Var(\A)$ is completely join-prime in $\LV$ is called a \emph{splitting algebra} in~$\CV$.

Recall that a complete lattice underlies a Heyting algebra if and only if it satisfies the \emph{join-infinite distributive law} (JID). We will use both this fact and its dual in the discussion below.

In order to apply Theorem~\ref{thm:Heyt-cov2splitting} to the lattice $\LV$, we need $\LV$ to form a Heyting algebra. In fact, in that case $\LV$ will 
form a double Heyting algebra and therefore Theorem~\ref{thm:DoubHeyt-cov2splitting} will also be applicable. Indeed, if $\LV$ forms a Heyting algebra, then it will be a distributive and dually algebraic lattice, and consequently will satisfy the meet-infinite distributive law, whence it also forms a dual Heyting algebra. Before stating the characterisation of when $\LV$ forms a Heyting algebra, we note that congruence distributivity is sufficient to guarantee that the join-irreducible subvarieties are join-dense in~$\LV$, a condition necessary to be able to apply Theorem~\ref{thm:Heyt-cov2splitting}.

\begin{lemma}\label{lem:SIjoindense}
Let $\CV$ be a congruence-distributive variety.

\begin{enumerate}[label={\upshape(\arabic*)},leftmargin=1.75\parindent]

\item $\Var(\A)$ is join-prime and therefore join-irreducible in $\LV$, for every subdirectly irreducible algebra $\A$ in~$\CV$.

\item The join-irreducible elements are join-dense in~$\LV$.

\end{enumerate}
\end{lemma}

\begin{proof}
(1) follows from the fact that, for all subvarieties $\CV_1$ and $\CV_2$ of~$\CV$, a subdirectly irreducible algebra belongs to $\CV_1 \vee \CV_2$ if and only if it belongs to either $\CV_1$ or $\CV_2$ (see J\'onsson~\cite[Lemma~4.1]{Jon}). Since every variety is generated by its subdirectly irreducible algebras, (2) is an immediate consequence of~(1).
\end{proof}

We do not know of an intrinsic characterisation of varieties $\CV$ such that $\LV$  
forms a Heyting algebra (and therefore is a double Heyting algebra). Nevertheless, the following theorem gives a characterisation of when $\LV$ forms a Heyting algebra, in terms of the lattice itself.

Other than for the inclusion of (1), the result is given in Davey~\cite{D79} where it is derived as an immediate consequence of the characterisation of down-set lattices -- see Davey~\cite[Proposition~1.1]{D79} and Davey and Priestley~\cite[Theorem~10.29]{ilo}. A subset $Y$ of an ordered set $\X$ is a \emph{down-set} if $\down x \subseteq Y$, for all $x\in Y$. 
%bd Following tradition, we 
We denote the lattice of all down-sets of $\X$ by $\bd(\X)$.

\goodbreak

\begin{theorem}\label{thm:OP}
Let $\CV$ be a variety. The following are equivalent\textup:

\begin{enumerate}[label={\upshape(\arabic*)},leftmargin=1.75\parindent]

\item $\LV$ forms a Heyting algebra \textup(and therefore a double Heyting algebra\textup)\textup;

\item $\LV$ satisfies \textup{(JID)}\textup;

\item $\LV$ is distributive and algebraic\textup;

\item $\LV$ is completely distributive\textup;

\item every completely join-irreducible element of $\LV$ is completely join-prime\textup;

\item the completely join-prime elements are join-dense in~$\LV$\textup;

\item $\LV$ is isomorphic to $\bd(\cat P)$ via $\CA \mapsto  \{\, \CB\in \cat P \mid \CB \subseteq \CA\,\}$, where $\cat P$ is the ordered set of subvarieties of $\CV$ that are completely join-prime in~$\LV$.

\end{enumerate}
\end{theorem}

The following lemma gives a simple sufficient condition for $\LV$ to be algebraic.

\begin{lemma}\label{lem:locfin}
Let $\CV$ be a locally finite variety. Then $\LV$ is an algebraic lattice. A subvariety $\CA$ of $\CV$ is compact in $\LV$ if and only if it is finitely generated, that is, $\CA = \Var(\A)$ for some finite algebra~$\A$.
\end{lemma}

\begin{proof}
Every variety is generated by its finitely generated members. Hence, since $\CV$ is locally finite, every subvariety of $\CV$ equals the join in $\LV$ of finitely generated subvarieties. Thus, it remains to prove that $\CA$ is compact in $\LV$ if and only if it is finitely generated. 

Assume that $\CA$ is compact in $\LV$. Since $\CA$ is the join in $\LV$ of its finitely generated subvarieties, it follows that there are finitely many finite algebras $\A_1, \dots, \A_n$ in $\CA$ such that 
\[
\CA = \Var(\{\A_1, \dots, \A_n\}) = \Var(\A_1\times \dots\times \A_n),
\]
whence $\CA$ is finitely generated. Now assume that $\CA = \Var(\A)$, for some finite algebra~$\A$. We shall prove that $\CA$ is compact in $\LV$. Let $\CV_i$ be subvarieties of $\CV$, for $i\in I$, with $\CA \subseteq \bigvee_{i\in I} \CV_i$. Hence $\A\in \Var(\bigcup_{i\in I} \CV_i) = \HSP(\bigcup_{i\in I} \CV_i)$. As $\A$ is finite, it follows that $\A$ is a homomorphic image of a finitely generated (and therefore finite) subalgebra $\B$ of a product $\prod_{s\in S} \A_s$, with $\A_s \in \bigcup_{i\in I} \CV_i$, for all $s\in S$. As $\B$ is finite, there is a finite subset $T$ of $S$ such that $\B$ embeds into $\prod_{t\in T} \A_t$. It follows at once that there is a finite subset $J$ of $I$ such that $\A \in \HSP(\bigcup_{j\in J} \CV_j)$, whence $\CA \subseteq \bigvee_{j\in J} \CV_j$. Hence $\CA$ is compact in $\LV$, as claimed.
\end{proof}

By combining Theorem~\ref{thm:OP} and Lemmas~\ref{lem:easy1}, \ref{lem:Birk}, \ref{lem:SIjoindense} and~\ref{lem:locfin} with some simple applications of J\'onsson's Lemma~\cite{Jon}, we obtain the following result. Most of this was already known, but our proofs are simpler and more direct. For example, (3) was first proved by Day~\cite[Corollary~3.8]{Day} using the concept of a finitely projected algebra and (4) was first proved in Davey~\cite[Theorem~3.3]{D79}.

\begin{theorem}\label{thm:fingenCD}
Let $\CV$ be a locally finite congruence-distributive variety. 
\begin{enumerate}[label={\upshape(\arabic*)},leftmargin=1.75\parindent]

\item $\LV$ is a distributive doubly algebraic lattice and hence forms a double Heyting algebra.

\item The following are equivalent for a subvariety $\CA$ of $\CV$\textup:

\begin{enumerate}[label={\upshape(\roman*)}]

\item $\CA$ is generated by a finite algebra and is join-irreducible in~$\LV$\textup;

\item $\CA$ is completely join-prime in~$\LV$\textup;

\item $\CA = \Var(\A)$, for some \textup(unique up to isomorphism\textup) finite subdirectly irreducible algebra~$\A$.

\end{enumerate}

\item Every finite subdirectly irreducible algebra in $\CV$ is a splitting algebra in~$\CV$.

\item $\LV$ is isomorphic to $\OSi$ where $\SifinV$ is a transversal of the isomorphism classes of finite subdirectly irreducible algebras in $\CV$ ordered by 
$\A \sqsubseteq \B$ if and only if $\A\in \mathbb{HS}(\B)$.
\end{enumerate}
\end{theorem}

\begin{proof}
(1) follows immediately from Theorem~\ref{thm:OP} and Lemma~\ref{lem:locfin}.

We now prove (2). Assume (i); so $\CA$ is finitely generated and join-irreducible in~$\LV$. Since $\CV$ is congruence distributive, $\LV$ is distributive and hence $\CA$ is join-prime in~$\LV$. By Lemma~\ref{lem:locfin}, $\CA$ is compact in $\LV$. Since every compact, join-prime element of a complete lattice is completely join-prime, (ii) follows. Now assume (ii). By Lemmas~\ref{lem:easy1} and~\ref{lem:Birk}, $\CA$ is generated by a finitely generated, and therefore finite, subdirectly irreducible algebra~$\A$. Hence (iii) holds. The uniqueness claim is an easy consequence of J\'onsson's Lemma~\cite[Corollary~3.4]{Jon}. Indeed, $\Var(\A) = \Var(\B)$, for finite subdirectly irreducible algebras $\A,\B\in \CV$, implies $\A\in \mathbb{HS}(\B)$ and $\B\in \mathbb{HS}(\A)$ and hence $\A\cong \B$. Finally, (iii) implies~(i) follows directly from Lemma~\ref{lem:SIjoindense}(1).

    (3) follows immediately from Lemma~\ref{lem:easy1} and the implication (iii)~$\Rightarrow$~(ii) in~(2).
    
    Finally, we prove (4). By Theorem~\ref{thm:OP}(6), we have $\LV\cong\cat O(\cat P)$, where $\cat P$ is the ordered set of subvarieties of $\CV$ that are completely join-prime in~$\LV$. The equivalence of (ii) and (iii) in (2) shows that $\cat P\cong \SifinV$ since (again by J\'onnson~\cite[Corollary~3.4]{Jon}), for finite subdirectly irreducible algebras $\A$ and $\B$, we have $\Var(\A)\subseteq \Var(\B)$ if and only if $\A\in \mathbb{HS}(\B)$. 
\end{proof}

Thus, for a locally finite, congruence-distributive variety $\CV$, its lattice of subvarieties $\LV$ is richly endowed with splittings. In particular, by Theorems~\ref{thm:Heyt-cov2splitting} and~\ref{thm:DoubHeyt-cov2splitting}, every cover in $\LV$ gives rise to a splitting. In the remainder of the paper we will see that the assumption of local finiteness is crucial here: the varieties we study are congruence distributive and generated by their finite members but not locally finite and have almost no splittings at all.

%%%%%%%%%%%%%%%%%%%%%%%%%%%%%%%%%%%%%%%%%%%%%%%%%%%%%%%%%%%%%%%%%%%%%%%
\section{Splitting algebras in varieties of logic: the Non-splitting Lemma}\label{sec:split-alg-var-log}
%%%%%%%%%%%%%%%%%%%%%%%%%%%%%%%%%%%%%%%%%%%%%%%%%%%%%%%%%%%%%%%%%%%%%%%

By a \emph{variety of logic} we mean any variety of algebras that forms
an algebraic semantics for some well-behaved logic. 
%tk made this simpler
% Since the notion of
%a well-behaved logic is itself not well-behaved (not even well-defined),
We will not enter into details, but intuitively, we wish to
include all logics that have a conjunction, an equivalence, a truth constant, and
a unary (term-defined, possibly trivial) connective resembling a modal operator.

Thus, in this section we will work with a fixed ambient variety $\CR$ of
algebras of finite signature $\tau$, such that there exist binary terms
$\wedge$, $\eqv$, a unary term $\delta$, and a constant term $1$, whose
interpretations in $\CR$ have the following properties: 
\begin{enumerate}  
\item[(P1)] $\wedge$ is a semilattice operation,
\item[(P2)] $x\eqv y\leq 1$, and $x\eqv y = 1$ if and only if $x = y$,
\item[(P3)] $\delta$ is order preserving and satisfies $\delta x\leq x$,
\item[(P4)] for each congruence $\vartheta$, the filter
$\up(1/\vartheta)$ is closed under $\delta$,
\item[(P5)] each filter closed under $\delta$ and containing $1$ is of the form
  $\up(1/\vartheta)$ for some congruence $\vartheta$.
\end{enumerate}
Note that (P2) implies that $\CR$ is congruence regular with respect to
$1$, as we have $x\equiv_\vartheta y$ if and only if
$x\eqv y\equiv_\vartheta 1$. 
Moreover, (P2) implies that every non-trivial congruence $\vartheta$ has $a\equiv_\vartheta 1$ for some $a<1$. Further, (P3) together with (P4) imply that $\delta 1 = 1$ in every algebra $\mathbf{A}\in\CR$.

The following lemma characterises finite subdirectly irreducible algebras in
$\CR$. Given a unary operation $f$ we denote its $n$-fold composite by~$f^n$.

\begin{lemma}
Let $\mathbf{A}\in\CR$ be finite and subdirectly irreducible. Let $\mu$
be its monolith. Then there exists $n\in \mathbb{N}$ such that for 
each $a\in 1/\mu$ with $a<1$ we have $\delta^{n+1}a = \delta^na$, and $\up(1/\mu) = \up(\delta^n a)$. 
\end{lemma}

\begin{proof}
Note that if $\delta a = a$ for some $a<1$, then
$\up(a\wedge 1)$ satisfies the conditions in~(P5), and so it is 
of the form $\up(1/\vartheta)$ for some congruence
$\vartheta\geq\mu$. Let $F =  
\up(1/\mu)$.  By finiteness, $F$ is principal, so $F =
\up b$. By (P2), (P3) and (P4) we get that $b<1$ and $\delta b = b$.
Moreover, $b$ is the only element in $F$ with these properties. For if
$b<a<1$ and $\delta a =a$ hold, then $\up a =
\up(1/\alpha)$ for some congruence
$\alpha$ with $0<\alpha<\mu$ contradicting the fact that
$\mu$ is the monolith of $\mathbf{A}$. Thus, for every $a$ with $b < a <1$
we have $\delta a < a$, so by finiteness $\delta^ka = b = \delta^{k+1}a$ for some $k$.
Also by finiteness, we can take $n\in \mathbb{N}$ large enough to satisfy
$\delta^na = b = \delta^{n+1}a$ for every $a$ with $b<a<1$.
\end{proof}

Let $\CR_n$ be the subvariety of $\CR$ defined by
$\delta^{n+1}x = \delta^nx$. Every finite subdirectly irreducible algebra $\mathbf{A}\in\CR$ belongs to $\CR_n$ for some $n$. We will use $\mu_\bot$ to denote the smallest element of the filter $1/\mu$, where $\mu$ is the monolith of $\mathbf{A}$.

\subsection{Algebras describing themselves}
Let $\mathbf{A}\in\CR$ be finite and subdirectly irreducible.  
Fix a set $X$ of variables with $|X|=|A|$, and index them by the elements of
$A$. Define the \emph{term-diagram} of ${\bf A}$ to be the $|A|$-ary term
\[
\Delta_{\bf A} =
\bigwedge\{x_{f(a_1,\dots,a_n)}\eqv f(x_{a_1},\dots,x_{a_n})
\mid a_1,\dots,a_n\in A, f\in \tau\}. 
\]
Throughout the paper, we will use $\mathbf{A}\leq\mathbf{B}$ to indicate that $\A$ embeds into~$\B$.

\begin{lemma}\label{valu}
Let $\mathbf{A}$ and $\mathbf{B}$ be algebras from $\CR$, with
$\mathbf{A}$ subdirectly irreducible and $|A| = k$. Then 
$\mathbf{A}\leq\mathbf{B}$ if and only if there exists a
$k$-tuple $\ov{b}$ of elements of $B$ such that
$\mathbf{B}\models (\Delta_{\mathbf{A}}\approx 1)[\ov{b}]$ and
$\mathbf{B}\models (x_{\mu_\bot} \not\approx 1)[\ov{b}]$.  
\end{lemma}

\begin{proof}
For the forward direction,  
define $\ov{b}$ putting $b_a = a$ for every $a\in A$.
Clearly $x_{\mu_\bot}[\ov{b}] = b_{\mu_\bot}\neq 1$, so
$\mathbf{B}\models (x_{\mu_\bot} \not\approx 1)[\ov{b}]$. Moreover,
for each $f\in\tau$ and each $n$-tuple $( a_1,\dots,a_n)\in A^n$,
we have
\begin{align*}
(x_{f(a_1,\dots,a_n)}\eqv f(x_{a_1},\dots,x_{a_n}))[\ov{b}] &=
b_{f(a_1,\dots,a_n)}\eqv f(b_{a_1},\dots,b_{a_n})\\
     &= f(a_1,\dots,a_n)\eqv f(a_1,\dots,a_n) 
     = 1.
\end{align*}  
Hence, $\mathbf{B}\models (\Delta_{\mathbf{A}}\approx 1)[\ov{b}]$. 

For the converse, let $\ov{b}$ be a $k$-tuple with the required
properties. Define a map $h\colon A\to B$ by $h(a) = x_a[\ov{b}]$. This map is a homomorphism by the definition of $\Delta_\mathbf{A}$ and the properties of $\eqv$. Moreover, we have $h(1) = x_1[\ov{b}] = 1^{\bf B} \neq x_{\mu_\bot}[\ov{b}] = h({\mu_\bot})$. Therefore, as $\mathbf{A}$ is subdirectly irreducible and every non-trivial congruence on $\mathbf{A}$ contains $(1,{\mu_\bot})$, the kernel of $h$ must be the trivial congruence on $\mathbf{A}$. Hence, $h$ is an embedding.
\end{proof}  

\begin{NonSplit}\label{no-splitting}
Let $\mathbf{A}\in\CR$ be finite and subdirectly irreducible.
The following are equivalent\textup:
\begin{enumerate}[label={\upshape(\arabic*)},leftmargin=1.75\parindent]
\item $\mathbf{A}$ is not a splitting algebra in $\CR$\textup;
\item $\forall i\in\mathbb{N}\; \exists{\bf B}\in\CR\colon
\mathbf{A}\not\in \Var(\mathbf{B})$ and $\mathbf{B}\not\models
  \delta^i(\Delta_{\mathbf{A}})\leq x_{\mu_\bot}$\textup;
\item $\forall i\in\mathbb{N}\; \exists k\geq i\; \exists{\bf B}\in\CR\colon
\mathbf{A}\not\in \Var(\mathbf{B})$ and $\mathbf{B}\not\models
  \delta^k(\Delta_{\mathbf{A}})\leq x_{\mu_\bot}$. 
\end{enumerate}
\end{NonSplit}

\begin{proof}
To prove the implication from (2) to (1) take algebras 
$\mathbf{B}_i$, for each $i\in \mathbb{N}$, such that
$\mathbf{B}_i\not\models \delta^i(\Delta_{\mathbf{A}})\leq x_{\mu_\bot}$ and 
$\mathbf{A}\not\in \Var(\mathbf{B}_i)$. Let $k = |A|$.
Choose a $k$-tuple $\ov{b(i)} = \big( b(i)_1,\dots,b(i)_{k-1},s(i)\big)$ of
elements of $\mathbf{B}_i$ 
such that  
\[
\delta^i(\Delta_{\mathbf{A}})[b(i)_1,\dots,b(i)_{k-1}]\not\leq s(i) \ \text{in $\mathbf{B}_i$}. 
\]
Let $\mathbf{B} = \prod_{i\in \mathbb{N}}\mathbf{B}_i$, and consider the $k$-tuple 
\[
\ov{b} = 
\bigl(( b(i)_1\mid i\in\mathbb{N}),\dots,(
b(i)_{k-1}\mid i\in\mathbb{N}),
( s(i)\mid i\in\mathbb{N})\bigr) \in B^k.
\]
Let $s$ stand for $\big( s(i)\mid i\in\mathbb{N}\big)$.
By the choice of $\ov{b}$ we have that 
$\forall i\in\mathbb{N}\colon 
\delta^i(\Delta_\mathbf{A})[\ov{b}]\not\leq s$ in $\mathbf{B}$.
It follows that the filter $F =
\up\bigl\{\delta^i(\Delta_\mathbf{A})[\ov{b}]\mid
i\in\mathbb{N}\bigr\}$ does not
contain $s$. 

Therefore, taking the congruence 
$\theta$ corresponding to $F$, we obtain
that 
\[
\mathbf{B}/\theta\models
\Delta_\mathbf{A}[\ov{b}/\theta] = 1 \ \text{ and } \ s/\theta \neq 1\ \text{in the quotient $\mathbf{B}/\theta$}.
\]
By Lemma~\ref{valu}, we get 
that $\mathbf{A} \leq \mathbf{B}/\theta$. Hence,
$\mathbf{A} \in \Var(\mathbf{B})$. Now, to derive a contradiction, assume
$\mathbf{A}$ is a splitting algebra. Then there exists the largest subvariety
$\CV$ of $\CR$ such that 
$\mathbf{A}\notin\CV$. Since $\mathbf{A}\notin \Var(\mathbf{B}_i)$ for all
$i\in\mathbb{N}$, we have that $\Var(\mathbf{B}_i)\subseteq \CV$ for every
$i\in\mathbb{N}$. But then $\mathbf{B} = \prod_{i\in \mathbb{N}}\mathbf{B}_i$
belongs to $\CV$, and therefore 
$\mathbf{A}\in \CV$, which contradicts the assumption that $\mathbf{A}$
is splitting.

To show that (1) implies (2) we will prove the contrapositive. Assume that 
$\exists {i\in\mathbb{N}}\;\forall\mathbf{B}\in \CR\colon
\mathbf{A}\not\in \Var(\mathbf{B})$ implies
$\mathbf{B}\models \delta^i(\Delta_{\mathbf{A}})\leq x_{\mu_\bot}$.
Let $m$ be the smallest with this property. 
We will now show that $\mathbf{A}$ is a splitting algebra. Namely, we claim
that the subvariety $\CW$ of $\CR$ defined by 
the identity $\delta^m(\Delta_{\mathbf{A}})\wedge x_{\mu_\bot} \approx x_{\mu_\bot}$ is the
largest subvariety of $\CR$ to which $\mathbf{A}$ does not belong. 
Obviously, $\mathbf{A}\notin\CW$, as otherwise we would have
$\mathbf{A}\models \delta^m(\Delta_{\mathbf{A}})\wedge x_{\mu_\bot} \approx x_{\mu_\bot}$ which
cannot be the case by Lemma~\ref{valu}.
Take any subvariety $\CV$ of $\CR$, with 
$\mathbf{A}\not\in \CV$. Let $\mathbf{F}$ be the free countably generated algebra
in $\CV$ so that $\CV = \Var(\mathbf{F})$.  This, by our assumption, 
implies that $\mathbf{F}\models \delta^m(\Delta_{\mathbf{A}})\wedge x_{\mu_\bot} \approx
x_{\mu_\bot}$. Hence,  
$\CV\models \delta^m(\Delta_{\mathbf{A}})\wedge x_{\mu_\bot} \approx x_{\mu_\bot}$ and
therefore $\CW\supseteq \CV$ as claimed. 

Finally, (2) is equivalent to (3) since, by (P3), the map $\delta$ is decreasing.
\end{proof}

The remainder of the paper is devoted to using the Non-splitting Lemma to prove that several familiar varieties of logic contain no splitting algebras except for some very small algebras. Note however, that the Non-splitting Lemma
can only be used to prove that certain \emph{finite}
algebras are not splitting. Fortunately, as we already mentioned in Section~\ref{sec:intro}, McKenzie~\cite{McK} proved that if
$\CV$ is congruence distributive and generated by its finite members
then every splitting algebra is finite, so 
the Non-splitting Lemma
suffices.  

Kracht~\cite{Kra92} shows that 
the only algebra splitting the variety of tense algebras is (term equivalent to)
the two-element Boolean algebra; the same is proved in Kowalski and
Ono~\cite{KO00} for 
for the variety of FL\textsubscript{ew}-algebras\footnote{Called
  \emph{residuated lattices} there, at variance with present terminology. See
  the next section.}. 
Kowalski and Miyazaki~\cite{KM09} prove that there
are only two splitting algebras in the variety of KTB-algebras. 
Each of these applications of the Non-splitting Lemma
involved a pair of constructions:
an expansion followed by a distortion. We will demonstrate the process in two
further cases. The reader will see that the constructions 
need to be precisely tailored to each particular case. This was also true 
for all previously known examples, so it does not seem likely that a generic 
construction can be found.

%%%%%%%%%%%%%%%%%%%%%%%%%%%%%%%%%%%%%%%%%%%%%%%%%%%%%%%%%%%%%%%%%%%%%%%
\section{Residuated lattices}\label{sec:res-latt}
%%%%%%%%%%%%%%%%%%%%%%%%%%%%%%%%%%%%%%%%%%%%%%%%%%%%%%%%%%%%%%%%%%%%%%%

A \emph{residuated lattice}  is an algebra 
$\mathbf{A} = \langle A;\wedge,\vee,\ld,\rd,\cdot,1\rangle$
such that $\langle A;\wedge,\vee\rangle$ is a lattice, and
$\langle A;\cdot,\ld,\rd, 1\rangle$ is a residuated monoid, that is, an ordered
monoid satisfying
\[
y\leq x\ld z\iff xy\leq z \iff x\leq z\rd y.
\]
The operations $\ld$ and $\rd$ are called, respectively,
\emph{left division} (or \emph{right residuation}) and
\emph{right division} (or \emph{left residuation}).
Multiplication binds stronger than divisions, which 
bind stronger than the lattice operations.
The following identities will be important later.
\begin{enumerate}[label={\upshape(\arabic*)},leftmargin=1.75\parindent]
\item $1\geq x$,
\item $xy = yx$,
\item $x^{n+1} = x^n$,
\end{enumerate}
A residuated lattice satisfying (1), (2), or (3) is called
\emph{integral}, \emph{commutative}, or \emph{$n$-potent}, respectively.
We write $\mathsf{RL}$ for the variety of all residuated lattices, and
$\mathsf{IRL}$, $\mathsf{CRL}$, $\mathsf{CIRL}$, respectively, for
the varieties of integral, commutative, and commutative integral residuated
lattices. In the commutative case, the left and right residuals become
opposites, for we have $x\ld y = y\rd x$. It is then customary to blur the
distinction between then and write write $x\ra y$ for both.
For a residuated lattice $\m{A}$, an element $a\in A$ is called \emph{negative}
if $a\leq 1$, and \emph{strictly negative} if $a<1$; \emph{\textup(strictly\textup) positive}
elements are defined dually. If a residuated lattice $\m{A}$ has a unique largest strictly
negative element, then $\m{A}$ is subdirectly irreducible. For commutative
residuated lattices the converse is also true. 

For more details on residuated lattices, and for any unexplained
nomenclature, we refer the reader to Galatos \emph{et al.}~\cite{GJKO07}.
Residuated lattices expanded by a constant $0$, are known as FL-algebras
(especially among logicians, because of the connection with \emph{Full Lambek
  calculus}). In older literature, the name `residuated lattices' was used for
what is now called FL\textsubscript{ew}-algebras: a subvariety of FL-algebras consisting
of commutative, integral FL-algebras satisfying $0\leq x$. This was the
terminology used in Kowalski and Ono~\cite{KO00}, for example.

The variety $\mathsf{CRL}$ of commutative residuated lattices
satisfies (P1)--(P5), with $x\eqv y =  (x\ra y)\wedge (y\ra x)\wedge 1$ and 
$\delta x = x^2$, so the Non-splitting Lemma~\ref{no-splitting} applies in principle. To apply it in practice, we need two constructions given below. Each will be given in a rather general form, with a view to possible applications in a wider class of residuated lattices. However, the generality will be somewhat
evasive, as the constructions seem to generalise in incompatible ways.

\subsection{Expansions of commutative residuated lattices}
Let $\mathbf{A}$ be a commutative residuated lattice, and
let $c\in A$ be an arbitrary strictly negative element. Note that we have
$ca \leq a$ for all $a\in A$. Let $A_0 = \{a\in A\mid ca<a\}$
and let $D$ be a copy of $A_0$ disjoint from $A$,
so that $D = \{d_a\mid a\in A_0\}$. Let $P = A\cup D$. We will define a binary
relation (denoted $\leq$) and a binary operation (denoted $\cdot$) on $P$.
For $x, y\in P$, we put $x\leq y$ if any of the following holds: 
\begin{align*}
x,y\in A &\text{ and } x\leq^{\mathbf{A}} y,\\
x = d_a\in D, y\in A &\text{ and } a \leq^{\mathbf{A}} y,\\
x\in A, y = d_a\in D &\text{ and } x \leq^\m{A} ca,\\
x=d_a, y = d_b\in D &\text{ and } a \leq^\m{A} b.
\end{align*}
Intuitively, we insert a new element between each pair $ca<a$ in such a way
that $ca < d_a < a$ holds. In particular, $c < d_1 <1$.

It is not difficult to show that the relation $\leq$ defined above is an 
order on~$P$. Next, for all $x,y\in P$, we put:
\begin{equation*}\label{multiplication} 
x\cdot y = y\cdot x =\begin{cases}
xy     &\text{ if } x,y \in A,\cr
d_{ay} &\text{ if } x=d_a\in D,\ y\in A,\ cay<ay,\cr
ay     &\text{ if } x=d_a\in D,\ y\in A,\ cay=ay,\cr
cab    &\text{ if } x=d_a\in D,\ y=d_b\in D.
\end{cases}
\end{equation*}

\begin{lemma}\label{pomonoid}
The structure\/ $\m{P} = 
\langle P;\leq, \cdot, 1\rangle$ is 
an ordered commutative monoid. Moreover, if $\m{A}$ is integral, then
$x\leq 1$ holds for all
$x\in P$.
\end{lemma}

\begin{proof}
The main part of the proof is a tedious case-checking exercise, most of which we
omit, especially that it is nearly identical to the proof of Fact~4 in Kowalski
and Ono~\cite{KO00}.  Here is one case 
as an example. Let $x = d_a$, $y\in A$, and let $z = d_b$. Assume moreover that
$cay<ay$ and $cyb = yb$. Then we have $(d_a\cdot y)\cdot d_b = d_{ay}\cdot d_b
= cayb$, but observe that $cayb = ayb$ since  $cyb = yb$.
Next, associating the other way we obtain $d_a\cdot (y\cdot d_b) = d_a\cdot yb =
ayb$, as $cayb = ayb$.

The moreover part follows immediately from the construction of $\m{P}$.
\end{proof}

The next lemma shows that every element of $P$ either
belongs to $A$ or is of the form $d_1\cdot a$, for some $a\in A$. 
We will write $d$ instead of $d_1$ from now on.

\begin{lemma}\label{multiplication-properties}
For all $x,y\in A$, the following hold\textup: 
\begin{enumerate}[label={\upshape(\arabic*)},leftmargin=1.75\parindent]
\item if $cx<x$, then $d\cdot x = d_x$, otherwise $d\cdot x = x$,
\item $y\leq d\cdot x$ if and only if $y\leq cx$,
\item $d\cdot x\cdot d\cdot y = cxy$,
\item $d\cdot x\leq d\cdot y$ if and only if $d\cdot x\leq y$
  if and only if $x\leq y$.
\end{enumerate}
\end{lemma}

\begin{proof}
All claims are easily derived from the definition of multiplication in $\m{P}$.
\end{proof}

Although $\mathbf{P}$ is in general neither a residuated monoid, nor a lattice,
it will be convenient to view it as a partial algebra in the signature of
residuated lattices, with meet, join, and the residual only partially defined.
This makes the statement of the next lemma clear.

\begin{lemma}
The residuated lattice $\mathbf{A}$ is a subalgebra of\/ $\mathbf{P}$.
\end{lemma}

\begin{proof}
The proofs of preservation of meet, join and multiplication from $\m{A}$ are
straightforward, so we will only show that $a\ra^{\mathbf{A}} b$ satisfies
\[
\forall x\in P\colon\; a\cdot x \leq b \Longleftrightarrow x\leq a\ra^{\mathbf{A}} b
\]
in $\mathbf{P}$, for all $a,b\in A$.  This equivalence clearly holds for 
all $x\in A$,
by residuation in~$A$. Let $x = d_s$ for some $s\in A$. Then $x = d\cdot s$,
so we have $a\cdot x = a\cdot d\cdot s = d\cdot a\cdot s$, and thus
\begin{align*}
a\cdot x =   a\cdot d\cdot s \leq b &\Longleftrightarrow d\cdot as \leq b
  & & (\text{since $a\cdot s = as$})\\
                         &\Longleftrightarrow as\leq b
  & & (\text{by Lemma~\ref{multiplication-properties}(4)})\\
                         &\Longleftrightarrow s\leq a\ra^{\mathbf{A}} b
  & & (\text{by residuation in $\m{A}$})\\
               &\Longleftrightarrow d\cdot s\leq a\ra^{\mathbf{A}} b
  & & (\text{by Lemma~\ref{multiplication-properties}(4)}) .\qedhere 
  \end{align*}
\end{proof}

Next, we will expand $\m{P}$ to a residuated lattice. To this end, we will
use a version of \emph{residuated frames}, defined and put to good use
in Galatos and Jipsen~\cite{GJ13}. 

Let $\mathbf{M} = 
\langle M;\leq,\cdot,1\rangle$ be a commutative ordered monoid,
and $W$ a set. A binary relation $N\subseteq M\times W$ is called
a \emph{nuclear relation on} $\mathbf{M}$ if, for every $x\in M$ and $w\in W$ there exists a subset $x\Rightarrow w$ of $W$ such that for every
$y\in M$ the following equivalence holds
\[
x\cdot y\ N\ w \quad\text{ if and only if }\quad y\ N\ x\Rightarrow w
\]
where $y\ N\ x\Rightarrow w$ abbreviates 
$y\ N\ u$ for all $u\in x\Ra w$.
The importance of being nuclear resides in the fact that every nuclear
relation $N$ on $\mathbf{M}$ gives rise to a residuated lattice which
preserves all partial residuated lattice structure that exists in $\mathbf{M}$.

\begin{lemma}[\cite{GJ13}]\label{nuclear}
Let\/ $\mathbf{M} = 
\langle M;\leq,\cdot,1\rangle$
be an ordered monoid. Let $W$ be any set, and let
$N\subseteq M\times W$ be nuclear. Further, let
$\gamma_N$ be the closure operator on $M$ associated with the polarities of $N$.
Then the complete lattice $L[M]$ 
of closed subsets of $M$ carries a residuated lattice structure
$\m{L}[\m{M}] = 
\langle L[M];\wedge, \vee, \cdot, \ra, 1\rangle$, such that 
\begin{enumerate}[label={\upshape(\arabic*)},leftmargin=1.75\parindent]
\item The operations in $\m{L}[\m{M}]$ are given by:
\begin{itemize}
\item $X\wedge Y = X\cap Y$,
\item $X\vee Y = \gamma_N(X\cup Y)$,  
\item $X\cdot Y = \gamma_N\{x\cdot^{\m{M}} y\mid x\in X,\ y\in Y\}$,
\item $X\ra Y = \{z\in M\mid \forall x\in X\colon z\cdot x \in Y\}$,  
\item $1 = \gamma_N(1^\m{M})$,
\end{itemize}
for all closed $X,Y\subseteq M$.
\item $\mathbf{M}$ embeds into $\mathbf{L}[\mathbf{M}]$ as an ordered monoid.
\item If\/ $\m{M}$ is integral, so is $\mathbf{L}[\mathbf{M}]$.
\item If\/ $\m{M}$ is commutative, so is $\mathbf{L}[\mathbf{M}]$.  
\item If\/ $\m{M}$ is finite, so is $\mathbf{L}[\mathbf{M}]$.
\item The embedding preserves all existing meets, joins and residuals from
  $\mathbf{M}$.
\end{enumerate}
\end{lemma}

We will now exhibit a suitable nuclear relation on $\m{P}$.
For each $x\in P$ define $\lambda_x\colon P\to P$ by
$\lambda_x(y) = x\cdot y$. Let $\Lambda = \{\lambda_x\mid x\in P\}$
and $W = \Lambda\times A$. Define a binary relation $N\subseteq P\times W$ putting
$x\ N\ (\lambda, a)$ if $\lambda(x)\leq a$. Next, for $x\in P$, $\lambda\in
\Lambda$ and $a\in A$ define $x \Ra (\lambda, a)$ to be the singleton
$\{(\lambda_x\circ\lambda, a)\}$. Then we have
\begin{align*}
x\cdot y\ N\ (\lambda, a) &\Longleftrightarrow \lambda(x\cdot y)\leq a\\
      &\Longleftrightarrow \lambda\circ\lambda_x(y)\leq a\\
      &\Longleftrightarrow y\ N\ (\lambda\circ\lambda_x, a)\\
      &\Longleftrightarrow y\ N\ x\Ra (\lambda, a)  
\end{align*}                            
and so $N$ is nuclear, as claimed. The next result is an immediate corollary.

\begin{lemma}\label{expn}
Let\/ $\m{A}$ and $\m{P}  = 
\langle P;\leq, \cdot, 1\rangle$
be as in
Lemma~\ref{pomonoid}, and let $N$ be the nuclear relation defined above.
Then $\m{L}[\m{P}]$ is a residuated lattice such that
${\m{A}\leq \m{P}\leq\m{L}[\m{P}]}$.
\end{lemma}

From now until Lemma~\ref{expansion}, we will keep $\m{A}$, $\m{P}$ and
$\m{L}[\m{P}]$ fixed. To proceed, we need to
describe the elements of $\m{L}[\m{P}]$ (the closed sets of
$\m{P}$) more concretely. To lighten the notation, we put
$\widehat{X} = \bigvee\{x\in A\mid x\in X\}$ and
$\widetilde{X} = \bigvee\{x\in A\mid d\cdot x\in X\}$.

\goodbreak

\begin{lemma}\label{expn-props-1}
Let $X$ be a subset of $P$ satisfying the following conditions\textup: 
\begin{enumerate}[label={\upshape(\arabic*)},leftmargin=1.75\parindent]
\item $X$ is a non-empty down-set,
\item $\forall x, y\in A\colon  x\in X$ and $y\in X$ imply $x\vee y\in X$,
\item $\forall x, y\in A\colon  d\cdot x\in X$ and $d\cdot y\in X$ imply
  $d\cdot(x\vee y)\in X$.
\end{enumerate}
Then $X = \dw \widehat{X}\cup \dw(d\cdot\widetilde{X})$.
\end{lemma}

\begin{proof}
Assume $X\subseteq P$ satisfies (1)--(3). By the finiteness of $P$, 
we have $\widehat{X} \in X$ and $d\cdot\widetilde{X}\in X$, so
$\dw \widehat{X}\cup \dw(d\cdot\widetilde{X})\subseteq X$, since $X$ is a down-set.
Let $x\in X$. If $x\in A$, then $x\leq \widehat{X}$. If $x = d_a\in D$, then
$a\leq \widetilde{X}$, so $x = d\cdot a\leq d\cdot\widetilde{X}$. Thus,
$X = \dw \widehat{X}\cup \dw(d\cdot\widetilde{X})$.
\end{proof}

\begin{lemma}\label{expn-props-2}
Every closed $X\subseteq P$ is of the form:
\[
X = \dw \widehat{X}\cup \dw(d\cdot\widetilde{X}).
\]
Moreover, for every $a\in A$, the sets $\dw a$ and $\dw d\cdot a$ are both closed.
\end{lemma}

\begin{proof}
Assume $X$ is closed. We will show that $X$ satisfies (1)--(3)
of Lemma~\ref{expn-props-1}.
Note that these conditions are preserved by intersections; for (2) and (3) it is
immediate, for (1) it follows from the fact that $\Pb$ has the smallest
element. Thus, it suffices to prove that (1)--(3) hold for basic
closed sets, that is, sets of the form $\{x\in P\mid x \ N\ (\lambda_u, s)\}$ for some $u\in P$, $s\in A$. We will use the standard 
notation for basic closed sets, writing $X^\triangleleft$ for
$\{z\mid \forall x\in X\colon z\ N\ x\}$, and simplifying
$\{x\}^\triangleleft$ to~$x^\triangleleft$.

Let $\bot$ be the smallest element of $\Pb$.
We claim that $\bot\in (\lambda, s)^\triangleleft$. Indeed, 
we have $\bot\in (\lambda_u, s)^\triangleleft$ if and only if
$u\cdot\bot\leq s$, and this holds for all $u\in P$ and $s\in A$. 
Next, let $x\in (\lambda_u, s)^\triangleleft$
and $y\leq x$. Then we have $u\cdot x\leq s$, so by the 
monotonicity of multiplication in $\m{P}$, we get 
$u\cdot y\leq s$, and thus $y\ N\ (\lambda_u, s)$.  This proves (1).

Now let $x, y\in (\lambda_u, s)^\triangleleft$. Then 
$u\cdot x\leq s$ and  $u\cdot y\leq s$. If $u\in A$, we get
$ux\vee uy = u(x\vee y)\leq s = u\cdot (x\vee y)$, so
$x\vee y\in (\lambda_u, s)^\triangleleft$. If $u = d_a\in D$, then
from $u\cdot x\leq s$ and  $u\cdot y\leq s$ we get
$ax\leq s$ and $ay\leq s$, so reasoning as before we get $a\cdot(x\vee y)\leq
s$, and hence $x\vee y\in (\lambda_u, s)^\triangleleft$ again. This proves (2). 

Next, let $d\cdot x, d\cdot y\in (\lambda_u, s)^\triangleleft$. Then 
$u\cdot d\cdot x\leq s$ and $u\cdot d\cdot y\leq s$. If  $u\in A$, these imply
$d\cdot ux\leq s$ and $d\cdot uy\leq s$, and further
$ux\leq s$ and $uy\leq s$. These hold if and only if $u(x\vee y)\leq s$, from
which it follows that $d\cdot u\cdot(x\vee y)\leq s$. Since
$d\cdot u\cdot (x\vee y) = u\cdot d\cdot(x\vee y)$, we have that
$d\cdot(x\vee y)\in (\lambda_u, s)^\triangleleft$.
If $u = d_a\in D$, we obtain 
$d\cdot d\cdot a\cdot x\leq s$ and $d\cdot d\cdot a\cdot y\leq s$; therefore,
$cax\leq s$ and  $cay\leq s$. This holds if and only if 
$ca(x\vee y)\leq s$, which implies $d_a\cdot d\cdot(x\vee y)\leq s$, which
in turn implies $d\cdot (x\vee y)\in (\lambda_u, s)^\triangleleft$.
This proves (3). 

It remains to show that $\dw a$ and $\dw(d\cdot a)$ are closed.
For $\dw a$ consider $Z = \bigcap\{(\lambda_u,s)^\triangleleft\mid u\cdot a\leq s\}$.
For every $z\leq a$, we have $z\in Z$, by monotonicity of
multiplication. For the converse, note that the basic set
$(\lambda_1,a)^\triangleleft$ is a member of
$\{(\lambda_u,s)^\triangleleft\mid u\cdot a\leq s\}$, so
$z\in Z$ implies $z\in (\lambda_1,a)^\triangleleft$, that is
$1\cdot z\leq a$. Thus, $\dw a = Z$.

For $\dw(d\cdot a)$ consider
$Z' = \bigcap\{(\lambda_u,s)^\triangleleft\mid u\cdot d\cdot a\leq s\}$.
For every $z\leq d\cdot a$ we have $z\in Z'$, as before. For the converse,
note that $(\lambda_1,a)^\triangleleft$ and $(\lambda_d,ca)^\triangleleft$ are members of
$\{(\lambda_u,s)^\triangleleft\mid u\cdot d\cdot a\leq s\}$. If
$z\in Z\cap A$, then since $z\in (\lambda_d,ca)^\triangleleft$, we have
$d\cdot z\leq ca$, and so $z\leq ca$; therefore $z\leq d\cdot a$.
If $z\in Z\cap D$, say $z = d_b$ for some $b\in A$, then 
since $z\in (\lambda_1,a)^\triangleleft$, we have
$d\cdot b\leq a$, and so $b\leq a$; hence $d\cdot b\leq d\cdot a$.
\end{proof}

Let $\m{L}$ be a subdirectly irreducible commutative residuated lattice,
let $\mu$ be the monolith of $\m{L}$, and let $F_\mu = 
\up(1/\mu)$.
We define the \emph{depth} of $\mu$ to be the least $n\in \mathbb{N}$ such that,
for all $a\in 1/\mu$ with $a<1$, we have $a^{n+1} = a^n$. If no such $n$ exists,
the the depth of $\mu$ is undefined.

\begin{lemma}\label{expn-more-props}
Assume $\m{A}$ is subdirectly irreducible with monolith $\mu$ of depth $n$,
and $c\prec 1$ is the unique largest strictly negative element of $A$.
The following hold\textup: 
\begin{enumerate}[label={\upshape(\arabic*)},leftmargin=1.75\parindent]  
\item $\m{L}[\m{P}]$ is subdirectly irreducible. 
\item The monolith $\nu$ of\/ $\m{L}[\m{P}]$ has depth 
at least $2n$.
\item $\mu = \nu\restrictedto{\m{A}}$.  
\item $\m{L}[\m{P}]/\nu$ is isomorphic to $\m{A}/\mu$. 
\end{enumerate}  
\end{lemma}  

\begin{proof}
For (1), note 
that $\dw d$ is a closed set whose unique cover in the inclusion ordering
is $\dw 1 = P$. Claim (2) follows by construction. To see it, note that 
$d^{2k} = c^k$, for all~$k$, so if $k<n$ we have 
$d^{2k+1} = d\cdot d^{2k} = d\cdot c^k < c^k$, since $c^{k+1}<c^k$.
For $k = n$, we have $d^{2n+1} = d\cdot d^{2n} = d\cdot c^n =
c^n$, since $cc^n = c^{n+1} = c^n$. As multiplication is preserved in
$\m{L}[\m{P}]$, we have $(\dw d)^{2n-1}> (\dw d)^{2n} = (\dw d)^{2n+1}$,
whence the monolith of $\m{L}[\m{P}]$ has depth at least $2n$.
Next, (3) follows from (2) using the fact that $\m{A}\leq \m{L}[\m{P}]$.
To prove (4), we begin by showing that  
$X\equiv_\nu \dw\widehat{X}$ holds for every closed
$X\subseteq P$. For each
closed $X\subseteq P$, by Lemma~\ref{expn-props-2}, we have
$X = \dw\widehat{X}\cup \dw(d\cdot\widetilde{X})$,
so $\dw \widehat{X}\subseteq X$, and thus $\dw \widehat{X}\ra X = P = 1^{L[P]}$.
Now, since $\nu = \m{Cg}^{\m{L}[\m{P}]}(\dw d, P)$,
it suffices to show that $\dw d\subseteq X\ra \dw \widehat{X}$. This is further
equivalent to $d\in X\ra \dw \widehat{X}$. Let $x\in X$ and consider
$d\cdot x$. If $x\leq \widehat{X}$, we have $d\cdot x\leq x \leq \widehat{X}$, so 
$d\cdot x\in\dw \widehat{X}$. Assume $x\not\leq \widehat{X}$, so that
$x\leq d\cdot \bigvee\{y\in A\mid d\cdot y\in X\}$. In particular,
$x\notin A$, so $x = d\cdot b$ for some $b\in A$. Then we get
\[
d\cdot b \leq d\cdot \bigvee\{y\in A\mid d\cdot y\in X\}
\]
and thus 
\[
d\cdot x = c\cdot b \leq
c\cdot \bigvee\{y\in A\mid d\cdot y\in X\}
= \bigvee\{cy\in A\mid d\cdot y\in X\}.
\]
But $\bigvee\{cy\in A\mid d\cdot y\in X\}\in X$ because $X$ is closed, so
$d\cdot x\in \dw\widehat{X}$ as required.
We have shown that $X\equiv_\nu \dw\bigvee\{x\in A\mid x\in X\}$ holds for
every closed $X$. It follows that we have $\dw \widehat{X} \in X/\nu$, from
which it follows in turn that
every congruence class of $\nu$ contains (an image of) an element of $A$. 
This proves (4).
\end{proof}  

By iterating the construction, we immediately obtain the next lemma.

\begin{lemma}\label{expansion}
Let $\mathbf{A}$ be a subdirectly irreducible commutative integral
residuated lattice with monolith of depth $n$. For each
natural number $k$ there
exists a subdirectly irreducible commutative integral
residuated lattice $\m{E}$ with monolith $\vartheta$, such that\textup:
\begin{enumerate}[label={\upshape(\arabic*)},leftmargin=1.75\parindent]  
\item $\m{A}\leq \m{E}$,
\item $\vartheta$ has depth at least $k$,
\item $\mu = \vartheta\restrictedto{\m{A}}$,  
\item $\m{E}/\vartheta$ is isomorphic to $\m{A}/\mu$. 
\end{enumerate}
Moreover, if $\mathbf{A}$ is finite, so is $\mathbf{E}$.
\end{lemma}

The residuated lattice $\m{E}$ obtained above will be called an
\emph{expansion of $\m{A}$ of depth $m$}, where $m$ is the depth of the monolith
of $\m{E}$, or simply an \emph{expansion} of $\m{A}$.

\subsection{Truncated products of residuated lattices}
Let $\m{A}$, $\m{B}$ be residuated lattices, and let $c$ and $q$ be strictly
negative elements of $\m{A}$ and $\m{B}$, respectively. The \emph{truncated product} $\m{A}\odot\m{B}$ of $\m{A}$ and $\m{B}$ is the algebra with the universe 
\[
A\odot B = \{a\in A\mid a\leq c\}\times \{b\in B\mid b\leq q\}
\cup\{\lan 1,1\ran\}
\]
and operations defined  below. 
{\allowdisplaybreaks
\begin{align*}
1 &= \lan 1,1\ran,\cr
\lan a,i\ran\wedge \lan b,j\ran &= \lan a\wedge b, i\wedge j\ran,\\
\lan a,i\ran\vee \lan b,j\ran   &= \lan a\vee b, i\vee j\ran,\\
\lan a,i\ran\cdot\lan b,j\ran   &= \lan a\cdot b, i\cdot j\ran,\\
\lan a,i\ran\lres \lan b,j\ran    &= 
\begin{cases}
\lan a\lres b\wedge c, q\ran & \text{ if } a\not\leq b,\ i\leq j,\\
\lan c, i\lres j\wedge q\ran & \text{ if } a\leq b,\ i\not\leq j,\\
\lan a\lres b\wedge c, i\lres j\wedge q\ran & \text{ if } a\not\leq b,\ i\not\leq j,\\
\lan 1,1\ran & \text{ otherwise.}
\end{cases}\\
\lan a,i\ran\rres \lan b,j\ran    &= 
\begin{cases}
\lan a\rres b\wedge c, q\ran & \text{ if } a\not\leq b,\ i\leq j,\\
\lan c, i\rres j\wedge q\ran & \text{ if } a\leq b,\ i\not\leq j,\\
\lan a\rres b\wedge c, i\rres j\wedge q\ran & \text{ if } a\not\leq b,\ i\not\leq j,\\
\lan 1,1\ran & \text{ otherwise.}
\end{cases}                              
\end{align*}}%
Until the end of this subsection we will keep 
$\m{A}$, $\m{B}$, $c$ and $q$ fixed.

\begin{lemma}\label{si-rl}
$\m{A}\odot\m{B}$ is a subdirectly irreducible integral residuated
lattice.  
\end{lemma}

\begin{proof}
Note that $A\odot B$, viewed
as a subset of $A\times B$ is closed under meet, join, and multiplication. 
Thus, $\m{A}\odot\m{B}$ is a lattice ordered commutative monoid; in fact,
it is a lattice ordered submonoid of the the direct product $\mathbf{A}\times\mathbf{B}$.
Next, since~$\lres$ and~$\rres$ are defined symmetrically, it suffices to
verify residuation equivalences for one of these. Moreover, the third case in the
definition of $\lres$ is precisely what it would be 
in the direct product, so only two first cases remain.

Let $a\not\leq b$ and $i\leq j$, so that $\lan a,i\ran\lres \lan b,j\ran =
\lan a\lres b\wedge c,q\ran$. 
Note that we must have $j < 1$. For all $\lan s,k\ran$, we have
$\lan a,i\ran\cdot\lan s,k\ran = \lan a\cdot s, i\cdot k\ran\leq\lan b,j\ran$ if
and only if 
$a\cdot s\leq b$ and $k\cdot i\leq j$. Therefore
$s\leq a\lres b$ and since $a\lres b < 1$ we have $\lan s,k\ran<1$, so 
$s\leq c$ and $k\leq q$.
Thus, $\lan s,k\ran \leq \lan a\lres b\wedge c,q\ran$.
Conversely, $\lan s,k\ran\leq \lan a\lres b\wedge c,q\ran$ if and only  
if $s\leq a\lres b\wedge c$ and $k<1$, which implies $a\cdot s\leq b$ and $i\cdot k\leq
i\leq j$.  The case $a\leq b$ and $i\not\leq j$ is symmetric.

To show that $\m{A}\odot\m{B}$ is subdirectly irreducible it suffices to note that
$\lan c,q\ran$ is its unique coatom.
\end{proof}

If $\m{A}$ and $\m{B}$ are themselves integral and have unique coatoms,
all nontrivial quotients of $\m{A}\odot\m{B}$ coincide with quotients of
$\m{A}\times\m{B}$.

\begin{lemma}\label{congr}
Let $\m{A}$, $\m{B}$ be integral residuated lattices, with unique coatoms
$c$, $q$, respectively. Let $\alpha$ be a
non-trivial congruence on~$\m{A}\odot\m{B}$. Then there is a congruence
$\alpha'$ on $\m{A}\times\m{B}$ that extends $\alpha$ and satisfies
$(\m{A}\times\m{B})/\alpha' \cong (\m{A}\odot\m{B})/\alpha$. Moreover, $\alpha'
= (\rho_1\vee\alpha')\times(\rho_2\vee\alpha')$, where $\rho_1$, $\rho_2$ are
the kernels of the respective projection homomorphisms. 
\end{lemma}  

\begin{proof}
Let $F$ be the filter of $\m{A}\odot\m{B}$ corresponding to the congruence $\alpha$.
Note that the ordered monoid reduct of $\m{A}\odot\m{B}$ is a subalgebra
of the ordered monoid reduct of $\m{A}\times\m{B}$.
Consider $\up F$ taken in $\A\times \B$. We have
\begin{multline*}
\up F = F\cup \{(x,1)\in A\times B\mid\exists y\in B\colon (x,y)\in F\}\\
\cup \{(1,y)\in A\times B\mid \exists x\in A\colon (x,y)\in F\}.
\end{multline*}
Then $\up F$ is a filter on $\m{A}\times\m{B}$, and it is also closed under
multiplication. Let $\alpha'$ be the congruence on
$\m{A}\times\m{B}$ determined by $\up F$. Then $\alpha'\restrictedto{A\odot B} = \alpha$,
so $\alpha'$ extends $\alpha$.

As $\alpha$ is non-trivial we have $(\lan c,q\ran, \lan 1, 1\ran)\in \alpha$, 
whence
$(\lan c,1\ran, \lan 1, 1\ran)\in \alpha'$ and
$(\lan 1,q\ran, \lan 1, 1\ran)\in \alpha'$. 
Therefore, $(\lan c,i\ran, \lan 1, i\ran)\in \alpha'$ and
$(\lan a,q\ran, \lan a, 1\ran)\in \alpha'$ for all
$i \in B$, $a\in A$.
It follows that every congruence class of $\alpha'$ contains
a representative from $A\odot B$, so the natural map
$a/\alpha \mapsto a/\alpha'$ is bijective. 
Inspecting the definitions of the operations in $\m{A}\odot\m{B}$
we see that the map above is also a homomorphism.
The moreover part follows from congruence distributivity.
\end{proof}  

Truncated products of integral residuated lattices with unique coatoms commute
with ultrapowers. In will be important in subsection~\ref{ssec:two} below.

\begin{lemma}\label{ultrap}
Let $\m{A}$, $\m{B}$ be as in Lemma~\ref{congr}. Then,  
any ultrapower $(\m{A}\odot\m{B})^I/U$ is isomorphic to
$\m{A}^I/U\odot\m{B}^I/U$.  
\end{lemma}  

\begin{proof}
This is a slight variation of the standard proof of the fact that ultraproducts commute with finite products. The reader is asked to verify that the map 
\[
\bigl\lan (a_i\mid i\in I)/U,(b_i\mid i\in I)/U\bigr\ran\mapsto 
\bigl(\lan a_i, b_i\ran_i\mid i\in I\bigr)/U
\] 
is the required isomorphism. The only differences from the standard proof for
products are: (i) the condition that $a_i = 1$ if and only if $b_i = 1$, and 
(ii) the definition of residuation.      
\end{proof}

\subsection{No splittings: algebras with at least three elements}\label{ssec:three}

Now, let $\m{A}$ be a finite subdirectly irreducible commutative integral
residuated lattice, with at least three elements and with monolith $\mu$ of depth
$n$. We shall apply the Non-splitting Lemma~\ref{no-splitting} to prove that $\A$
does not split the subvariety lattice of $\mathsf{CIRL}$.  
Since $\delta x = x^2$, 
\[
\mathbf{B}\not\models   
\delta^k(\Delta_{\mathbf{A}})\leq x_{\mu_\bot}\quad\text{becomes}\quad
\mathbf{B}\not\models  (\Delta_{\mathbf{A}})^{2k}\leq x_{\mu_\bot}. 
\]
Let $i\in \mathbb{N}$ and 
let $\m{E}$ be the expansion of $\m{A}$ 
of even depth $m = 2k$ with $k\geq i$ and $m>n$.
For the distortion part of the construction, we will make use of \emph{Wajsberg
  hoops} (see Blok and Ferreirim~\cite{BF00} for more on hoops).
Recall that a Wajsberg hoop $\m{C}_n$ is the commutative lattice-ordered
monoid on the universe $\{0,-1,\dots,-n+1\}$, with truncated addition and  
with residuation defined naturally by 
$i\ra j = \max\{0, i-j\}$. For consistency with previous notation (and 
tradition) we will present Wajsberg hoops multiplicatively, defining
$q^i = -i$, so that $1 = q^0 = 0$, $q = q^1 = -1$, and so on. Clearly,
$q$ is the unique coatom of $\m{C}_{n}$. 

Now, for the first prime $p$ with $p\geq |E|$, consider $\m{C}_{p+1}$.
Note that $\m{C}_{p+1}$ is strictly simple, and has $p+1$ elements.
Form $\m{E}\odot\m{C}_{p+1}$ with $c$ and $q$ chosen to be the unique coatoms
of $\m{E}$ and $\m{C}_{p+1}$, respectively.   
In the next three lemmas we show that
$\A \notin \Var(\m{E}\odot\m{C}_{p+1})$ and 
$\m{E}\odot\m{C}_{p+1}\not\models  (\Delta_{\mathbf{A}})^{m}\leq x_{\mu_\bot}$,
hence establishing Condition (3) of the Non-splitting  
Lemma~\ref{no-splitting} with $\B := \m{E}\odot\m{C}_{p+1}$.

Note that since $\m{E}$ and $\m{C}_{p+1}$ are commutative 
and integral, and $c$ and $q$ are their largest strictly negative elements, the definitions of residuals in $\m{E}\odot\m{C}_{p+1}$ simplify to:
\begin{align*}
\lan a,i\ran\ra \lan b,j\ran    &=   
\begin{cases}
\lan a\ra b, q\ran & \text{ if } a\not\leq b,\ i\leq j,\\
\lan c, i\ra j\ran & \text{ if } a\leq b,\ i\not\leq j,\\
\lan a\ra b, i\ra j\ran & \text{ otherwise.}
\end{cases}
\end{align*}

\begin{lemma}\label{hom-image}
For each non-trivial congruence $\vartheta$ on $\m{E}\odot\m{C}_{p+1}$, 
the quotient algebra $(\m{E}\odot\m{C}_{p+1})/\vartheta$ is isomorphic to a proper 
homomorphic image of $\m{A}$.
\end{lemma}

\begin{proof}
By Lemma~\ref{congr} we have that $(\m{E}\odot\m{C}_{p+1})/\alpha$ is
isomorphic to $(\m{E}\times\m{C}_{p+1})/\alpha'$, where $\alpha'$ is the
natural extension of $\alpha$.  Since $\alpha$ is non-trivial, we have
$(d,q)\in \alpha'$, and so 
$\alpha' = \beta\times\gamma$ for some congruences
$\beta$ on $\m{E}$ and $\gamma$ on $\m{C}_{p+1}$ such that
$(d,1)\in\beta$ and $(q,1)\in\gamma$. But $\m{C}_{p+1}$ is simple, so
$\gamma$ is the full congruence on $\m{C}_{p+1}$, and therefore
$\alpha' = \beta\times 1$, that is, $\alpha'$ properly contains the projection
onto the first coordinate. Since $\beta$ is non-trivial, the claim follows.
\end{proof}

\begin{lemma}\label{not-in-variety}
If $\m{A}$ is not isomorphic to $\m{2}$, then 
$\m{A}\not\in \Var(\m{E}\odot\m{C}_{p+1})$.
\end{lemma}

\begin{proof}
By J\'onsson's Lemma, the congruence extension property and finiteness, we have
that if $\m{A}\in \Var(\m{E}\odot\m{C}_{p+1})$, then
$\m{A}\in \SH(\m{E}\odot\m{C}_{p+1})$. 
By Lemma~\ref{hom-image}, for each non-trivial congruence $\vartheta$
on $\m{E}\odot\m{C}_{p+1}$ we have $|E\odot C_{p+1}/\theta| < |A|$, and   
thus $\m{A}\notin \SH^+(\m{E}\odot\m{C}_{p+1})$, where
$\mathbb{H}^+$ stands for nontrivial homomorphic images.
 
It follows that there is an embedding $e\colon \m{A}\to \m{E}\odot\m{C}_{p+1}$  of $\m{A}$ into $\m{E}\odot\m{C}_{p+1}$. We will derive a contradiction from this. First we shall show that $\pi_2 \circ e\colon \A \to \m{C}_{p+1}$ is an embedding, where $\pi_2\colon \m{E}\odot\m{C}_{p+1} \to \m{C}_{p+1}$ is the restriction of the second projection.
Suppose that $\pi_2 \circ e$ is not an embedding; then there exist $a, b\in A$ with $a \not\leq b$ such that $e(a) = \lan v,i\ran$ and $e(b) = \lan u,i\ran$. Then $a\ra b\neq 1$ and so
$\lan v,i\ran\ra \lan u,i\ran\neq\lan 1,1\ran$ and $v \not \leq u$, so  
$e(a\ra b) = \lan v,i\ran\ra \lan u,i\ran = \lan v\ra u, q\ran$.
Put $w = v\ra u$. Consider the chain of powers of
$\big\{\lan w,q\ran^s\mid 1\leq s\leq p\big\}$.
By definition, $\lan w,q\ran^s = \lan w^s,q^s\ran$, and
since $q, q^2,\dots, q^p$ are all distinct, this chain has $p$ distinct
elements and so the element $a\ra b$ generates $p$ distinct
elements in $\m{A}$. But $p \geq |E| > |A|$ so this is a contradiction. 
It follows that $\pi_2 \circ e\colon \A \to \m{C}_{p+1}$ is an embedding.

Since $\m{A}$ has at least three elements, the subalgebra $(\pi_2 \circ e)(\A)$ of $\m{C}_{p+1}$ contains a non-zero and non-unit element. But in $\m{C}_{p+1}$ each non-zero and non-unit element generates the whole algebra, so $(\pi_2 \circ e)(\A) = \m{C}_{p+1}$, whence $\pi_2 \circ e\colon \A \to \m{C}_{p+1}$ is an isomorphism. This gives $|A| = p+1$, contradicting the fact that $|A| <p$.
\end{proof}

\begin{lemma}\label{assign}
$\m{E}\odot\m{C}_{p+1}\not\models (\Delta_{\m{A}})^m\leq x_{\mu_\bot}$.   
\end{lemma}

\begin{proof}
Define an $|A|$-tuple $\ov{w}$ putting
$w_1 = \lan 1,1\ran$ and $w_a = \lan a,q\ran$, if
$a\neq 1$. Note that $w_{\mu_\bot} = \lan \mu_\bot, q\ran$.
For $\diamond \in\{\vee,\wedge,\ra\}$, we have
\[
w_{a\diamond b}\lra w_a\diamond w_b = 
\lan a\diamond b, q\ran\lra \lan a,q\ran\diamond\lan b,q\ran =
\lan a\diamond b,q\ran\lra \lan a\diamond b, q\ran = \lan 1,1\ran. 
\]
In the case of multiplication, we have:
\begin{multline*}
w_{a\cdot b}\lra w_a\cdot w_b = 
\lan a\cdot b,q\ran\lra \lan a,q\ran\cdot \lan b,q\ran =
\lan a\cdot b,q\ran\lra\lan a\cdot b, q\cdot q\ran\\ 
= \lan a\cdot b,q\ran\ra\lan a\cdot b, q^2\ran
 = \lan d, q\ra q^2\ran = \lan d,q\ran.
\end{multline*}
It follows that $\Delta_{\m{A}}[\ov{w}] = \lan d, q\ran$.
Recall that the depth of the monolith of $\m{E}$ is~$m$. We then reason as follows. 
First, the lattice operations and multiplication
in ${\m{E}\odot\m{C}_{p+1}}$ coincide with the operations in the direct product
$\m{E}\times\m{C}_{p+1}$. Thus, for each $k$ with $1<k < m$, 
we have $\lan d,q\ran^k = \lan d^k, q^k\ran$.
Since $k < m<|E|\leq p$, we have that $\lan d^k, q^k\ran\not\leq
\lan d^m,q\ran = \lan \mu_\bot, q\ran$. Hence,
$(\Delta_{\m{A}})^k [\ov{w}]\not\leq w_{\mu_\bot}$, that is,
${\m{E}\odot\m{C}_{p+1}\not\models (\Delta_{\m{A}})^k\leq x_{\mu_\bot}}$ for all
$k\leq m$. In particular, $\m{E}\odot\m{C}_{p+1}\not\models
(\Delta_{\m{A}})^m\leq x_{\mu_\bot}$, as claimed.
\end{proof}

\begin{theorem}\label{splits-not}
No finite subdirectly irreducible algebra $\m{A}\in \mathsf{CIRL}$ with at least
three elements splits the subvariety lattice of\/ $\mathsf{CIRL}$.
\end{theorem}

\begin{proof}
By the Non-splitting Lemma~\ref{no-splitting} and 
Lemmas~\ref{not-in-variety} and~\ref{assign}.
\end{proof}

\subsection{No splittings: the two-element algebra}\label{ssec:two}

Finally, we consider the case of the two-element algebra. By integrality,
this algebra is isomorphic to the $0$-free reduct of the two-element
Boolean algebra, that is, to the Wajsberg hoop $\m{C}_2$. 
It follows from the construction that
the expansion of $\m{C}_2$ with monolith of depth $2n$ is isomorphic
to~$\m{C}_{2^n+1}$. In particular, every expansion of
$\m{C}_2$ is simple.

Let $\m{C}_{\omega}$ be the infinite simple
Wajsberg hoop, that is, the commutative lattice-ordered monoid on the universe
$\{0,-1,-2,\dots\}$ with the usual addition, and with residuation defined by 
$i\ra j = \max\{0, i-j\}$, analogously to $\m{C}_{n}$ of the previous subsection.
As with $\m{C}_{n}$, we present $\m{C}_{\omega}$ multiplicatively, putting
$q^i = -i$, so that $1 = q^0 = 0$ and $q = q^1 = -1$ is the unique coatom
of $\m{C}_{\omega}$.  

\begin{lemma}\label{hom-image-2}
Let $\alpha\neq 0$ be a congruence on an ultrapower
$\m{G} = (\m{C}_{2^n+1}\odot\m{C}_{\omega})^I/U$. The
quotient algebra $\m{G}/\alpha$ is isomorphic to a proper 
homomorphic image of $\m{C}_{\omega}^I/U$.
\end{lemma}

\begin{proof}
To lighten the notation, we let $\m{E} = \m{C}_{2^n+1}$ and
$\m{H} = \m{C}_{\omega}^I/U$.
Since $\m{E}$ is finite, it follows by Lemma~\ref{ultrap} that
$\m{G} \cong \m{E}\odot\m{H}$. By 
Lemma~\ref{congr} we have that $(\m{E}\odot\m{H})/\alpha$ is
isomorphic to $(\m{E}\times\m{H})/\alpha'$, where $\alpha'$ is the
natural extension of $\alpha$ from that lemma.
Let $\ov{q} = q^I/U$ so that $\ov{q}$ is the unique coatom of
$\m{H}$, and let $d$ be the unique coatom of $\m{E}$.
Then, since $\alpha$ is non-trivial, we have
$(d,\ov{q})\in \alpha'$, and so
$\alpha' = \beta\times\gamma$ for some congruences
$\beta$ on $\m{E}$ and $\gamma$ on $\m{H}$ such that
$(d,1)\in\beta$ and $(\ov{q},1)\in\gamma$. But $\m{E}$ is simple, so
$\beta$ is the full congruence on~$\m{E}$, and therefore
$\alpha' = 1\times\gamma$, that is, $\alpha'$ properly contains the projection
onto the second coordinate. Since $\gamma$ is non-trivial, the claim follows.
\end{proof}

\begin{lemma}
Let $\m{E}$ be an expansion of\/ $\m{C}_2$. Then  
$\m{C}_2$ does not belong to the variety
$\Var(\m{E}\odot\m{C}_{\omega})$.
\end{lemma}

\begin{proof}
By J\'onsson's Lemma and the congruence extension property, we have that
if $\m{C}_2\in \Var(\m{E}\odot\m{C}_{\omega})$, then
$\m{C}_2\in \mathbb{SHP_U}(\m{E}\odot\m{C}_{\omega})$.
By Lemma~\ref{hom-image-2}, for every non-trivial congruence
$\vartheta$ on any ultrapower $(\m{E}\odot\m{C}_{\omega})^I/U$ we have
that $(\m{E}\odot\m{C}_{\omega})^I/U$ is isomorphic to a proper quotient of
$\m{C}_{\omega}^I/U$. Since the cancellative identity ${x\ra xy = y}$ 
holds in $\m{C}_{\omega}$, it holds in every member of
$\mathbb{SH^+P_U}(\m{E}\odot\m{C}_{\omega})$, where $\mathbb{H^+}$ stands for proper homomorphic
images. Thus, if $\m{C}_2\in \Var(\m{E}\odot\m{C}_{\omega})$, then
$\m{C}_2 \in \SPU(\m{E}\odot\m{C}_{\omega})$. However, it is easy to see from the
construction that $\m{E}\odot\m{C}_{\omega}$ contains no idempotent elements
distinct from $1$, and since this property is expressible by a universal formula, it is
preserved by $\SPU$. But the bottom element of $\m{C}_2$ is an idempotent distinct
from $1$, so ${\m{C}_2\notin \SPU(\m{E}\odot\m{C}_{\omega})}$ proving the claim.
\end{proof}

\begin{theorem}\label{2-splits-not}
The algebra $\m{C}_2$ does not split the subvariety lattice of\/
$\mathsf{CIRL}$.
\end{theorem}  

\begin{proof}
We shall establish Condition(2) of the Non-splitting
Lemma~\ref{no-splitting}. Let $i\in\mathbb{N}$. 
Let $\m{E}$ be an expansion of $\m{C}_2$ with the
monolith of depth greater than $i$. Consider $\m{E}\odot\m{C}_{\omega}$, and 
the pair $\ov{w} = (w_1,w_0)$ with
$w_1 = \lan 1,1\ran$ and $w_{\mu_\bot} = w_0 = \lan 0,q\ran$.
By inspecting the diagram $\Delta_{\m{C}_2}$ we
see that ${w_{0\cdot 0} = w_{0} = \lan 0,q\ran}$ and
therefore $w_{0\cdot 0}\eqv w_0\cdot w_0 =
\lan 0,q\ran \eqv \lan 0,q^2\ran = \lan d, q\ran$. 
Further inspection of the diagram ensures that
$(\Delta_{\m{C}_2})[\ov{w}] = \lan d, q\ran$. 
Now, $0 = d^{2n}$ with $2n \geq i$, so 
$(\Delta_{\m{C}_2})^i[\ov{w}] = \lan d^i, q^i\ran \not\leq \lan 0, q \ran = w_0$. 
The statement now follows by the Non-splitting Lemma.
\end{proof}  

\subsection{More negative results and some questions}
We have applied the expand-and-distort technique inside the variety
of commutative integral residuated lattices, but it clearly applies in any
variety of residuated lattices containing $\mathsf{CIRL}$. Thus, the next
result is immediate.

\begin{corollary}\label{outside-CIRL}
Let $\CR$ be a variety of residuated lattices containing
$\mathsf{CIRL}$. Let $\m{A}\in \CR$ be finite, subdirectly irreducible,
commutative and integral. Then $\m{A}$ is not a splitting algebra in
$\CR$. In particular, the variety $\mathsf{CIRL}$ has no splittings.
\end{corollary}  

Outside $\mathsf{CIRL}$ the situation is less clear. In an arbitrary variety
$\CR$ containing $\mathsf{CIRL}$ there may exist non-integral or
non-commutative splitting algebras, perhaps infinite, if
$\CR$ is not generated by its finite members. 
For subvarieties of $\mathsf{CIRL}$, the following result is
all we know.

\begin{corollary}\label{inside-CIRL}
Let $\CR$ be a subvariety of $\mathsf{CIRL}$. If $\CR$
contains the variety of Wajsberg hoops and is closed under expansions and
truncated products, then no finite subdirectly irreducible algebra in
$\CR$ is splitting. 
\end{corollary}

\begin{proof} It is known (cf.~\cite{BF00}) that the variety of Wajsberg hoops is
  generated by 
the algebras $\m{C}_n$ for $n\in\mathbb{N}$. The conclusion follows immediately.
\end{proof}

The variety of Wajsberg hoops, and  the variety of all hoops are
closed under expansions. In fact, the expansion part of our constructions is
modelled after Wajsberg hoops (we encourage the reader to verify that the first
expansion of $\m{C}_n$ is $\m{C}_{2n-1}$). Neither is closed under
truncated products, so it would be interesting to characterise the smallest
variety containing (Wajsberg) hoops and closed under truncated products.
By Corollary~\ref{inside-CIRL} that variety contains no finite splitting
algebras, or no splitting algebras at all, if it is generated by its finite
members.  

%%%%%%%%%%%%%%%%%%%%%%%%%%%%%%%%%%%%%%%%%%%%%%%%%%%%%%%%%%%%%%%%%%%%%%%
\section{Cousins of double Heyting algebras}\label{sec:d-Heyting}
%%%%%%%%%%%%%%%%%%%%%%%%%%%%%%%%%%%%%%%%%%%%%%%%%%%%%%%%%%%%%%%%%%%%%%%

Recall that a double Heyting algebra is an algebra $\langle A;\join,\meet,\to,\dpc,0,1\rangle$ such that $\langle A;\join,\meet,\to,0,1\rangle$ is a Heyting algebra and $\langle A;\join,\meet,\dotdiv,0,1\rangle$ is a dual Heyting algebra.
\begin{definition}
An algebra $\langle A, \join, \meet, \to, \dpc, 0, 1\rangle$ is a \emph{dually pseudocomplemented Heyting algebra} (\h{} for short) if $\langle A; \join, \meet, \to, 0, 1\rangle$ is a Heyting algebra and $\dpc$ is a \emph{dual pseudocomplement} operation, i.e., 
\[
x \join y = 1 \iff y \geq \dpc x.
\]
For algebras with a Heyting algebra term reduct, we define the \emph{pseudocomplement} operation by $\neg x = x \to 0$. 
Similarly, for algebras with a dual Heyting algebra term reduct, the dual pseudocomplement operation is given by $\dpc x = 1 \dotdiv x$.
If $\mathbf{A}$ is a double Heyting algebra, then $\mathbf{A}^\flat$ will denote the $\langle \join, \meet, \to, \dpc, 0, 1 \rangle$-term reduct of~$\mathbf{A}$. 
Let $\mathsf{H}$ denote the class of Heyting algebras, let 
$\mathsf{H}^+$ denote the class of \h{s} and let $\mathsf{DH}$ denote the class of double Heyting algebras. 
\end{definition}

\begin{remark}
The abbreviation to \h{} is derived from an alternative notation $x^+$ instead of $\dpc x$ for the dual pseudocomplement. 
\end{remark}

For some algebraic properties of double Heyting algebras,
see Rauszer~\cite{rauszer1,rauszer2},
Wolter~\cite{wolter}, Sankappanavar~\cite{normalfilters},
and Taylor~\cite{Taylor2017}; refer to~\cite{normalfilters}
and~\cite{Taylor2017} for more on \h{s}.
In particular, both $\mathsf{DH}$ and $\mathsf{H}^+$ are equational classes. 
The following result
% of Sankappanavar
shows that the theory of \h{s} encapsulates much of the theory of double Heyting algebras.

\begin{theorem}[Sankappanavar~\cite{normalfilters}]\label{congruence-isomorphism}
	Congruences on a double Heyting algebra $\mathbf{A}$ are exactly the congruences on the \h{} term reduct of $\mathbf{A}$.
	More succinctly,
\[
\con(\mathbf{A}) = \con(\mathbf{A}^\flat).
\]
\end{theorem}

Also related is the class of congruence-regular double p-algebras.
A \emph{double p-algebra} is an algebra $\mathbf{A} = \langle A; \join, \meet, \neg, \dpc, 0, 1\rangle$ such that $\langle A; \join, \meet, 0, 1\rangle$ is a bounded lattice and $\neg$ and $\dpc$ are pseudocomplement and dual pseudocomplement operations, respectively.
An algebra $\mathbf{A}$ is \emph{congruence-regular} (or \emph{regular} for short) if, whenever two congruences on $\mathbf{A}$ share a class, they are 
equal. For the next result, Varlet~\cite{regdp,dpregular} proved the equivalence of conditions~\ref{varlet:1}, \ref{varlet:2}, and \ref{varlet:3}; condition~\ref{varlet:4} was included under the assumption of distributivity. 
Katri\v{n}\'{a}k~\cite{doublep} extended this by proving that \ref{varlet:2} implies distributivity.

\begin{theorem}[Varlet~\cite{regdp,dpregular}, Katri\v{n}\'{a}k~\cite{doublep}]\label{varlet}
Let $\mathbf{A}$ be a double p-algebra.
The following are equivalent\textup:
\begin{enumerate}[label={\upshape(\arabic*)},leftmargin=1.75\parindent]
\item $\mathbf{A}$ is regular\textup;\label{varlet:1}
\item for all $x,y \in A$, if $\neg x = \neg y~\text{and}~\dpc x = \dpc y$, then $x = y$\textup;\label{varlet:2}
\item every prime filter of $\mathbf{A}$ is minimal or maximal\textup; \label{varlet-prime-filters}\label{varlet:3}
\item $\mathbf{A}$ is distributive and $\mathbf{A} \models \dpc x \meet x \leq y \join \neg y$. \label{varlet:4}
\end{enumerate}
\end{theorem}

Notice that \ref{varlet:4}
provides an equational characterisation of regular double p-algebras.
\begin{definition} Let $\mathsf{RDP}$ denote the variety of regular double p-algebras.\end{definition}

The following result of Katri\v{n}\'{a}k shows that regular double p-algebras form a natural class of double Heyting algebras.

\begin{theorem}[Katri\v{n}\'{a}k~\cite{doublep}]\label{regular-implies-heyting}
Let $\mathbf{A}$ be a congruence-regular double p-algebra.
Then $\mathbf{A}$ is term-equivalent to a double Heyting algebra via the term
\[
x \rightarrow y = \neg\neg(\neg x \join \neg\neg y) \meet [\dpc (x \join \neg x) \join \neg x \join y \join \neg y]
\]
and its dual.
\end{theorem}

Thus $\mathsf{RDP}$ is term-equivalent to a subvariety of $\mathsf{H}^+$ and $\mathsf{DH}$.
By Theorem~\ref{congruence-isomorphism}, all double Heyting algebras and regular double p-algebras have their
congruences determined by their \h{} term reducts, and we will now see that they
fit exactly into the framework defined in Section~\ref{sec:split-alg-var-log}.

\begin{definition}
	Let $\mathbf{A}$ be an algebra with a Heyting algebra term reduct. 
	For all $x,y \in A$, let $x \lr y = (x \to y) \meet (y \to x)$. 
	For every filter $F$ of $\mathbf{A}$, let $\theta(F)$ be the Heyting algebra congruence given by
	\[
		\theta(F) = \{(x,y) \in A^2 \mid x \lr y \in F\}.
	\]
	If $\mathbf{A}$ has an \h{} term reduct, we define the term 
$\delta$ on $\mathbf{A}$ by $\delta x = \ps x$. 
\end{definition}

\begin{theorem}[Sankappanavar~\cite{normalfilters}]
Let $\mathbf{A}$ be an \h{} and let $F$ be a filter of $\mathbf{A}$. 
Then $\theta(F)$ is an \h{} congruence if and only if $F$ is closed under~$\delta$.
If $F$ is closed under $\delta$, then $F = 1/\theta(F)$.
\end{theorem}

It is easily verified that the term $\delta$ is order-preserving and satisfies $\delta x \leq x$; thus the variety
of \h{s} satisfies the properties (P1)--(P5) from Section~\ref{sec:split-alg-var-log}. 
Then, by Theorem~\ref{congruence-isomorphism}, the variety of double Heyting
algebras also fits into the framework. 
The next result lists without proof some properties of \h{s} we will apply.
The interested reader will find a characterisation of the subdirectly
irreducible algebras in Sankappanavar~\cite{normalfilters}. 

\begin{proposition}\label{useful}
	Let $\mathbf{A}$ be an \h{}. 
\begin{enumerate}[label={\upshape(\arabic*)},leftmargin=1.75\parindent]
\item For all $x \in A$, the following are equivalent\textup:
\begin{enumerate}[label={\upshape(\alph*)}]
\item $x$ is complemented\textup;
\item $\neg x$ is the complement of $x$\textup;
\item $\neg x = \dpc x$\textup;
\item $\neg\neg x = x$ or $\neg\dpc x = x$ or $\dpc\neg x = x$ or $\dpc\dpc x = x$.
\end{enumerate}
\item If $\mathbf{A}$ is subdirectly irreducible, then $0$ and $1$ are the only complemented elements in $\mathbf{A}$.
\item If $\mathbf{A}$ is simple, then for all $x \in A\comp\{1\}$, there exists $n \in \omega$ such that $\delta^nx = 0$.
\item If $\mathbf{A}$ is finite and subdirectly irreducible, then $\mathbf{A}$ is simple.
\end{enumerate}
\end{proposition}

\goodbreak

Naturally, these results also hold for double Heyting algebras and regular double p-algebras.
The subvariety of $\mathsf{H}^+$ satisfying the identity $\dpc x \approx \neg x$ is just the variety of \h{s} whose underlying lattice is Boolean, and similarly for $\mathsf{DH}$ and $\mathsf{RDP}$.
That subvariety is term-equivalent to the variety of Boolean algebras, so we will call it \emph{the variety of Boolean algebras}, and the adjective \emph{Boolean} will be used to describe its members. 

\subsection{The restricted Priestley duality for $\mathsf{H}$, $\mathsf{H^+}$ and $\mathsf{DH}$}

To apply the Non-splitting Lemma~\ref{no-splitting}, we will make use of a similar ``expand and distort'' technique as for residuated lattices.
Instead of distorting the lattice structure directly, we will utilise the topological duality for distributive lattices.

\begin{definition}Let $\X = \langle X; \leq\rangle$ be an ordered set and let $Y \subseteq X$. 
We will let $\leq_Y$ denote the order $\leq$ restricted to $Y$, that is, ${\leq_Y} = Y^2 \cap {\leq}$. 
The set of minimal elements of $\X$ will be denoted by $\min(\X)$, and for each $Y \subseteq X$, we let ${\min_\X(Y) = \min(\X) \cap Y}$.
Similarly, the set of maximal elements of $\X$ will be denoted by $\max(\X)$
and we let $\max_\X(Y) = \max(\X) \cap Y$. 
Define $\up Y := \bigcup\{\up y \mid y\in Y\}$ and $\down Y := \bigcup\{\down y \mid y\in Y\}$ and  
let $\updown Y = \up Y \cup \down Y$. 
Note that there is a distinction between $\updown Y$ and the sets $\up\down Y$ and $\down\up Y$.
We will say that $\X$ is \emph{connected} if, for all $x \in X$, there exists $n \in \mathbb{N}$ such that $\updown^n x = X$. The set $Y$ is an \emph{up-set} of $\X$ if $\up x \subseteq Y$, for all $x\in Y$, and the lattice of up-sets of $\X$ will be denoted by $\bu(\X)$.
\end{definition}

We will not dwell on
% the details of Priestley's
Priestley duality for distributive
lattices and treat it as assumed knowledge, referring 
readers to Davey and Priestley~\cite{ilo} for further detail. 
To ensure that the reader is oriented correctly, we will note that for the purposes of this paper, the dual space of a distributive lattice is the space of prime filters, and the lattice is recovered by taking clopen up-sets. 

\begin{definition}
Let $\mathbf{L}$ be a bounded distributive lattice and let $\mathcal{F}_p(\mathbf{L})$ denote the set of prime filters of $\mathbf{L}$.
Then the \emph{Priestley dual} of $\mathbf{L}$ is the ordered topological space $\cat F_{\!p}(\mathbf{L}) = \langle \mathcal{F}_p(\mathbf{L}); \subseteq, \mathcal{T}\rangle$, where the topology $\mathcal{T}$ is generated by the sub-basis 
\[
\{X_a \mid a \in L\} \cup \{\mathcal{F}_p(\mathbf{L})\comp X_a \mid a \in L\},
\]
with $X_a = \{F \in \mathcal{F}_p(\mathbf{L}) \mid a \in F\}$.
A \emph{Priestley space} is a structure $\X = \langle X; \leq, \mathcal{T}\rangle$ such that $\langle X; \leq\rangle$ is an ordered set, 
$\langle X; \mathcal{T}\rangle$ is a compact topological space and, for all $x,y \in X$ with $x \nleq y$, there exists a clopen up-set $U$ such that $x \in U$ and $y \notin U$.
The lattice of clopen up-sets of a Priestley space $\X$ is denoted by $\but(\X)$, and the context will determine any further algebraic structure.
\end{definition}

%Priestley's
Priestley duality establishes that the category of bounded distributive lattices with bounded lattice homomorphisms is dually equivalent to the category of Priestley spaces with continuous order-preserving maps.
The properties in the next result will be used at various times without reference.

\begin{proposition}
Let $\X$ be a non-empty Priestley space. 
\begin{enumerate}[label={\upshape(\arabic*)},leftmargin=1.75\parindent]
\item The sets $\min(\X)$ and $\max(\X)$ are non-empty.
Moreover, for all $x \in X$, both $\min_\X(\down x)$ and $\max_\X(\up x)$ are non-empty. \label{priestley:1} 
\item Let $Y$ and $Z$ be disjoint closed subsets of $\X$ such that $Y$ is an up-set and $Z$ is a down-set. Then there exists a clopen up-set $W$ such that $Y \subseteq W$ and $W \cap Z = \varnothing$.\label{priestley:2}
\item If $\but(\X)$ is pseudocomplemented, then $\max(\X)$ is closed, and if $\but(\X)$ is dually pseudocomplemented, then $\min(\X)$ is closed. \label{priestley:3}
\end{enumerate}
\end{proposition}
For \ref{priestley:1} and \ref{priestley:2}, see Exercise 11.15 and Lemma~11.21 in~\cite{ilo}.
A proof of \ref{priestley:3} can be found in~\cite{priestley-p-algebra}.

\begin{definition}\label{space-definition}Let $\X$ be a Priestley space.
Consider the following three conditions on $\X$:
\begin{enumerate}[label=(S\arabic*), ref=S\arabic*,leftmargin=3em]
\item $\down U$ is open, for every open set $U$ in $\X$, \label{p:cond:1}
\item $\up U$ is open, for every open set $U$ in $\X$, \label{p:cond:2}
\item $\up U$ is open, for every clopen down-set $U$ in $\X$. \label{p:cond:3}
\end{enumerate}
A Priestley space is a \emph{Heyting space} if it satisfies \eqref{p:cond:1}, an \emph{\h[space]{}} if it satisfies \eqref{p:cond:1} and \eqref{p:cond:3}, and a \emph{double Heyting space} if it satisfies \eqref{p:cond:1} and \eqref{p:cond:2}.
\end{definition}

In~\cite{stone}, Priestley classified the dual spaces of distributive
pseudocomplemented lattices and it was further elaborated on in
Priestley~\cite{priestley-p-algebra}.
The restricted Priestley
duality for Heyting algebras is generally attributed to Esakia~\cite{esakia} and often treated as folklore.
A detailed exposition can be found in the appendix of Davey and
Galati~\cite{coalgheyting}. 
Combining the results of those papers and dualising appropriately yields the next theorem.

\begin{theorem}\label{space-characterisation}
Let $\X$ be a Priestley space.
Then $\X$ is a Heyting space \textup(resp.\
H\textsuperscript{+}-space, double Heyting space\textup) if and only if $\but(\X)$ is the underlying lattice of a Heyting algebra \textup(resp.\
H\textsuperscript{+}-algebra, double Heyting algebra\textup). 
\end{theorem}

\begin{lemma}\label{dual-operations}
Let $\X$ be a Priestley space and let $U, V \in \ut(\X)$.
If the corresponding operation is defined in $\but(\X)$, then
\begin{enumerate}[label={\upshape(\arabic*)},leftmargin=1.75\parindent]
\item $\neg U = X \comp \down U$,
\item $\dpc U = \up (X \comp U)$,
\item $U \to V = X \comp \down(U \comp V)$, 
\item $U \dotdiv V = \up(U \comp V)$, 
\item $\ps U = X \comp \down\up(X \comp U)$.
\end{enumerate}
\end{lemma}

\begin{definition}\label{morphism-definition}
 Let $\X$ and $\Y$ be Priestley spaces and let $\varphi \colon X \to Y$ be a continuous order-preserving map.
We will then say that $\varphi\colon \X \to \Y$ 
is a \emph{morphism}.
Consider the following three conditions on $\varphi$:
\begin{enumerate}[label=(M\arabic*), ref=M\arabic*,leftmargin=3em]
\item $\forall x \in X\colon \varphi(\up x) = \up \varphi(x)$, \label{m:cond:1}
\item $\forall x \in X\colon \varphi(\down x) = \down \varphi(x)$, \label{m:cond:2}
\item $\forall x \in X\colon \varphi(\min_\X(\down x)) = \min_\Y(\down\varphi(x))$. \label{m:cond:3}
\end{enumerate}
A morphism is a \emph{Heyting morphism} if it satisfies \eqref{m:cond:1}, an \emph{H\textsuperscript{+}-morphism} if it  satisfies \eqref{m:cond:1} and \eqref{m:cond:3}, and a \emph{double Heyting morphism} if it satisfies \eqref{m:cond:1} and~\eqref{m:cond:2}.
For each $U \subseteq Y$, let $\varphi^{-1}(U) = \{x \in X \mid \varphi(x) \in U\}$.
\end{definition}

Note that a double Heyting morphism is also an H\textsuperscript{+}-morphism.
Also note that either of the conditions \eqref{m:cond:1} and \eqref{m:cond:2} on their own imply that the map is order-preserving, whereas \eqref{m:cond:3} is independent of this fact.
Since we apply condition \eqref{m:cond:3} only
in tandem with \eqref{m:cond:1} or \eqref{m:cond:2}, the order-preserving assumption is redundant.
By combining results from the papers cited earlier we obtain the next result.

\begin{theorem}
Let $\X$ and $\Y$ be Priestley spaces and let $\varphi \colon X \to Y$ be a continuous map.
Then $\varphi$ is a Heyting morphism \textup(resp.\
H\textsuperscript{+}-morphism, double Heyting morphism\textup) if and only if the map $\varphi^{-1} \colon \but(\Y) \to \but(\X)$ is a Heyting algebra homomorphism \textup(resp.\
\h{} homomorphism, double Heyting algebra homomorphism\textup). 
\end{theorem}

\begin{definition}
For convenience, we will often leave the codomain of a morphism implicit.
If $\X$ and $\Y$ are Priestley spaces and $\varphi \colon \X \to \Y$ is a morphism, then we will say that $\varphi$ is a \emph{morphism on $\X$}. 
\end{definition}

The proof of the following useful lemma is completely trivial.
\begin{lemma}\label{minimals}
	Let $\X$ be an H\textsuperscript{+}-space, let $\varphi$ be an H\textsuperscript{+}-morphism on $\X$, and let~${x \in X}$.
If $x$ is maximal, then $\varphi(x)$ is maximal in $\codom(\varphi)$, and if $x$ is minimal, then $\varphi(x)$ is minimal in $\codom(\varphi)$.
\end{lemma}

\subsection{Finite embeddability property}\label{sec:fep}
To obtain a complete characterisation of splitting algebras in a variety
with the help of the Non-splitting Lemma~\ref{no-splitting}, we need to work with a variety generated by its finite members. It is well known that $\mathsf{H}$ is generated by its
finite members, and various proofs of this result exist. A~standard
algebraic proof uses the fact that distributive lattices are locally finite.
Using a straightforward modification of that proof, one obtains the same result
for $\mathsf{H}^+$ and~$\mathsf{DH}$. However, no easy modification of the
standard proof seems to work for $\mathsf{RDP}$. We will therefore prove a
stronger generic result, from which all the other results we need follow as
corollaries.
We note however that, using Priestley duality, Adams, Sankappanavar, and Vaz~de Carvalho~\cite{ASVdeC19} recently proved that $\mathsf{RDP}$ is generated by its finite members. They also studied the subvariety $\mathsf{RDP}_n$ of $\mathsf{RDP}$ determined by the identity $\delta^{n+1}(x) = \delta^n(x)$ and proved that they are also generated by their finite members. It is worth remarking that $\mathsf{RDP}_1$ is locally finite and so has many splittings.

We need a few technical concepts first. Let $\CV$ be a variety, let
$\mathbf{A}\in\CV$, and let $P\subseteq A$. The algebraic structure 
$\mathbf{P}$ with the universe $P$ and partial operations defined
by putting
\[
f^\mathbf{P}(\ov{x})=\begin{cases}
  f^\mathbf{A}(\ov{x}) & \text{if } f^\mathbf{A}(\ov{x})\in P\\
  \text{undefined} & \text{if } f^\mathbf{A}(\ov{x})\in P
  \end{cases}
\]
will be called a \emph{partial algebra in} $\CV$. There are other, inequivalent,
ways of defining partial algebras in a variety, but our result does not depend
on which definition we choose. Note that although the class of partial algebras
in $\CV$ is closed under isomorphism (because $\CV$ is), it is customary to call
partial algebras in $\CV$ \emph{partial algebras embeddable in} $\CV$. We will
follow that custom.

\begin{definition}
A variety $\CV$ is said to have the \emph{finite embeddability property\/} (FEP), if
for every finite partial algebra $\mathbf{P}$ embeddable in $\CV$, there exists
a finite algebra $\mathbf{B}\in\CV$ such that $\mathbf{P}$ embeds into $\mathbf{B}$.
\end{definition}

The finite embeddability property was implicitly known in 1940s in the context
of word problems, but it was formally introduced by Evans~\cite{Eva69}, where
it was shown that FEP implies solvability of the word problem. FEP for classes
of residuated structures was investigated in a series of articles around the turn
of the millenium; a representative example is~\cite{BvA05}, where an important general construction was devised. The following well-known result, whose proof can
be found in~\cite{BvA05}, is crucial for our purposes.

\begin{proposition}\label{lem:fep-fin-gen}
For a variety $\CV$ the following are equivalent\textup:
\begin{enumerate}[label={\upshape(\arabic*)},leftmargin=1.75\parindent]
\item $\CV$ has the finite embeddability property\textup;
\item $\CV$ is generated as a quasivariety by its finite members.
\end{enumerate}    
\end{proposition}  

We are now ready to present our construction, which the reader familiar with modal
logic will recognise as a variant of \emph{filtration}.
For the remainder of this subsection,
$\CV$ will be a subvariety of $\mathsf{DH}$ or of $\mathsf{H}^+$. Let
$\mathbf{P}$ be a finite partial algebra embeddable in $\CV$. Further, let
$\mathbf{A}\in\CV$ be an algebra into which $\mathbf{P}$ is embedded, and
let $\X$ be the dual space of $\mathbf{A}$. Then $P$ can be identified with a
finite collection $\mathcal{P}$ of clopen up-sets of $\X$. Define a binary relation $\simeq$ on $X$ by putting
\[
x\simeq y \text{ if } \forall U\in\mathcal{P}\colon x\in U
\Longleftrightarrow y\in U. 
\]
Clearly, $\simeq$ is an equivalence relation. Since $\mathcal{P}$ is finite,
$\simeq$ has finitely many equivalence classes. Let $Y = X/{\simeq}$.
Next, we define a binary relation $\sqsubseteq^Y$ on $Y$ by
$[x]\sqsubseteq^Y[y]$ if $\exists x'\in [x], y'\in [y]\colon x'\leq y'$.
Finally, we define $\leq^Y$ to be the transitive closure of
$\sqsubseteq^Y$. Note that $x\leq y$ implies $[x]\leq^Y[y]$, but not conversely.
Clearly the structure $\Y:=\langle Y;\leq^Y\rangle$ is a finite ordered set,
but in general not a quotient space of $\X$.

Let $\bu(\Y)$ be an algebra of an appropriate signature, whose elements are
the up-sets of $\Y$.
%(Note that $\Y$ is not a quotient space of $\X$).
We will call this algebra the \emph{$\mathbf{P}$-filtrate} of $\mathbf{A}$,
and typically denote it by $\mathbf{A}_{\mathbf{P}}$.
Let $\varphi\colon \mathcal{P}\to \uu(\Y)$
be defined by $\varphi(U) = \{[u]\mid u\in U\}$.
The next lemma shows that $\varphi$ preserves
the structure of~$\mathcal{P}$, and so $\mathbf{P}$ is
isomorphic to a partial subalgebra of
$\mathbf{A}_{\mathbf{P}}$. 

\begin{lemma}\label{lem:preserv}
Let $U$ and $V$ be elements of $\mathcal{P}$, let $x,y \in X$,
and let $\star$ be any operation in   
the set $\{\neg, \dpc, \cap,\cup,\to, \dotdiv\}$. Then the following hold\textup{:}
\begin{enumerate}[label={\upshape(\arabic*)},leftmargin=1.75\parindent]
\item If $[x]\leq^Y[y]$ and $x\in U$, then $y\in U$.  
\item $\varphi(U)$ is an up-set of $Y$.
\item If $U\star V\in \mathcal{P}$, then
$\varphi(U)\star\varphi(V) = \varphi(U\star V)$.
\end{enumerate}
\end{lemma}  

\begin{proof}
For (1), by the definition of $\leq^Y$, we have $[x]\leq^Y [y]$ if and only if
for some elements $z_1,u_1,\dots,z_n,u_n\in X$ we have
$x\simeq z_1\leq u_1\simeq z_2\leq\cdots \simeq  z_n \leq u_n \simeq y$. 
If $x\in U$, then using alternately the definition of $\simeq$ and
the fact that $U$ is an up-set, we get that $y\in U$. Next, (2) is an immediate
consequence of (1). For (3), we will only consider $\ra$ as an example. Assume
$U\ra V\in \mathcal{P}$. By Lemma~\ref{dual-operations}
$U\ra V = X\comp\dw(U\comp V)$,  equivalently,
\[
x\in U\ra V
\quad\text{ if and only if }\quad\forall y\geq x\colon y\in U 
\Longrightarrow  y\in V.
\]
Let $[x]\in \varphi(U)\ra\varphi(V)$. Take any $x'\in[x]$; 
we claim that $x'\in U\ra V$.
Assume that $x'\leq z$ and $z\in U$. Then $[x]\leq^Y[z]$ and
$[z]\in\varphi(U)$, so by assumption $[z]\in\varphi(V)$. Hence, $z\in V$, proving
one inclusion. For the other inclusion, let $[x]\in\varphi(U\ra V)$; take
$[y]\geq^Y[x]$ with $[y]\in \varphi(U)$. Thus $x\in U\ra V$ and
$y\in U$. Since $[x]\leq^Y[y]$, by~(1) we get that  $y\in U\ra V$.
Hence, $y\in V$, showing that $[y]\in\varphi(V)$ as required.
\end{proof}
  
\begin{corollary}\label{PintoA_P}
Let $\mathbf{A}\in\CV$, let $\mathbf{P}$ a finite partial subalgebra of $\mathbf{A}$ and let $\mathbf{A}_{\mathbf{P}}$ be the resulting $\mathbf{P}$-filtrate of $\mathbf{A}$. Then $\mathbf{P}$~is isomorphic to a partial subalgebra of $\mathbf{A}_{\mathbf{P}}$.
\end{corollary}  

Note that $\mathbf{A}_{\mathbf{P}} = \bu(\Y)$ 
is a finite double Heyting algebra, or an
H$^+$-algebra, by construction. Thus the next result follows at once.

\begin{theorem}
The finite embeddability property holds for the varieties
$\mathsf{DH}$ and~$\mathsf{H}^+$ and thus they are generated by their finite members. 
\end{theorem}

To conclude that FEP holds for an arbitrary subvariety $\CV$ of either of $\mathsf{DH}$ or~$\mathsf{H}^+$, we need to make sure that
$\bu(\Y)$ belongs to $\CV$. In general it is not the case,
but if the membership in $\CV$ is determined by some property of $\X$ preserved by
$\Y$, then it is. For example, if $\X$ is a chain then so is $\Y$; hence the
property of being a chain is preserved. In fact any property definable by a 
positive first-order formula in the language of $\leq$ is preserved, since by a
classical model-theoretic preservation result positive formulas are preserved
by homomorpisms. More refined preservation results are often
obtained by first expanding the partial algebra $\mathbf{P}$ somewhat, to include
some desired up-sets of the dual space. Below is an example which will
suffice for our purposes.  

\begin{lemma}\label{lem:height-2}
Let $\mathbf{A}\in\CV$ and let $\mathbf{P}$ be a finite partial subalgebra of
$\mathbf{A}$ that is closed under $\dpc$.  
If $\A$ is an $\mathsf{RDP}$-algebra, then so is the corresponding
$\mathbf{P}$-filtrate $\mathbf{A}_{\mathbf{P}}$ of~$\A$.
\end{lemma}

\begin{proof}
Let $\X$ be the dual space of~$\A$, so by Theorem~\ref{varlet}, every element of $\X$ is either minimal or maximal. Suppose that $\mathbf{A}_{\mathbf{P}}$ is not 
an $\mathsf{RDP}$-algebra. Then, by Theorem~\ref{varlet} again, $[x]<^Y[z]<^Y[y]$ holds for some $x,y,z\in X$.
By construction, there exist clopen up-sets
$U,V$ of $\X$ such that $x\notin U$, $z\in U$, and $z\notin V$, $y\in V$.   
Reasoning as in the proof of Lemma~\ref{lem:preserv}(1), we   
obtain a configuration ${a< u\simeq v< b}$, with $u,v,b\in U$ and
$a\notin U$, so $a\in\dpc U$ and $u\in \dpc U$.
Since $\mathbf{P}$ is closed under~$\dpc$,
we have $\dpc U\in \mathcal{P}$; hence $v\in \dpc U$. 
As $v\in \dpc U$, by Lemma~\ref{dual-operations}(2) there exists $w\leq v$ with $w\notin U$. But $v<b$, so $v$ is a minimal element of $\X$. Hence $v = w \notin U$, contradicting the fact that $v\in U$.
\end{proof}

\begin{theorem}\label{thm:double-p-fep}
The variety of regular double p-algebras has the finite embeddability property 
and thus it is generated by its finite members.
\end{theorem}  

\begin{proof}
Let $\A$ be a regular double p-algebra and let $\mathbf{S}$ be a finite partial subalgebra of~$\A$.
As the equality
$\dpc\dpc\dpc x\approx \dpc x$ holds in $\mathsf{H}^+$, the set
${P := S\cup \{\dpc s, \dpc\dpc s\mid s\in S\}}$ (where $\dpc$ is taken in~$\mathbf{A}$) is closed under $\dpc$, and finite. Let $\mathbf{P}$ be the
partial subalgebra of $\mathbf{A}$ with the universe~$P$, and let
$\mathbf{A}_{\mathbf{P}}$ be the corresponding $\mathbf{P}$-filtrate of $\mathbf{B}$. Then we have $\mathbf{S}\leq \mathbf{P}\leq\mathbf{A}_{\mathbf{P}}$,
 by Corollary~\ref{PintoA_P}. As $\mathbf{A}_{\mathbf{P}}\in\mathsf{RDP}$ by Lemma~\ref{lem:height-2}, we are done.
\end{proof}

%%%%%%%%%%%%%%%%%%%%%%%%%%%%%%%%%%%%%%%%%%%%%%%%%%%%%%%%%%%%%%%%%%%%%%%
\section{Subvarieties of \h{s}}\label{sec:subvarofH+}
%%%%%%%%%%%%%%%%%%%%%%%%%%%%%%%%%%%%%%%%%%%%%%%%%%%%%%%%%%%%%%%%%%%%%%%

\subsection{Small subvarieties}

Let $\mathbf{2}$ and $\mathbf{3}$ denote, respectively, the 
two-element and three-element chains.
Any structure on the chains will be determined by the context.
Since every non-trivial double Heyting algebra has $\{0,1\}$ as a subuniverse, the variety $\Var(\mathbf{2})$ of Boolean algebras is the smallest non-trivial subvariety of $\mathsf{DH}$. The same thing applies to $\mathsf{H^+}$ and $\mathsf{RDP}$.
The next obvious candidate is the three-element chain. 
We will begin by characterising, in terms of the dual space, the double Heyting algebras that have a subalgebra isomorphic to $\mathbf{3}$.
The next lemma shows that it suffices to do so for \h{s}. 

\begin{lemma}\label{remark-subalgebra} 
	If $\mathbf{A}$ is a double Heyting algebra, then $\mathbf{3} \leq \mathbf{A}$ if and only if $\mathbf{3} \leq \mathbf{A}^\flat$.
\end{lemma}
\begin{proof}
	Let $\mathbf{A}$ be a double Heyting algebra.
If $\ub{3}$ embeds into $\mathbf{A}$, then it is obvious that $\ub{3}$ embeds into $\mathbf{A}^\flat$.
For the converse, it is easily checked that, for all $x \in A$, we have $x \dotdiv 1 = 0$, $0 \dotdiv x = 0$, and $x \dotdiv 0 = x$.
Consequently, if $\{0,1,x\}$ is a subuniverse of $\mathbf{A}^\flat$, then it is closed under $\dotdiv$ and therefore it is a subuniverse of $\mathbf{A}$.
\end{proof}

By observing that finite products of \h{s} correspond to disjoint unions of ordered sets in the dual, what follows is a consequence of Corollary~4.5 in~\cite{normalfilters}.
\begin{proposition}\label{si-simple}
Let $\X$ be a finite ordered set. 
Then, as an \h{} or a double Heyting algebra, $\bu{(\X)}$ is simple if and only if $\X$ is connected.
\end{proposition}

\begin{definition}
	Let $\X$ be an ordered set. If $x \in \min(\X) \cap \max(\X)$, then we will call $x$ \emph{order-isolated}.
\end{definition}

Recall that under the duality, if $\mathbf{A}$ and $\mathbf{B}$ are \h{s}, then an embedding $h \colon \mathbf{A} \to \mathbf{B}$ corresponds to a surjective \h[morphism] $\varphi \colon \cat F_{\!p}(\mathbf{B}) \to \cat F_{\!p}(\mathbf{A})$.

\begin{proposition}\label{3-chain}
	Let $\X$ be an \h[space]{}. Then there exists a surjective \h[morphism]{} $\varphi\colon \X \to {\mathbf{2}}$ if and only if $\X$ has no order-isolated elements. 
\end{proposition}
\begin{proof}
If $\max(\X) \cap \min(\X) = \varnothing$, then since $\min(\X)$ and $\max(\X)$ are closed subsets of $\X$, there exists a clopen up-set $U$ such that $\max(\X) \subseteq U$ and $\min(\X) \cap U = \varnothing$. It is then easily verified that the set $\{\varnothing, U, X\}$ is a subalgebra of $\but(\X)$ isomorphic to~$\mathbf{3}$. Conversely, let $x \in X$, assume that $x$ is order-isolated, and let $\varphi$ be a morphism on $\X$. By Lemma~\ref{minimals} it follows that $\varphi(x)$ is both minimal and maximal. Since ${\mathbf{2}}$ has no such elements, the codomain of $\varphi$ cannot be ${\mathbf{2}}$.
\end{proof}

\begin{corollary}
Let $\X$ be a double Heyting space. Then there exists a surjective double Heyting morphism $\varphi\colon \X \to {\mathbf{2}}$ if and only if $\X$ has no order-isolated elements. 
\end{corollary}

This is not enough to show that every non-trivial and non-Boolean subvariety of \h{s} contains the three-element chain.
The \h[space]{} depicted in Figure~\ref{fig:counterexample} is the dual of a subdirectly irreducible \h{} (its congruence lattice is a three-element chain) and it has an order-isolated element, so the algebra has no subalgebra isomorphic to $\mathbf{3}$. 
Yet, as we will see shortly, the variety it generates contains $\mathbf{3}$. 
On the other hand, it is true that $\mathbf{3}$ embeds into every \emph{finite} non-Boolean subdirectly irreducible \h{}.
Indeed, by Proposition~\ref{useful}, if $\bu(\X)$ is a finite subdirectly irreducible \h{}, then it is simple.
Then $\X$ is connected by Proposition~\ref{si-simple}, so it cannot have any order-isolated elements unless $|X| = 1$.

\begin{corollary}
	If $\mathbf{A}$ is a finite non-Boolean subdirectly irreducible double Heyting algebra or \h{}, then $\mathbf{3} \leq \mathbf{A}$.
\end{corollary}

\begin{figure}[ht]
\begin{tikzpicture}
\foreach \x[evaluate=\x as \y using {3*\x/28},evaluate=\x as \yi using {1.5-3*\x/28},evaluate=\x as \i using {int(2*\x)}, evaluate=\x as \j using {int(2*\x+1)}] in {0,...,4} {
\draw (\x,\y) node[fence node] (x\x) {};
\draw (\x,\yi) node[fence node] (y\x) {};
\draw (x\x) -- (y\x);
}
\draw \foreach \x [remember=\x as \lastx (initially 0)] in {1,...,4}{(y\lastx) -- (x\x)};
\draw (7,0.75) node[fence node] (limit) {};
\coordinate (temp) at ($(5, {15/28})$);
\draw (y4) -- ($(y4)!0.4!(temp)$) ;
\draw[loosely dotted] (y4) -- (limit) -- (x4);
\end{tikzpicture}
\caption{}\label{fig:counterexample}
\end{figure}

To prove that every non-trivial and non-Boolean subvariety of double Heyting algebras contains the three-element chain, the next lemma will be useful. 
For convenience, let $\sigma x = \dpc\neg x$. 

\begin{lemma}\label{dq-counting}
Let $\X$ be an \h[space]{} and let $U$ be a clopen up-set in $\X$. If $U \neq \varnothing$, then $\delta^n\sigma^{n+1} U \neq \varnothing$, for all $n \in \omega$.
\end{lemma}
\begin{proof}
Assume that $U \neq \varnothing$ and suppose that $\delta^n\sigma^{n+1}U = \varnothing$. 
This means that $(\down\up)^n(X \comp (\up\down)^{n+1}U) = X$. 
Then, for each $u \in U$, there exists $y \in X\comp(\up\down)^{n+1}U$ such that $u \in (\down\up)^n y$.
But then $y \in (\down\up)^nu \subseteq \up(\down\up)^n\down u \subseteq (\up\down)^{n+1}U$, a contradiction. 
\end{proof}

\begin{theorem}\label{3-subalgebra2}
	Let $\A$ be an \h{}.
	If $\A$ is not Boolean, then $\mathbf{3} \in \Var(\A)$.
	More precisely, if $\A$ is non-Boolean and subdirectly irreducible, then
	there exists a congruence 
	${\alpha \in \con(\mathbf{A}^\omega)}$ such that 
	$\mathbf{3} \leq \mathbf{A}^\omega/\alpha$. 
\end{theorem}
\begin{proof}
Let $\X$ be the Priestley dual of $\mathbf{A}$ and assume that $\mathbf{A}$ is non-Boolean and subdirectly irreducible.
If $X$ has no order-isolated elements, then we are covered by Proposition~\ref{3-chain}.
So, assume that $X$ has at least one order-isolated element.
Recall that $\min_\X(U) = \min(\X) \cap U$ and $\max_\X(U) = \max(\X) \cap U$, for all $U \subseteq X$.
If $X = \min(\X)$, then $\mathbf{A}$ is Boolean.
So $X \comp {\min(\X)}$ is non-empty, and since $\min(\X)$ is closed, there exists a non-empty clopen up-set $U \subseteq X$ such that $\min_\X(U) = \varnothing$.
Then $U$ cannot contain any order-isolated elements.
But $\X$ does, so we must have $\sigma^i U = (\up\down)^i U \neq X$, for all $i \in \omega$.
Additionally, if there exists $i \in \omega$ such that $\sigma^i U = \sigma^{i+1}U$, then $\sigma^iU$ is complemented, and in a subdirectly irreducible \h{} this only occurs if $\sigma^iU = \varnothing$ or $\sigma^iU = X$.
We have already seen that $\sigma^i U \neq X$, for all $i \in \omega$. 
Moreover, we have $\sigma U = 0$ if and only if $U = 0$, and so, by induction, the former case does not occur either. 
Therefore, the members of $\langle \sigma^iU\rangle_{i \in \omega}$ are pairwise distinct. Let $U_i = \sigma^iU$. 

Since $\max(\X)$ is closed, $\max_\X(U_i)$ is also closed.
Hence, for each $i \in \omega$, there is a non-empty clopen up-set $V_i$ such that $\max_\X(U_i) \subseteq V_i$ and $\min_\X(V_i) = \varnothing$.
Let $M_i = V_i \cap U_i$, and observe that $\max_\X(M_i) = \max_\X(U_i)$ and $\min_\X(M_i) = \varnothing$.
Because they share their maximal elements, we have $\down M_i = \down U_i$ and it follows that $\neg M_i = \neg U_i$.
Moreover, since $\min_\X(M_i) = \varnothing$, we have $\up(X \comp M_i) = X$ and therefore $\dpc M_i = X$. 

Now let $\mathbf{H} = \mathbf{A}^\omega$.
Denote the tuple $\langle{M_i}\rangle_{i \in \omega}$ by $M$, and let $\alpha$ be
the congruence 
\[
\alpha = \cg^\mathbf{H}(\neg M, 0).
\]
In any \h{}, $\neg x = 1$ if and only if $x = 0$, so we then have 
\[
\alpha = \cg^\mathbf{H}(\neg\neg M, 1) = \cg^\mathbf{H}(\neg\neg\langle{U_i}\rangle_{i \in \omega}, 1).
\]
To see that $\alpha$ is not the full congruence on $\mathbf{H}$, we will suppose that it is.
Then there exists $n \in \omega$ such that $\delta^n\neg\neg\langle{U_i}\rangle_{i \in \omega} = 0$. 
Since $\delta$ is order-preserving and $\neg\neg x \geq x$, we have $\delta^n\langle{U_i}\rangle_{i \in \omega} = 0$.
In other words, for each $i \in \omega$, we have $\delta^nU_i = \delta^n\sigma^i U = \varnothing$.
But by Lemma~\ref{dq-counting} this is impossible.
Hence, $\mathbf{H}/\alpha$ is a non-trivial algebra. 

We finish the proof by showing that 
${\mathbf{3}}$ embeds into $\mathbf{H}/\alpha$.
Since $\dpc M_i = X$, for all $i \in \omega$, it follows that $\dpc M = 1$ in $\mathbf{H}$, so $\dpc M/\alpha = 1/\alpha$.
By definition of $\alpha$, we have $\neg M/\alpha = 0/\alpha$.
These two facts combined with the fact that $\mathbf{H}/\alpha$ is non-trivial imply that $M/\alpha \notin \{0/\alpha, 1/\alpha\}$.
We thus conclude that $\{0/\alpha, M/\alpha, 1/\alpha\}$ is the underlying set of a subalgebra of $\mathbf{H}/\alpha$ isomorphic to ${\mathbf{3}}$. 
\end{proof}

A similar argument also applies to double Heyting algebras, but assuming ignorance of the proof, we can still prove the analogous result as a direct corollary.
Let $\mathbf{A}$ be a non-Boolean subdirectly irreducible double Heyting algebra.
By the previous result, there exists a congruence $\alpha$ on $(\mathbf{A}^\flat)^\omega$ such that $\mathbf{3} \leq (\mathbf{A}^\flat)^\omega/\alpha$.
But since the operations $\to$ and $\dotdiv$ depend only on the underlying lattice, it follows that $(\mathbf{A}^\flat)^\omega = (\mathbf{A}^\omega)^\flat$.
By Theorem~\ref{congruence-isomorphism}, we have $\con(\mathbf{A}^\omega) = \con((\mathbf{A}^\omega)^\flat)$, so $\alpha$ is a congruence on $\mathbf{A}^\omega$.
But we also have $(\mathbf{A}^\omega/\alpha)^\flat = (\mathbf{A}^\omega)^\flat/\alpha = (\mathbf{A}^\flat)^\omega/\alpha$.
So, by Lemma~\ref{remark-subalgebra}, it follows that $\mathbf{3} \leq \mathbf{A}^\omega/\alpha$, as claimed.
The next two results follow by observing that by the previous result, the only subvarieties not containing $\mathbf{3}$ are the trivial subvariety and the variety of Boolean algebras.

\begin{corollary}[Wolter~\cite{wolter}]
In $\cat{L}(\mathsf{DH})$, the variety $\Var(\mathbf{3})$ is completely join-irreducible and covers the variety $\Var(\mathbf{2})$. 	Hence, $\mathbf{3}$ is a splitting algebra in $\mathsf{DH}$.
\end{corollary}

\begin{corollary}
In $\cat{L}(\mathsf{H}^+)$,	the variety $\Var(\mathbf{3})$ is completely join-irreducible and covers the variety $\Var(\mathbf{2})$. Hence, $\mathbf{3}$ is a splitting algebra in $\mathsf{H}^+$. 
\end{corollary}

\begin{corollary}
In $\cat{L}(\mathsf{RDP})$,	the variety $\Var(\mathbf{3})$ is completely join-irreducible and covers the variety $\Var(\mathbf{2})$. Hence, $\mathbf{3}$ is a splitting algebra in $\mathsf{RDP}$.
\end{corollary}

\subsection{Fences and double-pointed ordered sets}
In this subsection, again, $\CV$ will be a subvariety of $\mathsf{H}^+$, or of
$\mathsf{DH}$, generated by its finite members.
This includes $\mathsf{RDP}$ as a special case.
All candidates for splitting algebras in $\CV$ are finite and subdirectly irreducible, and so,
by Proposition~\ref{useful}(4), simple.
By Proposition~\ref{si-simple}, the dual spaces of finite simple algebras are finite connected ordered sets.
Thus, henceforth, we will focus on these. 

If $\X$ is a Priestley space and $\varphi$ is a constant \h[morphism]{} on $\X$, then under the duality, $\varphi$ corresponds to the two-element Boolean subalgebra of $\but(\X)$.
Since $\mathbf{2}$ embeds into every \h{}, to avoid an overload of exemptions, we will disregard it in most of what follows.
Using Lemma~\ref{minimals}, the following result is easy to prove.
\begin{lemma}\label{non-constant}
	Let $\X$ be a finite connected ordered set and let $\varphi$ be an \h[morphism]{} on $\X$.
	Then $\varphi$ is a constant map if and only if, for every $m_1 \in \max(\X)$ and every $m_2 \in \min(\X)\comp{\max(\X)}$, we have $\varphi(m_1) \neq \varphi(m_2)$.
\end{lemma}

\begin{definition}
A non-trivial finite ordered set $\X$ is a \emph{fence} if there is an enumeration $x_1, \dots, x_n$ of elements of $X$, where $n = |X|$, such that the only order relations on $\X$ are given by one of the following:
\begin{enumerate}
\item $x_1 < x_2 > x_3 < \dots > x_{n-1} < x_n$, \label{fence:1} 
\item $x_1 < x_2 > x_3 < \dots < x_{n-1} > x_n$, or \label{fence:2}
\item $x_1 > x_2 < x_3 > \dots > x_{n-1} < x_n$. \label{fence:3}
\end{enumerate}
Examples of fences of each type are given in Figure~\ref{fence-figure}. 
We will permit the two-element fence under this definition, which is covered by 
each of \eqref{fence:1}, \eqref{fence:2}, and \eqref{fence:3}. Note that, by
assumption, a fence is finite and has at least two elements.
\end{definition}
\begin{figure}[ht] 
\centering
\begin{subfigure}{0.25\textwidth}\centering
\begin{tikzpicture}
\tikzfence{6}
\end{tikzpicture}
\caption{}\label{fig:fence:1}
\end{subfigure}
\quad
\begin{subfigure}{0.25\textwidth}\centering
\begin{tikzpicture}
\tikzfence{7}
\end{tikzpicture}
\caption{}\label{fig:fence:2}
\end{subfigure}
\quad
\begin{subfigure}{0.25\textwidth}\centering
\begin{tikzpicture}
\tikzfencedual{7}
\end{tikzpicture}
\caption{}\label{fig:fence:3}
\end{subfigure}
\caption{The fences (\subref{fig:fence:1}), (\subref{fig:fence:2}), and (\subref{fig:fence:3}) are of type \eqref{fence:1}, \eqref{fence:2}, and \eqref{fence:3} respectively.}\label{fence-figure}
\end{figure}
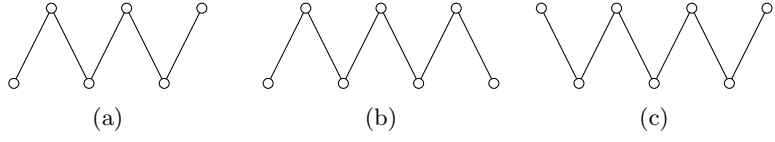

In the study and application of finite ordered sets, fences are of the utmost importance;
however, this definition of a fence is not particularly user friendly, so we will give a characterisation that is more suited to the current setting.

\begin{definition}Let $\X$ be an ordered set and let $\tau_1, \tau_2 \in X$ with $\tau_1 \neq \tau_2$.
We will say that the pair $(\tau_1,\tau_2)$ is an \emph{up-tail} if $\tau_1$ is maximal  and $\down \tau_1 = \{\tau_1,\tau_2\}$.
Dually, $(\tau_1,\tau_2)$ is a \emph{down-tail} if $\tau_1$ is minimal and $\up \tau_1 = \{\tau_1,\tau_2\}$.
In either case we will say that the pair $(\tau_1,\tau_2)$ is a \emph{tail} and that $\X$ \emph{has a tail}.
\end{definition}

\begin{figure}[ht]
\begin{subfigure}{0.45\textwidth}\centering
	\begin{tikzpicture}[scale=0.7,nodes={draw,circle,fill=white}]
\draw (0,0) circle [rotate=60, x radius = 1.5, y radius = 0.8] (0,0);
\node[scale=0.4] at (-1.8,-1.22) (y) {};
\node[scale=0.4] at ($(y)+(135:1 and 1)$) (x) {};
\coordinate (first) at ($(y)+(40:1.3 and 1.3)$);
\coordinate (second) at ($(y)+(30:1.25 and 1.25)$);
\coordinate (third) at ($(y)+(20:1.2 and 1.2)$);
\draw (y) -- (first);
\draw (y) -- (second);
\draw (y) -- (third);
\draw (x) -- (y);

\node [inner sep=1,text height=1.5ex,text depth=.25ex,draw=none,above=0.32 of x] (x_label) {$\tau_1$};
\draw [densely dotted,very thin] (x) -- (x_label);
\node [inner sep=1,text height=1.5ex,text depth=.25ex,draw=none] at (y |- x_label) (y_label) {$\tau_2$};
\draw [densely dotted,very thin] (y) -- (y_label);
\end{tikzpicture}
\caption{}\label{tail:a}
\end{subfigure}
\begin{subfigure}{0.45\textwidth}\centering
	\begin{tikzpicture}[scale=0.7,yscale=-1,nodes={draw,circle,fill=white}]
\draw (0,0) circle [rotate=60, x radius = 1.5, y radius = 0.8] (0,0);
\node[scale=0.4] at (-1.8,-1.22) (y) {};
\node[scale=0.4] at ($(y)+(135:1 and 1)$) (x) {};
\coordinate (first) at ($(y)+(40:1.3 and 1.3)$);
\coordinate (second) at ($(y)+(30:1.25 and 1.25)$);
\coordinate (third) at ($(y)+(20:1.2 and 1.2)$);
\draw (y) -- (first);
\draw (y) -- (second);
\draw (y) -- (third);
\draw (x) -- (y);

\node [inner sep=1,text height=1.5ex,text depth=.25ex,draw=none,below=0.32 of x] (x_label) {$\tau_1$};
\draw [densely dotted,very thin] (x) -- (x_label);
\node [inner sep=1,text height=1.5ex,text depth=.25ex,draw=none] at (y |- x_label) (y_label) {$\tau_2$};
\draw [densely dotted,very thin] (y) -- (y_label);
\end{tikzpicture}
\caption{}\label{tail:b}
\end{subfigure}
\caption{In (\subref{tail:a}), the pair $(\tau_1,\tau_2)$ is an up-tail, and in (\subref{tail:b}), the pair $(\tau_1,\tau_2)$ is a down-tail.}
\end{figure}
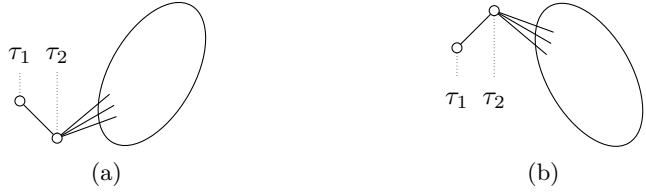

Observe that a tail $(\tau_1,\tau_2)$ is both a down-tail and an up-tail if and only if $\updown\{\tau_1,\tau_2\} = \{\tau_1,\tau_2\}$.
Also note that if $(\tau_1,\tau_2)$ is a down-tail, then $\tau_2$ must be minimal, and if it is an up-tail, then $\tau_2$ must be maximal.

\begin{lemma}\label{tails-to-tails}
Let $\X$ be a connected ordered set, let $x,y \in X$, and let $\varphi$ be a non-constant H\textsuperscript{+}-morphism on~$\X$.
If $(\tau_1,\tau_2)$ is a down-tail in $\X$, then $(\varphi(\tau_1),\varphi(\tau_2))$ is a down-tail in $\langle \varphi(X); \leq_{\varphi(X)}\rangle$.
\end{lemma}
\begin{proof}
Assume that $(\tau_1,\tau_2)$ is a down-tail.
Then $\up\varphi(\tau_1) = \varphi(\up\tau_1) = \{\varphi(\tau_1),\varphi(\tau_2)\}$.
By Lemma~\ref{minimals}, since $\tau_1$ is minimal, $\varphi(\tau_1)$ is also minimal.
Because $\varphi$ is non-constant, we have $\varphi(\tau_1) \neq \varphi(\tau_2)$ by Lemma~\ref{non-constant}, so $(\varphi(\tau_1), \varphi(\tau_2))$ is a down-tail.
\end{proof}

It is false that H\textsuperscript{+}-morphisms must preserve up-tails, although a dual argument to the one above shows that double Heyting morphisms do.
This marks a notable distinction between the two types of morphism, and mildly complicates some of the proofs that follow.

\begin{lemma}\label{fence-characterisation}
Let $\X$ be a non-trivial finite connected ordered set.
The following are equivalent\textup:
\begin{enumerate}[label={\upshape(\arabic*)},leftmargin=1.75\parindent]
\item $\X$ is a fence\textup; \label{fence-characterisation:1}
\item $|\up x| \leq 3$ and $|\down x| \leq 3$, for all $x \in X$, and if $|X| > 2$, then $\X$ has two tails\textup; \label{fence-characterisation:2}
\item $|\up x| \leq 3$ and $|\down x| \leq 3$, for all $x \in X$, and $\X$ has at least one tail. \label{fence-characterisation:3}
\end{enumerate}
\end{lemma}
\begin{proof}
\ref{fence-characterisation:1} $\Rightarrow$ \ref{fence-characterisation:2} $\Rightarrow$ \ref{fence-characterisation:3} is obvious. 
As for \ref{fence-characterisation:3} $\Rightarrow$ \ref{fence-characterisation:1}, we proceed by induction. 
If $|X| \in \{2,3\}$, then the implication is obvious. 
So, assume that $|X| > 3$, that \ref{fence-characterisation:3} holds for $\X$, and that the characterisation holds for all fences of a smaller size than~$\X$. 
Let $(x,y)$ be a tail in $\X$. 
Assume first that $(x,y)$ is a down-tail. 
Then $x$ is minimal and $y$ is maximal.
Since $\X$ is connected and $|X| > 3$, there must be some $z \in X$ with $x \neq z$ such that $z \in \down y$.
Since $|\down y| \leq 3$, we have that $z$ is minimal and $\down y = \{x,y,z\}$.
Now consider the ordered set $Y = X \comp \{x\}$, with the order inherited from $\X$.
Then, in $\Y$, we have $\down y = \{y,z\}$, so $(y,z)$ is an up-tail in~$\Y$.
Clearly all of the conditions in \ref{fence-characterisation:3} hold for $\Y$.
Thus, $\Y$ is a fence, where the description of the order is of the form $\cdots< w > z < y$. 
Hence the order on $\X$ is of the form $\cdots < w > z < y > x$, and we conclude that $\X$ is a fence. 
A similar argument holds if we had instead assumed $(x,y)$ to be an up-tail.
\end{proof}

\begin{remark}\label{fence-algebra}
	Since every element of a fence is minimal or maximal, if $\X$ is a fence, then by Theorem~\ref{varlet}, the lattice $\bu(\X)$ underlies a regular double p-algebra.
	Then by Theorem~\ref{regular-implies-heyting}, up to term-equivalence, there is no difference between treating $\bu(\X)$ as a double p-algebra, an \h{}, or a double Heyting algebra.
\end{remark}

The proof of the next lemma illustrates the utility of Lemma~\ref{fence-characterisation}.

\begin{proposition}\label{image-of-fence-baby}
Let $\F= \langle F; \leq\rangle$ be a fence and let $\varphi$ be a non-constant H\textsuperscript{+}-morphism on~$\F$.
Then $\varphi(\F) = \langle \varphi(F); \leq_{\varphi(F)}\rangle $ is also a fence.
\end{proposition}
\begin{proof}
Observe by Remark~\ref{fence-algebra} that an H\textsuperscript{+}-morphism on $\F$ is a double Heyting morphism as well.
So the dual of Lemma~\ref{tails-to-tails} applies.
As $\F$ is connected, the image $\varphi(\F)$ is also connected, and by using Lemma~\ref{tails-to-tails} and its dual we see that $\varphi(\F)$ has at least one tail.
For all $x \in F$, we have $|\up x| \leq 3$, so $|\up\varphi(x)| = |\varphi(\up x)| \leq 3$.
Similarly, we have $|\down\varphi(x)| \leq 3$.
Thus, by Lemma~\ref{fence-characterisation}, $\varphi(\F)$ is a fence.
\end{proof}

\begin{definition}
A structure $\mathbf{S} = \langle S; \bot^\mathbf{S}, \top^\mathbf{S}, \leq^\mathbf{S}\rangle$ is a \emph{double-pointed ordered set} if the reduct $\ov{\mathbf{S}} := \langle{S; \leq}\rangle$ is a finite ordered set, $\bot^\mathbf{S}$ and $\top^\mathbf{S}$ are nullary operations such that 
$\bot^\mathbf{S} \neq \top^\mathbf{S}$, and $\bot^\mathbf{S}$ is minimal and $\top^\mathbf{S}$ is maximal. 
\end{definition}

Note that, by definition, a double-pointed ordered set has at least two elements.
The constraint that $\bot^\mathbf{S}$ is minimal and $\top^\mathbf{S}$ is maximal is somewhat artificial, but we justify it for a few reasons.
Although we can generalise some of the machinery below, the result we apply in the next subsection, namely Corollary~\ref{main-corollary}, is false if $\top^\mathbf{S}$ is left arbitrary.
We will also apply the results only with both $\bot^\mathbf{S}$ minimal and $\top^\mathbf{S}$ maximal.
Lastly, removing these constraints on $\bot^\mathbf{S}$ and $\top^\mathbf{S}$ produces somewhat more cluttered proofs with no proportional increase in enlightenment.
The next definition is essential to the ``expand-and-distort'' construction.

\begin{definition}\label{arrow-definition}
Let $\mathbf{S}$ and $\mathbf{T}$ be double-pointed ordered sets and assume that 
$S \cap T = \varnothing$.
Then  $\mathbf{S} \arrow \mathbf{T} = \langle{S \cup T; \bot^{\mathbf{S \arrow T}}, \top^{\mathbf{S \arrow T}}, \leq^{\mathbf{S \arrow T}}}\rangle$ is the double-pointed ordered set defined by 
\begin{enumerate}
\item ${\leq^{\mathbf{S \arrow T}}} = {\leq^{\mathbf{S}}} \cup {\leq^{\mathbf{T}}} \cup \{(\bot^\mathbf{T}, \top^\mathbf{S})\}$,
\item $\bot^{\mathbf{S}\arrow \mathbf{T}} = \bot^\mathbf{S}$, 
\item $\top^{\mathbf{S} \arrow \mathbf{T}} = \top^\mathbf{T}$.
\end{enumerate}
Figure~\ref{arrow-figure} illustrates the construction.
It is easy to verify that $\mathbf{S} \arrow \mathbf{T}$ is an ordered set and that $\arrow$ is associative.
To avoid excessive formality, we will always assume that different objects have disjoint underlying sets.
\end{definition}

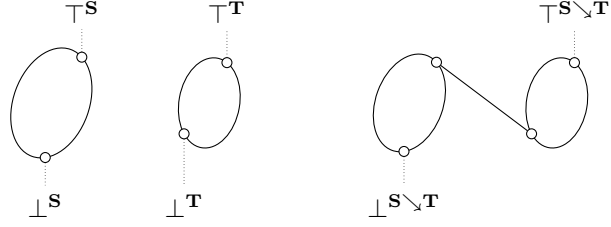
\begin{figure}[ht]
\centering
\begin{tikzpicture}[nodes={draw,circle,fill=white},scale=0.7]
\begin{scope}[scale=0.6,rotate=-20,xscale=1.2,yscale=1.8]
\draw (0,0) circle [x radius = 1, y radius = 1] (0,0);
\coordinate (beta-1) at ($(0,0)+(40:1)$);
\coordinate (beta-2) at ($(0,0)+(70:1)$);
\coordinate (alpha) at ($(0,0)+(210:0.5)$);
\coordinate (alpha-1) at ($(0,0)+(230:1)$);
\coordinate (alpha-2) at ($(0,0)+(290:1)$);
\node[fence node] at (alpha-2) (alpha-drawa) {};
\node[fence node] at (beta-2) (beta-drawa) {};
\end{scope}

\begin{scope}[shift={(3,0)},scale=0.8]
\begin{scope}[scale=0.6,rotate=-10,xscale=1.2,yscale=1.8]
\draw (0,0) circle [x radius = 1, y radius = 1] (0,0);
\coordinate (beta) at ($(0,0)+(90:0.6)$);
\coordinate (beta-1) at ($(0,0)+(70:1)$);
\coordinate (beta-2) at ($(0,0)+(115:1)$);
\coordinate (alpha) at ($(0,0)+(260:0.5)$);
\coordinate (alpha-1) at ($(0,0)+(230:1)$);
\coordinate (alpha-2) at ($(0,0)+(300:1)$);
\node[fence node] at (alpha-1) (alpha-drawb) {};
\node[fence node] at (beta-1) (beta-drawb) {};

\end{scope}
\end{scope}

\begin{scope}[scale=0.8,shift={(8.5,0)}]
\begin{scope}[rotate=-20,xscale=0.8,yscale=1.2]
\draw (0,0) circle [x radius = 1, y radius = 1] (0,0);
\coordinate (beta-1) at ($(0,0)+(40:1)$);
\coordinate (beta-2) at ($(0,0)+(70:1)$);
\coordinate (alpha) at ($(0,0)+(210:0.5)$);
\coordinate (alpha-1) at ($(0,0)+(230:1)$);
\coordinate (alpha-2) at ($(0,0)+(290:1)$);
\node[fence node] at (alpha-2) (alpha-drawc) {};
\node[fence node] at (beta-2) (beta-drawc) {};
\end{scope}

\begin{scope}[shift={(3.5,0)},scale=0.6]
\begin{scope}[rotate=-10,xscale=1.2,yscale=1.8]
\draw (0,0) circle [x radius = 1, y radius = 1] (0,0);
\coordinate (beta) at ($(0,0)+(90:0.6)$);
\coordinate (beta-1) at ($(0,0)+(70:1)$);
\coordinate (beta-2) at ($(0,0)+(115:1)$);
\coordinate (alpha) at ($(0,0)+(260:0.5)$);
\coordinate (alpha-1) at ($(0,0)+(230:1)$);
\coordinate (alpha-2) at ($(0,0)+(300:1)$);
\node[fence node] at (alpha-1) (alpha-drawd) {};
\node[fence node] at (beta-1) (beta-drawd) {};
\end{scope}
\end{scope}
\end{scope}

\node[rectangle,draw=none,fill=none,below=0.3 of alpha-drawa] (label-1) {$\bot^\mathbf{S}$};
\node[rectangle,draw=none,fill=none] at (alpha-drawb |- label-1) (label-2) {$\bot^\mathbf{T}$};
\node[rectangle,draw=none,fill=none] at (alpha-drawc |- label-1) (label-3) {$\bot^\mathbf{S\arrow T}$};

\draw [densely dotted,very thin] (alpha-drawa) -- (label-1);
\draw [densely dotted,very thin] (alpha-drawb) -- (label-2);
\draw [densely dotted,very thin] (alpha-drawc) -- (label-3);

\node[rectangle,inner sep=1,text height=1.5ex,text depth=.25ex,draw=none,fill=none,above=0.3 of beta-drawa] (label-a) {$\top^\mathbf{S}$};
\node[rectangle,inner sep=1,text height=1.5ex,text depth=.25ex,draw=none,fill=none] at (beta-drawb |- label-a) (label-b) {$\top^\mathbf{T}$};
\node[rectangle,inner sep=1,text height=1.5ex,text depth=.25ex,draw=none,fill=none] at (beta-drawd |- label-a) (label-c) {$\top^\mathbf{S\arrow T}$};
\draw (beta-drawc) -- (alpha-drawd);

\draw [densely dotted,very thin] (beta-drawa) -- (label-a);
\draw [densely dotted,very thin] (beta-drawb) -- (label-b);
\draw [densely dotted,very thin] (beta-drawd) -- (label-c);
\end{tikzpicture}

\caption{The double-pointed ordered sets $\mathbf{S}$ and $\mathbf{T}$ are on the left, and $\mathbf{S} \arrow \mathbf{T}$ is on the right.}\label{arrow-figure}
\end{figure}

Henceforth, if $\varphi$ is a morphism on an ordered set $\X$ and $S \subseteq X$, then $\mathbf{S} := \langle S; \le_S\rangle$ is the induced sub-ordered set and $\varphi(\Sb) := \langle \varphi(S); \leq_{\varphi(S)}\rangle $.
The next few lemmas will exhibit properties of morphisms on ordered sets of the form $\ov{\mathbf{S} \arrow \mathbf{T}}$.
\begin{lemma}\label{image-of-arrow}
Let $\mathbf{S}$ and $\mathbf{T}$ be double-pointed ordered sets, assume $\mathbf{S}$ is connected, and let $\varphi$ be an H\textsuperscript{+}-morphism on $\ov{\mathbf{S} \arrow \mathbf{T}}$.
If $\varphi(\top^\mathbf{S}) \in \varphi(T)$, then $\varphi(\ov{\mathbf{S} \arrow \mathbf{T}}) = \varphi(\ov{\mathbf{T}})$. 
\end{lemma}
\begin{proof}
	Assume that $\varphi(\top^\mathbf{S}) \in \varphi(T)$ and let $x \in S$.
By the connectedness of $\mathbf{S}$, every element of $S$ is in the set $\updown^n\top^\mathbf{S}$, for some $n \in \omega$.
	We will prove that $\varphi(x) \in \varphi(T)$ implies $\varphi(\updown x) \subseteq \varphi(T)$. 
The result will then follow by induction, as $\varphi(\top^\mathbf{S}) \in \varphi(T)$.
Let $y \in \updown x$ and assume that $\varphi(x) \in \varphi(T)$.
Then there is some $t \in T$ such that $\varphi(x) = \varphi(t)$.
If $y \geq x$, then 
\[
\varphi(y) \in \varphi(\up x) = \up\varphi(x) = \up\varphi(t) = \varphi(\up t) \subseteq \varphi(T \cup \{\top^\mathbf{S}\}),
\]
which is a subset of $\varphi(T)$ by assumption.
If $y \leq x$, then there is some minimal element $w \leq y$, and then, with $\mathbf{X} = \mathbf{S \arrow T}$ and $\Y = \operatorname{codom}(\varphi)$,
\[
\varphi(w) \in \varphi(\min\nolimits_\X(\down x)) = \min\nolimits_\Y(\down\varphi(x)) = \min\nolimits_\Y(\down\varphi(t)) = \varphi(\min\nolimits_\X(\down t)).
\]
 So there is some $s \leq t$ such that $\varphi(w) = \varphi(s)$.
Then, since $y \geq w$, we have 
 \[\varphi(y) \in \up\varphi(w) = \up\varphi(s) = \varphi(\up s) \subseteq \varphi(T \cup \{\top^\mathbf{S}\}) \subseteq \varphi(T),\]
  as required.
\end{proof}

\begin{lemma}\label{up-tails}
	Let $\mathbf{S}$ and $\mathbf{T}$ be connected double-pointed ordered sets and let $\varphi$ be a non-constant \h[morphism]{} on $\ov{\mathbf{S} \arrow \mathbf{T}}$. 
	Assume that every element of $\mathbf{T}$ is minimal or maximal.
Then, for all $t \in T$, we have $\varphi(\down t) = \down\varphi(t)$.
It follows that if $(x,y)$ is an up-tail in $\mathbf{T}$, then $(\varphi(x),\varphi(y))$ is an up-tail in $\varphi(\ov{\mathbf{S} \arrow \mathbf{T}})$.
\end{lemma}
\begin{proof}
	Let $\mathbf{X} = \mathbf{S} \arrow \mathbf{T}$, let $\Y = \codom(\varphi)$, and let $t \in T$.
If $t$ is minimal, then $\varphi(t)$ is minimal, and the result holds trivially in that case.
Assume that $t$ is maximal.
Then $\varphi(t)$ is maximal.
Let $x \in X$ and assume $\varphi(x) \leq \varphi(t)$.
Then there is some element $y \in X$ such that $\varphi(y)$ is minimal and $\varphi(y) \leq \varphi(x) \leq \varphi(t)$, implying $\varphi(y) \in \min_\Y(\down\varphi(t)) = \varphi(\min_\X(\down t))$.
Therefore, there exists $w \in \min_\X(\down t)$ such that $\varphi(y) = \varphi(w)$.
So $\varphi(x) \in \up\varphi(w) = \varphi(\up w)$, and since ${\up w \subseteq T \cup \{\top^\mathbf{S}\}}$, we must have that $\varphi(x)$ is minimal or maximal by assumption.
Since $\varphi(t)$ is maximal and $\varphi(w) \leq \varphi(x) \leq \varphi(t)$, we conclude that ${\varphi(x) \in \{\varphi(w),\varphi(t)\} \subseteq \varphi(\down t)}$.
It follows that $\down\varphi(t) \subseteq \varphi(\down t)$, and the reverse inclusion holds because $\varphi$ is order-preserving.
To see that $\varphi$ preserves up-tails in~$\mathbf{T}$, use the dual of Lemma~\ref{tails-to-tails}.
\end{proof}

\begin{lemma}\label{image-of-fence}
Let $\mathbf{S}$ be a connected double-pointed ordered set, let $\mathbf{F}$ be a 
double-pointed fence, and let $\varphi$ be a non-constant H\textsuperscript{+}-morphism on $\ov{\mathbf{S} \arrow \mathbf{F}}$. Then $\varphi(\ov{\F})$ is a fence. 
\end{lemma}
\begin{proof}
We will use the characterisation of fences in Lemma~\ref{fence-characterisation}.
Since $\F$ is connected, so is $\varphi(\F)$.
If $|F| = 2$, then because $\varphi$ is non-constant, it follows that $\varphi(\F)$ is a connected ordered set with 2 elements, implying it is a two-element fence.
If $|F| > 2$, then it is easy to see that $\mathbf{S} \arrow \mathbf{F}$ contains at least one tail. Specifically, the two tails of $\F$ are tails in $\mathbf{S} \arrow \mathbf{F}$, unless $\bot^\mathbf{F}$ is the lower element of a down-tail, in which case the other tail in $\mathbf{F}$ is a tail in $\mathbf{S} \arrow \mathbf{F}$.
Then, either by using Lemma~\ref{tails-to-tails} or Lemma~\ref{up-tails}, there is at least one tail in $\varphi(\F)$.
It only remains to check that $|\varphi(F) \cap \down \varphi(x)| \leq 3$ and $|\varphi(F) \cap \up\varphi(x)| \leq 3$, for all $x \in F$.
Let $x \in F$. 
Since $\F$ is a fence, we have $|\up x| \leq 3$, and then $\varphi(\up x) = \up\varphi(x)$ implies $|\varphi(F) \cap \up\varphi(x)| \leq 3$.
Dually, by Lemma~\ref{up-tails}, we have $\varphi(\down x) = \down\varphi(x)$, and so $|\varphi(F) \cap \down\varphi(x)| \leq 3$. 
\end{proof}

\begin{definition}
A double-pointed ordered set $\mathbf{T}$ has a
% \emph{strong down-tail}
\emph{pointed down-tail}
if there exists $\tau_1, \tau_2 \in T$ such that $(\tau_1,\tau_2)$ is a down-tail with $\tau_1 =\bot^\mathbf{T}$. Note that $\top^\mathbf{T}$ is still an arbitrary maximal element of~$\T$. 
In what follows, we will let $\tau_1^\mathbf{T}$ and $\tau_2^\mathbf{T}$ denote $\tau_1$ and $\tau_2$ as stated here.
\end{definition}

We draw special attention to the fact that if a double-pointed ordered set
$\mathbf{T}$ has a
% strong
pointed down-tail, then $\arrow$ entails a more specific construction (see
%tk combined two figures into one
% Figures~\ref{down-tail-figure} and~\ref{down-tail-fence}). 
Figure~~\ref{down-tail-fence}). 

\begin{figure}[ht]
\centering
\begin{tikzpicture}[scale=0.7]
\begin{scope}[shift={(-3.5,0)},scale=0.72]
\begin{scope}[rotate=-10,xscale=1.2,yscale=1.8]
\draw (0,0) circle [x radius = 1, y radius = 1] (0,0);
\coordinate (beta-1) at ($(0,0)+(40:1)$);
\coordinate (beta-2) at ($(0,0)+(70:1)$);
\coordinate (alpha) at ($(0,0)+(210:0.5)$);
\coordinate (alpha-1) at ($(0,0)+(230:1)$);
\coordinate (alpha-2) at ($(0,0)+(290:1)$);
\node[fence node] at (alpha-2) (alpha-drawa) {};
\node[fence node] at (beta-2) (beta-drawa) {};
\end{scope}
\end{scope}

\begin{scope}[scale=0.6,rotate=-10,xscale=1.2,yscale=1.8]
\draw (0,0) circle [rotate=75, x radius = 1, y radius = 1] (0,0);
\node[fence node] at (-1.3,0.5) (y1) {};
\node[fence node] at ($(y1)+(-135:1 and 1)$) (x1) {};
\coordinate (first) at ($(y1)+(-40:1 and 1)$);
\coordinate (second) at ($(y1)+(-30:1 and 1)$);
\coordinate (third) at ($(y1)+(-20:1 and 1)$);
\coordinate (beta-T) at ($(0,0)+(70:1)$);
\node[fence node] at (beta-T) (beta-T-draw){};
\draw (y1) -- (first);
\draw (y1) -- (second);
\draw (y1) -- (third);
\draw (x1) -- (y1);
\end{scope}

\begin{scope}[shift={(6.8,0)}]
\begin{scope}[shift={(-3.5,0)},scale=0.72]
\begin{scope}[rotate=-10,xscale=1.2,yscale=1.8]
\draw (0,0) circle [x radius = 1, y radius = 1] (0,0);
\coordinate (beta-1) at ($(0,0)+(40:1)$);
\coordinate (beta-2) at ($(0,0)+(70:1)$);
\coordinate (alpha) at ($(0,0)+(210:0.5)$);
\coordinate (alpha-1) at ($(0,0)+(230:1)$);
\coordinate (alpha-2) at ($(0,0)+(290:1)$);
\node[fence node] at (alpha-2) (alpha-drawc) {};
\node[fence node] at (beta-2) (beta-drawc) {};
\end{scope}
\end{scope}

\begin{scope}[scale=0.6,rotate=-10,xscale=1.2,yscale=1.8]
\draw (0,0) circle [rotate=75, x radius = 1, y radius = 1] (0,0);
\node[fence node] at (-1.3,0.5) (y) {};
\node[fence node] at ($(y)+(-135:1 and 1)$) (x) {};
\coordinate (first) at ($(y)+(-40:1 and 1)$);
\coordinate (second) at ($(y)+(-30:1 and 1)$);
\coordinate (third) at ($(y)+(-20:1 and 1)$);
\coordinate (beta-ST) at ($(0,0)+(70:1)$);
\node[fence node] at (beta-ST) (beta-ST-draw){};
\draw (y) -- (first);
\draw (y) -- (second);
\draw (y) -- (third);
\draw (x) -- (y);
\draw (beta-drawc) -- (x);
\end{scope}
\end{scope}

\node[rectangle,draw=none,fill=none,below=0.3 of alpha-drawa] (label-1) {$\bot^\mathbf{S}$};
\node[rectangle,draw=none,fill=none] at (x1 |- label-1) (label-2) {$\bot^\mathbf{T} = \tau_1^\mathbf{T}$};
\node[rectangle,draw=none,fill=none] at (alpha-drawc |- label-1) (label-3) {$\bot^\mathbf{S\arrow T}$};
\draw [densely dotted,very thin] (alpha-drawa) -- (label-1);
\draw [densely dotted,very thin] (x1) -- (label-2);
\draw [densely dotted,very thin] (alpha-drawc) -- (label-3);

\node[rectangle,draw=none,fill=none,above=0.3 of beta-drawa] (label-a) {$\top^\mathbf{S}$};
\node[rectangle,draw=none,fill=none] at (y1 |- label-a) (label-b) {$\tau_2^\mathbf{T}$};
\node[rectangle,draw=none,fill=none] at (beta-T-draw |- label-a) (label-d) {$\top^\mathbf{T}$};
\node[rectangle,draw=none,fill=none] at (beta-ST-draw |- label-a) (label-c) {$\top^\mathbf{S\arrow T}$};

\draw [densely dotted,very thin] (beta-drawa) -- (label-a);
\draw [densely dotted,very thin] (y1) -- (label-b);
\draw [densely dotted,very thin] (beta-ST-draw) -- (label-c);
\draw [densely dotted, very thin] (beta-T-draw) -- (label-d);
\end{tikzpicture}
%\caption{Special case: $\mathbf{S} \arrow \mathbf{T}$ when $\mathbf{T}$ has a 
%  strong
%pointed
%down-tail.}\label{down-tail-figure}
%\end{figure}

%\begin{figure}[ht]
%\centering
\begin{tikzpicture}[scale=0.6]
\begin{scope}[rotate=-20,xscale=1.2,yscale=1.8,scale=0.8,shift={(-2,0)}]
\draw (0,0) circle [x radius = 1, y radius = 1] (0,0);
\coordinate (beta-1) at ($(0,0)+(40:1)$);
\coordinate (beta-2) at ($(0,0)+(70:1)$);
\coordinate (alpha) at ($(0,0)+(210:0.5)$);
\coordinate (alpha-1) at ($(0,0)+(230:1)$);
\coordinate (alpha-2) at ($(0,0)+(290:1)$);
\node[fence node] at (alpha-2) (alpha-drawa) {};
\node[fence node] at (beta-2) (beta-drawa) {};
\end{scope}
\tikzfence{10}
\draw (beta-drawa) -- (fence1);
\node[rectangle, below = 0.3 of alpha-drawa] (label-a) {$\bot^\mathbf{S}$};
\node[rectangle] at (fence1 |- label-a) (label-b) {$\bot^\mathbf{F} = \tau_1^\mathbf{F}$};
\node[rectangle, above = 0.3 of beta-drawa] (label-c) {$\top^\mathbf{S}$};
\node[rectangle] at (fence2 |- label-c) (label-d) {$\tau_2^\mathbf{F}$};
\draw [densely dotted,very thin] (alpha-drawa) -- (label-a);
\draw [densely dotted,very thin] (fence1) -- (label-b);
\draw [densely dotted,very thin] (beta-drawa) -- (label-c);
\draw [densely dotted, very thin] (fence2) -- (label-d);
\end{tikzpicture}
\caption{Special case. Upper: generic $\mathbf{S} \arrow \mathbf{T}$ when
  $\mathbf{T}$ has a pointed down-tail. Lower: specific $\mathbf{S} \arrow
  \mathbf{F}$ with~$\mathbf{F}$ a fence.}\label{down-tail-fence} 
\end{figure}
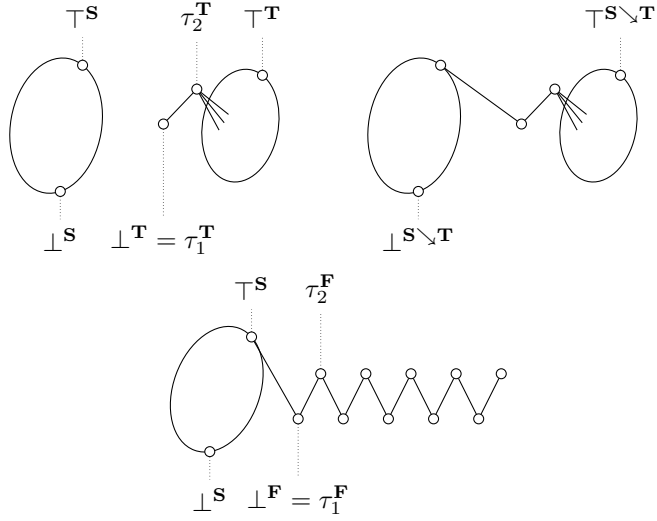

\begin{lemma}\label{only-element}Let $\mathbf{S}$ be a double-pointed ordered
set, let $\mathbf{T}$ be a double-pointed ordered set with a
% strong
pointed 
down-tail, and let $\varphi$ be an H\textsuperscript{+}-morphism on $\ov{\mathbf{S} \arrow \mathbf{T}}$.
If $\varphi(\top^\mathbf{S}) \notin \varphi(T)$, then ${\varphi(\tau_2^\mathbf{T}) \notin \varphi(T \comp \{\tau_2\})}$.
\end{lemma}
\begin{proof}
	Let $\tau_1 = \tau_1^\mathbf{T}$, let $\tau_2 = \tau_2^\mathbf{T}$, let $\mathbf{X} = \mathbf{S} \arrow \mathbf{T}$, and let $\Y = \operatorname{codom}(\varphi)$.
Assume that $\varphi(\top^\mathbf{S}) \notin \varphi(T)$. 
Suppose, by way of contradiction, that there is some $t \in T\comp\{\tau_2\}$ such that $\varphi(t) = \varphi(\tau_2)$.
Note that $\tau_1 \notin \down t$ because $t \neq \tau_2$.
But since $\tau_1$ is minimal, we have $\varphi(\tau_1) \in \min(Y)$. 
Then since $\varphi$ is order-preserving, we have $\varphi(\tau_1) \leq \varphi(\tau_2)$.
Thus, 
\[
\varphi(\tau_1) \in \min\nolimits_\Y(\down\varphi(\tau_2)) = \min\nolimits_\Y(\down\varphi(t)) = \varphi(\min\nolimits_\X(\down t)).
\]
So there exists $s \in \down t$ such that $\varphi(s) = \varphi(\tau_1)$ and $s \neq \tau_1$.
Note that $\up s \subseteq T$ because $s \neq \tau_1$.
By construction, we have $\top^\mathbf{S} > \bot^{\mathbf{T}} = \tau_1$, and then because $\varphi(\up \tau_1) = \up\varphi(\tau_1) = \up\varphi(s) = \varphi(\up s)$, it follows that there must be some $u \in \up s$ such that $\varphi(\top^\mathbf{S}) = \varphi(u)$.
By assumption, $u$ cannot be in $T$, but $u \in \up s \subseteq T$, a contradiction.
\end{proof}

The final leg of this subsection returns the focus to fences. From Lemma~\ref{image-of-arrow} and Lemma~\ref{image-of-fence} we obtain the next result.

\begin{lemma}\label{fence-corollary}
Let $\mathbf{S}$ be a connected double-pointed ordered set, let $\mathbf{F}$ be
a double-pointed fence with a
% strong
pointed
down-tail, and let $\varphi$ be a non-constant H\textsuperscript{+}-morphism on $\ov{\mathbf{S} \arrow \mathbf{F}}$.
If ${\varphi(\top^\mathbf{S}) \in \varphi(F)}$, then $\varphi(\ov{\mathbf{S} \arrow \mathbf{F}})$ is a fence.  
\end{lemma}

\begin{lemma}\label{one-to-one}Let $\mathbf{S}$ be a connected double-pointed
ordered set, let $\mathbf{F}$ be a double-pointed fence with a
% strong
pointed
down-tail, and let $\varphi$ be a non-constant H\textsuperscript{+}-morphism on $\ov{\mathbf{S} \arrow \mathbf{F}}$.
If $\varphi(\top^\mathbf{S}) \notin \varphi(F)$, then $\varphi{\restriction_F}$ is one-to-one.
\end{lemma}
\begin{proof}
	Assume $\varphi(\top^\mathbf{S}) \notin \varphi(F)$. 
If $|F| = 2$, the result holds because $\varphi$ is non-constant.
Assume $|F| \geq 3$.
Then there exists $\gamma \in F$ such that ${\down \tau_2^\mathbf{F} = \{\tau_1^\mathbf{F}, \tau_2^\mathbf{F},\gamma\}}$.
If $|F| = 3$, it needs only to be checked that $\varphi(\tau_1^\mathbf{F}) \neq \varphi(\gamma)$.
But since ${\up \gamma = \{\gamma,\tau_2^\mathbf{F}\}}$ and $\top^\mathbf{S} > \tau_1^\mathbf{F}$, if $\varphi(\tau_1^\mathbf{F}) = \varphi(\gamma)$, then $\varphi(\top^\mathbf{S}) \in \up \varphi(\tau_1^\mathbf{F}) = \up\varphi(\gamma) \subseteq \varphi(F)$, a contradiction.
So $\varphi(\tau^\mathbf{F}) \neq \varphi(\gamma)$. 
Let $|F| > 3$ and assume inductively that the result holds for all fences of a smaller size.
It is easy to see that $F' = F \comp\{\tau_1^\mathbf{F},\tau_2^\mathbf{F}\}$ 
forms a fence in which $\gamma$ is the minimum element of a down-tail. Let $\mathbf{F}'$ be the corresponding
double-pointed fence with a
% strong
pointed
down-tail in which 
$\bot^\mathbf{F'} = \gamma$ and $\top^\mathbf{F'}$ 
is chosen arbitrarily. 
Define $\mathbf{T}$ on $T = \{\tau_1^\mathbf{F},\tau_2^\mathbf{F}\}$ by $\bot^\mathbf{T} = \tau_1^\mathbf{F}$ and $\top^\mathbf{T} = \tau_2^\mathbf{F}$.
Then the underlying ordered sets of $\mathbf{F}$ and $\mathbf{T} \arrow \mathbf{F}'$ are equal.
Thus, the underlying ordered sets of $\mathbf{S} \arrow \mathbf{F}$ and $(\mathbf{S} \arrow \mathbf{T}) \arrow \mathbf{F}'$ are also equal.
Since $\varphi(\top^\mathbf{S}) \notin \varphi(F)$, it follows by Lemma~\ref{only-element} that $\varphi(\tau_2^\mathbf{F}) \notin \varphi({F}')$.
So by the inductive hypothesis, $\varphi$ is one-to-one on ${F}'$.
It remains to show that $\varphi(\tau_1^\mathbf{F}) \notin \varphi({F}')$.
But if this were not the case, then since $\top^\mathbf{S} > \tau_1^\mathbf{F}$, we would have $\varphi(\top^\mathbf{S}) \in \varphi(F)$, a contradiction.
\end{proof}

The next corollary is the key result used in Section~\ref{final-construction}.

\begin{corollary}\label{main-corollary}
Let $\mathbf{S}$ be a connected double-pointed ordered set, let $\mathbf{F}$ be
a double-pointed fence with a
% strong
pointed
down-tail, and let $\varphi$ be a non-constant H\textsuperscript{+}-morphism on $\ov{\mathbf{S} \arrow \mathbf{F}}$.
If $\varphi{\restriction_F}$ is not one-to-one, then $\varphi(\ov{\mathbf{S} \arrow \mathbf{F}})$ is a fence.  
\end{corollary}
\begin{proof}
By Lemma~\ref{one-to-one}, if $\varphi{\restriction_F}$ is not one-to-one, then $\varphi(\top^\mathbf{S}) \in \varphi(F)$, and then $\varphi(\ov{\mathbf{S} \arrow \mathbf{F}})$ is a fence by Lemma~\ref{fence-corollary}.
\end{proof}

\subsection{Expanding and distorting}\label{final-construction}
We will begin by providing a sufficient condition to apply the Non-splitting
Lemma~\ref{no-splitting}.  
By the end of this section we will have proved that the only splitting algebras
in $\mathsf{DH}$, $\mathsf{H}^+$, and $\mathsf{RDP}$ 
are the two-element and three-element chains.
Recall that $\delta x = \ps x$, and recall by Proposition~\ref{useful} that every finite subdirectly irreducible double Heyting algebra and \h{} is simple.

\begin{lemma}\label{var-sub}
Let $\mathbf{A}$ and $\mathbf{B}$ be finite simple \h{s} or double Heyting algebras.
Then $\mathbf{A} \in \Var(\mathbf{B})$ if and only if $\mathbf{A} \leq \mathbf{B}$.
\end{lemma}
\begin{proof}
Since both $\mathbf{A}$ and $\mathbf{B}$ are finite simple algebras, by
J\'onsson's Lemma, we have $\mathbf{A} \in \Var(\mathbf{B})$ if and only if
$\mathbf{A} \in \HS(\mathbf{B})$.
Both $\mathsf{DH}$ and $\mathsf{H}^+$ have the congruence extension property, so every non-trivial algebra in  $\HS(\mathbf{B})$ is in $\IS(\mathbf{B})$.
\end{proof}

Therefore, if $\CV$ is a variety of \h{s} or double Heyting algebras, condition~(2) of the Non-splitting Lemma~\ref{no-splitting} is implied by
\[
\forall i \in \omega\;\exists \mathbf{B}_i \in \CV\colon\; \mathbf{B}_i~\text{is simple,}~\mathbf{A} \nleq \mathbf{B}_i~\text{and}~\mathbf{B}_i \nvDash \delta^i\Delta_{\mathbf{A}} = 0.\tag{$\dagger$}\label{inline-ref}
\]

For convenience, in this paragraph we will speak only of \h{s} and take note that everything we say also applies to double Heyting algebras.
Let $\CV$ be a variety of \h{s} and let $\mathbf{A} \in \CV$.
From Proposition~\ref{si-simple}, if $\mathbf{A}$ is finite, then $\mathbf{A}$ is simple if and only if its Priestley dual is connected.

To simplify the presentation, from now on, given a double-pointed ordered set $\Sb$, we will use the notation $\Sb$ for both the double-pointed ordered set and its underlying ordered set~$\ov{\Sb}$.

The operation $\arrow$ clearly preserves connectedness, so an algebra of the form $\bu(\mathbf{X} \arrow \mathbf{Y})$ will be simple if and only if $\mathbf{X}$ and $\mathbf{Y}$ are connected double-pointed ordered sets.
This will ensure that the algebras we construct are simple.

\begin{definition}
Let $\mathbf{X}$ be a double-pointed ordered set. 
For each $i \geq 1$, let $X_i = X \times \{i\}$.
Now let $\mathbf{X}_i$ be the double-pointed ordered set with underlying set~$X_i$, with the order defined by $(x,i) \leq (y,i)$ if and only if $x \leq y$, and let $\bot^{\mathbf{X}_i} = (\bot,i)$ and $\top^{\mathbf{X}_i} = (\top,i)$.
For each $n \geq 1$, let 
\[
\mathbf{X}^{(n)} = \mathbf{X}_1 \arrow \mathbf{X}_{2} \arrow 
\cdots \arrow \mathbf{X}_{n-1} \arrow \mathbf{X}_{n}. 
\]
Note that $\bot^{\mathbf{X}^{(n)}} = (\bot,1)$ and $\top^{\mathbf{X}^{(n)}} = (\top,n)$.
See Figure~\ref{fig:x-n} for an illustration.
For two finite ordered sets $\X$ and $\Y$, we will say that $\X$ \emph{never maps onto $\Y$} if there is no surjective H\textsuperscript{+}-morphism $\varphi \colon \X \to \Y$. 
\end{definition}

\begin{figure}[ht]
\begin{subfigure}[b]{0.5\textwidth}\centering
\begin{tikzpicture}[scale=0.6,nodes={scale=0.6}]
\begin{scope}[rotate=-20,xscale=1.2,yscale=1.8,scale=0.8,shift={(-2,0)}]
\draw (0,0) circle [x radius = 1, y radius = 1] (0,0);
\coordinate (beta-1) at ($(0,0)+(40:1)$);
\coordinate (beta-2) at ($(0,0)+(70:1)$);
\coordinate (alpha) at ($(0,0)+(210:0.5)$);
\coordinate (alpha-1) at ($(0,0)+(230:1)$);
\coordinate (alpha-2) at ($(0,0)+(290:1)$);
\node[fence node] at (alpha-2) (alpha-drawa) {};
\node[fence node] at (beta-2) (beta-drawa) {};
\end{scope}
\end{tikzpicture}
\caption{$\mathbf{X}$}
\end{subfigure}
\begin{subfigure}[b]{0.33\textwidth}\centering
\begin{tikzpicture}[scale=0.7,nodes={scale=0.7}]
\tikzfence{8}
\node [draw=none] at (0,-0.9) {}; 
\end{tikzpicture}
 \caption{$\mathbf{F}$}
\end{subfigure}

\bigskip

\begin{subfigure}[b]{\textwidth}\centering
\begin{tikzpicture}[scale=0.6,nodes={scale=0.6}]
\begin{scope}[shift={(0,0)}]
\begin{scope}[rotate=-20,xscale=1.2,yscale=1.8,scale=0.8]
\draw (0,0) circle [x radius = 1, y radius = 1] (0,0);
\coordinate (beta-1) at ($(0,0)+(40:1)$);
\coordinate (beta-2) at ($(0,0)+(70:1)$);
\coordinate (alpha) at ($(0,0)+(210:0.5)$);
\coordinate (alpha-1) at ($(0,0)+(230:1)$);
\coordinate (alpha-2) at ($(0,0)+(290:1)$);
\node[fence node] at (alpha-2) (alpha-draw1) {};
\node[fence node] at (beta-2) (beta-draw1) {};
\end{scope}
\end{scope}

\begin{scope}[shift={(3,0)}]
\begin{scope}[rotate=-20,xscale=1.2,yscale=1.8,scale=0.8]
\draw (0,0) circle [x radius = 1, y radius = 1] (0,0);
\coordinate (beta-1) at ($(0,0)+(40:1)$);
\coordinate (beta-2) at ($(0,0)+(70:1)$);
\coordinate (alpha) at ($(0,0)+(210:0.5)$);
\coordinate (alpha-1) at ($(0,0)+(230:1)$);
\coordinate (alpha-2) at ($(0,0)+(290:1)$);
\node[fence node] at (alpha-2) (alpha-draw2) {};
\node[fence node] at (beta-2) (beta-draw2) {};
\end{scope}
\end{scope}

\begin{scope}[shift={(6,0)}]
\begin{scope}[rotate=-20,xscale=1.2,yscale=1.8,scale=0.8]
\draw (0,0) circle [x radius = 1, y radius = 1] (0,0);
\coordinate (beta-1) at ($(0,0)+(40:1)$);
\coordinate (beta-2) at ($(0,0)+(70:1)$);
\coordinate (alpha) at ($(0,0)+(210:0.5)$);
\coordinate (alpha-1) at ($(0,0)+(230:1)$);
\coordinate (alpha-2) at ($(0,0)+(290:1)$);
\node[fence node] at (alpha-2) (alpha-draw3) {};
\node[fence node] at (beta-2) (beta-draw3) {};
\end{scope}
\end{scope}

\begin{scope}[shift={(9,0)}]
\begin{scope}[rotate=-20,xscale=1.2,yscale=1.8,scale=0.8]
\draw (0,0) circle [x radius = 1, y radius = 1] (0,0);
\coordinate (beta-1) at ($(0,0)+(40:1)$);
\coordinate (beta-2) at ($(0,0)+(70:1)$);
\coordinate (alpha) at ($(0,0)+(210:0.5)$);
\coordinate (alpha-1) at ($(0,0)+(230:1)$);
\coordinate (alpha-2) at ($(0,0)+(290:1)$);
\node[fence node] at (alpha-2) (alpha-draw4) {};
\node[fence node] at (beta-2) (beta-draw4) {};
\end{scope}
\end{scope}
\draw (beta-draw1) -- (alpha-draw2);
\draw (beta-draw2) -- (alpha-draw3);
\draw (beta-draw3) -- (alpha-draw4);
\end{tikzpicture}
\caption{$\mathbf{X}^{(4)}$}\label{fig:x-n}
\end{subfigure}

\bigskip

\begin{subfigure}{\textwidth}\centering
\begin{tikzpicture}[scale=0.6,nodes={scale=0.6}]

\begin{scope}[shift={(0,0)}]
\begin{scope}[rotate=-20,xscale=1.2,yscale=1.8,scale=0.8]
\draw (0,0) circle [x radius = 1, y radius = 1] (0,0);
\coordinate (beta-1) at ($(0,0)+(40:1)$);
\coordinate (beta-2) at ($(0,0)+(70:1)$);
\coordinate (alpha) at ($(0,0)+(210:0.5)$);
\coordinate (alpha-1) at ($(0,0)+(230:1)$);
\coordinate (alpha-2) at ($(0,0)+(290:1)$);
\node[fence node] at (alpha-2) (alpha-draw1) {};
\node[fence node] at (beta-2) (beta-draw1) {};
\end{scope}
\end{scope}

\begin{scope}[shift={(3,0)}]
\begin{scope}[rotate=-20,xscale=1.2,yscale=1.8,scale=0.8]
\draw (0,0) circle [x radius = 1, y radius = 1] (0,0);
\coordinate (beta-1) at ($(0,0)+(40:1)$);
\coordinate (beta-2) at ($(0,0)+(70:1)$);
\coordinate (alpha) at ($(0,0)+(210:0.5)$);
\coordinate (alpha-1) at ($(0,0)+(230:1)$);
\coordinate (alpha-2) at ($(0,0)+(290:1)$);
\node[fence node] at (alpha-2) (alpha-draw2) {};
\node[fence node] at (beta-2) (beta-draw2) {};
\end{scope}
\end{scope}

\begin{scope}[shift={(6,0)}]
\begin{scope}[rotate=-20,xscale=1.2,yscale=1.8,scale=0.8]
\draw (0,0) circle [x radius = 1, y radius = 1] (0,0);
\coordinate (beta-1) at ($(0,0)+(40:1)$);
\coordinate (beta-2) at ($(0,0)+(70:1)$);
\coordinate (alpha) at ($(0,0)+(210:0.5)$);
\coordinate (alpha-1) at ($(0,0)+(230:1)$);
\coordinate (alpha-2) at ($(0,0)+(290:1)$);
\node[fence node] at (alpha-2) (alpha-draw3) {};
\node[fence node] at (beta-2) (beta-draw3) {};
\end{scope}
\end{scope}

\begin{scope}[shift={(9,0)}]
\begin{scope}[rotate=-20,xscale=1.2,yscale=1.8,scale=0.8]
\draw (0,0) circle [x radius = 1, y radius = 1] (0,0);
\coordinate (beta-1) at ($(0,0)+(40:1)$);
\coordinate (beta-2) at ($(0,0)+(70:1)$);
\coordinate (alpha) at ($(0,0)+(210:0.5)$);
\coordinate (alpha-1) at ($(0,0)+(230:1)$);
\coordinate (alpha-2) at ($(0,0)+(290:1)$);
\node[fence node] at (alpha-2) (alpha-draw4) {};
\node[fence node] at (beta-2) (beta-draw4) {};
\end{scope}
\end{scope}
\begin{scope}[shift={(11,-0.5)}]
\tikzfence{8}
\end{scope}
\draw (beta-draw1) -- (alpha-draw2);
\draw (beta-draw2) -- (alpha-draw3);
\draw (beta-draw3) -- (alpha-draw4);
\draw (beta-draw4) -- (fence1);
\end{tikzpicture}
\caption{$\mathbf{X}^{(4)} \arrow \mathbf{F}$}
\end{subfigure}
\caption{}\label{fig:all-built}
\end{figure}

If $\X$ never maps onto $\Y$, then in the dual this means that, as \h{s}, we have $\bu(\Y) \nleq \bu(\X)$.
This implies that $\bu(\Y) \nleq \bu(\X)$ when treating them as double Heyting algebras. Thus, we consider only H\textsuperscript{+}-morphisms in what follows.

\begin{proposition}\label{lemma-b}
Let $\mathbf{X}$ be a finite double-pointed ordered set and let $\mathbf{F}$ be
a double-pointed fence with a
% strong
pointed
down-tail.
Assume that $\mathbf{X}$ is not a fence and that $|F| > |X|$. 
Then, for all $i \geq 1$, the ordered set $\mathbf{X}^{(i)}\arrow \mathbf{F}$ never maps onto $\mathbf{X}$.
\end{proposition}
\begin{proof}
	Because $|F| > |X|$, by the pigeonhole principle, if $\varphi \colon \mathbf{X}^{(i)} \arrow \mathbf{F} \to \mathbf{X}$ is an \h[morphism]{}, then it is not one-to-one when restricted to $F$.
Hence, by Corollary~\ref{main-corollary}, $\varphi(\mathbf{X}^{(i)} \arrow \mathbf{F})$ is a fence.
Since $\mathbf{X}$ is not a fence, $\varphi$ is not surjective.
\end{proof}

This supplies us with our candidate algebras for condition~\eqref{inline-ref}, provided that the dual of the algebra is not a fence.
We require a special argument otherwise.
If $\mathbf{X}$ is a fence that has only
down-tails, then we can choose a large enough fence $\mathbf{F}$ with one up-tail.
In this case, by Lemma~\ref{up-tails}, for all $i \geq 1$, if $\varphi$ is a surjective \h[morphism]{} from $\varphi(\mathbf{X}^{(i)}\arrow \mathbf{F})$ to $\mathbf{X}$, then there is an up-tail in $\mathbf{X}$.
But $\mathbf{X}$ has none, so Proposition~\ref{lemma-b} holds in this case as well.
Similarly, if $\mathbf{X}$ is a fence with no down-tails, we can choose $\mathbf{F}$ so that it has only down-tails, and by Lemma~\ref{tails-to-tails}, the result still holds.
If $\mathbf{X}$ is a two-element fence, or in other words, a two-element chain, then $\bu(\mathbf{X}) \cong \mathbf{3}$ which we have already seen is a splitting algebra.
Thus, the only case that remains is if $\mathbf{X}$ is a fence with at least 3 elements, exactly one up-tail, and exactly one down-tail. 

\begin{proposition}\label{lemma-c}
Let $\mathbf{X}$ be a double-pointed 
fence and assume $|X| > 2$.
Then there is a fence $\mathbf{F}$ such that, for all $i \geq 1$, the ordered set $\mathbf{X}^{(i)} \arrow \mathbf{F}$ never maps onto~$\mathbf{X}$. 
\end{proposition}
\begin{proof}
We just discussed the case that $\mathbf{X}$ has no down-tails or no up-tails.
So assume that $\mathbf{X}$ has one up-tail and one down-tail.
Note that this implies that $|X| \neq 3$.
The elements of $\mathbf{X}$ have their order given by 
\[x_1 < x_2 > x_3 < \dots > x_{n-1} < x_n.\] 
Let $\mathbf{F}$ be a fence with $|X| + 1$ elements, with the order given by 
\[f_0 > f_1 < f_2 > f_3 < \dots > f_{n-1} < f_n,\] 
and let $\bot^\mathbf{F} = f_1$, with $\top^\mathbf{F}$ left arbitrary.
Let $n \geq 1$
and let $\varphi$ be a morphism from $\mathbf{X}^{(n)} \arrow \mathbf{F}$ to $\mathbf{X}$. 
Suppose, by way of contradiction, that $\varphi(\mathbf{X}^{(n)}\arrow \mathbf{F}) = \mathbf{X}$. 

The pair $(f_{n-1}, f_n)$ is an up-tail in $\mathbf{X}^{(i)}\arrow \mathbf{F}$, so Lemma~\ref{up-tails} tells us that $(\varphi(f_{n-1}),\varphi(f_n))$ is an up-tail in $\mathbf{X}$.
There is exactly one up-tail in $\mathbf{X}$, namely $(x_{n-1},x_n)$, so $\varphi(f_{n-1}) = x_{n-1}$ and $\varphi(f_n) = x_n$.
We will now prove inductively that $\varphi(f_k) = x_k$, for all $i \geq 1$.
Let $k > 1$ and assume that $\varphi(f_i) = x_i$, for all $i \geq k$.
We will show that $\varphi(f_{k-1}) = x_{k-1}$.
If $f_k$ is minimal, then, since $\varphi(f_k) = x_k$, we have $\varphi(\up f_k) = \up\varphi(f_k) = \up x_k = \{x_{k-1},x_k,x_{k+1}\}$, and because $\varphi(f_{k+1}) = x_{k+1}$, we must have $\varphi(f_{k-1}) = x_{k-1}$.
By Lemma~\ref{up-tails}, a dual argument holds if $f_k$ is maximal.
Thus, for all $i \geq 1$, we have $\varphi(f_i) = x_i$.
Since $(f_0,f_1)$ is an up-tail in $\mathbf{X}^{(i)}\arrow \mathbf{F}$, we must have that $(\varphi(f_0), \varphi(f_1))$ is an up-tail in $\mathbf{X}$.
But $\varphi(f_1) = x_1$, and $x_1$ is certainly not part of any up-tail in $\mathbf{X}$, a contradiction. 
\end{proof}

\begin{corollary}\label{fence-corollary2}
Let $\mathbf{A}$ be a finite simple \h{} and let $\mathbf{X} = \cat{F}_p(\mathbf{A})$ be the Priestley dual of $\A$.
Make $\X$ into a double-pointed ordered set by choosing $\bot^\mathbf{X}$ and $\top^\mathbf{X}$ arbitrarily.
Then there exists a fence $\mathbf{F}$
such that $\mathbf{A}$ does not embed into $\bu(\mathbf{X}^{(i)}\arrow \mathbf{F})$, for all $i \in \omega$.
\end{corollary}

The other part of condition~\eqref{inline-ref} is evaluating the term-diagram.
For this, the size and type of the fence is not important.
In fact, assuming it is a fence is not even necessary.
The following lemmas will aid in the calculation.

\begin{lemma}
Let $\mathbf{X}$ and $\mathbf{Y}$ be finite connected double-pointed ordered sets.
Let $U$ and $V$ be up-sets in $\mathbf{X}$ and then, for each $\star \in \{\join, \meet, \to, \dotdiv\}$, let $U \mathbin{\mathring{\star}} V$ be shorthand for $U \mathbin{\star}^{\bu(\X)} V$, and similarly for $\mathring{\dpc}U$.
Then, when evaluated in $\bu(\mathbf{X} \arrow \mathbf{Y})$, for each $\star \in \{\join,\meet,\dotdiv\}$, we have
\[
U \star V = U \mathbin{\mathring{\star}} V,
\]
for $\dpc$ we have
\[
 \dpc U = \mathring{\dpc}U \cup Y \cup \{\top^\mathbf{X}\},
\]
and for $\to$ we have
\[
U \to V = \begin{cases} (U \mathbin{\mathring{\to}}V)  \cup Y &\text{if $\top^\mathbf{X} \notin U \comp V$,} \\ 
(U \mathbin{\mathring{\to}} V) \cup Y\comp\{\bot^\mathbf{Y}\}&\text{otherwise.}\end{cases}
\]
\end{lemma}
\begin{proof}
First note that $U$ and $V$ are also up-sets in $\mathbf{X} \arrow \mathbf{Y}$.
Let $\Uparrow$ and $\Downarrow$ denote the operations $\up$ and $\down$ with respect to the order on~$\mathbf{X}$.
Recall by Lemma~\ref{dual-operations} that the operations listed above are given by:
\begin{alignat*}{2}
U \mathbin{\mathring{\join}} V &= U \cup V, &U \mathbin{\mathring{\meet}}V &= U \cap V,\\
U \mathbin{\mathring{\to}} V &= X \comp {\Downarrow}(U\comp V),&\qquad U \mathbin{\mathring{\dotdiv}} V &= {\Uparrow}(U\comp V),\\
\mathring{\dpc}U &= {\Uparrow}(X\comp U).
\end{alignat*}
The calculations for the lattice operations are trivial.
For $\dpc$, in $\bu(\mathbf{X} \arrow \mathbf{Y})$ we have
\begin{align*}
\dpc U & = \up \big[(X \cup Y)\comp U\big]
= \up\big[X \comp U \cup Y \comp U\big]
= \up(X\comp U) \cup \up Y.
\end{align*}
Since $\up(X \comp U)$ and $Y$ are disjoint, we have $\up(X \comp U) = {\Uparrow}(X\comp U) = \mathbin{\mathring{\dpc}}U$, and by construction, we have $\up Y = Y \cup \{\top^\mathbf{X}\}$.
Thus $\dpc U = \mathring{\dpc}U \cup Y \cup \{\top^\mathbf{X}\}$.
For $\dotdiv$, we have $U \dotdiv V = \up(U \comp V)$.
Since $U, V\subseteq X$, we have that $\up(U \comp V)$ and ${Y}$ are disjoint.
So $\up(U \comp V) = {\Uparrow}(U\comp V)$, which proves the claim.
For $\to$, we have
\begin{align*}
U \to V &= (X \cup Y) \comp\down(U \comp V)\\
&= \big[X \comp \down(U \comp V)\big] \cup \big[Y \comp \down(U \comp V)\big]\\
&= U \mathbin{\mathring{\to}} V \cup \big[Y \comp \down(U \comp V)\big].
\end{align*}
If $\top^\mathbf{X} \notin U\comp V$, then $Y \cap \down(U \comp V) = \varnothing$, and otherwise, $Y \cap \down(U \comp V) = \{\bot^\mathbf{Y}\}$.
So, 
\[
Y \comp \down(U \comp V) =  \begin{cases} Y &\text{if $\top^\mathbf{X} \notin U \comp V$,} \\ Y\comp\{\bot^\mathbf{Y}\}&\text{otherwise,}\end{cases}
\] 
completing the proof.
\end{proof}

\begin{lemma}
Let $\mathbf{X}$ and $\mathbf{Y}$ be finite connected double-pointed ordered sets.
Let $U$ and $V$ be up-sets in $\mathbf{X}$ and then, for each $\star \in \{\join, \meet, \to, \dotdiv\}$, let $U \mathbin{\mathring{\star}} V$ be shorthand for $U \mathbin{\star}^{\bu(\X)} V$, and similarly for $\mathring{\dpc}U$.
Then, when evaluated in $\bu(\mathbf{X} \arrow \mathbf{Y})$, for each $\star \in \{\join,\meet,\dotdiv\}$, we have
\[
(U \star V) \lr (U \mathbin{\mathring{\star}} V) = X \cup Y,
\]
for $\dpc$ we have
\[
\dpc U \lr \mathring{\dpc}U = \begin{cases}
X &\text{if $\top^\mathbf{X} \in \mathring{\dpc}U$},\\
X \comp \down\top^\mathbf{X}&\text{otherwise,}
\end{cases}
\]
and for $\to$ we have
\[
(U \to V) \lr (U \mathbin{\mathring{\to}} V) = {X}.
\]
\end{lemma}
\begin{proof}
The first part holds because $U \star V = U \mathbin{\mathring{\star}} V$ whenever $\star \in \{\join,\meet,\dotdiv\}$.
For $\dpc$, we have $\mathring{\dpc}U \subseteq \dpc U$, so
\begin{align*}
\dpc U \lr \mathring{\dpc}U &= \dpc U \to \mathring{\dpc}U = (X \cup Y) \comp \down(\dpc U \comp \mathring{\dpc}U).
\end{align*}
Now, 
\[
\down(\dpc U \comp \mathring{\dpc}U) = \begin{cases} 
Y & \text{if $\top^\mathbf{X} \in \mathring{\dpc}U$,} \\
Y \cup \{\top^\mathbf{X}\} &\text{otherwise.}
\end{cases}
\]
Hence,
\begin{align*}
\dpc U \to \mathring{\dpc}U &= \begin{cases}
(X \cup Y) \comp Y &\text{if $\top^\mathbf{X} \in \mathring{\dpc}U$},\\
(X \cup Y) \comp (Y \cup \down\top^\mathbf{X}) &\text{otherwise,}
\end{cases}\\
&= \begin{cases}
X &\text{if $\top^\mathbf{X} \in \mathring{\dpc}U$},\\
X \comp \down\top^\mathbf{X}&\text{otherwise,}
\end{cases}
\end{align*}
as required.
For $\to$, we have $U \mathbin{\mathring{\to}} V \subseteq U \to V$, so 
\[
(U \to V) \lr (U\mathbin{\mathring{\to}} V) = (U \to V) \to (U \mathbin{\mathring{\to}} V).
\]
First observe that
\[
(U \to V) \comp (U \mathbin{\mathring{\to}} V) =  \begin{cases} Y &\text{if $\top^\mathbf{X} \notin U \comp V$,} \\ Y\comp\{\bot^\mathbf{Y}\}&\text{otherwise,}\end{cases}
\]
and in either case we have $\down[(U \to V) \comp (U \mathbin{\mathring{\to}} V)] = Y$.
Hence,
\[
(U \to V) \to (U \mathbin{\mathring{\to}} V) = (X \cup Y)\comp \down[(U \to V) \comp (U \mathbin{\mathring{\to}} V)] = (X \cup Y) \comp Y = X,
\]
as claimed.
\end{proof}

The next lemma is now immediate.

\begin{lemma}\label{pre-yankov}
Let $\mathbf{X}$ and $\mathbf{Y}$ be double-pointed ordered sets.
Let $U$ and $V$ be up-sets in $\mathbf{X}$ and then, for each $\star \in \{\join, \meet, \to, \dotdiv\}$, let $U \mathbin{\mathring{\star}} V$ be shorthand for $U \mathbin{\star}^{\bu(\X)} V$, and similarly for $\mathring{\dpc}U$.
Let $\chi(U,V)$ denote the element of $\bu(\mathbf{X} \arrow \mathbf{Y})$ given by
\begin{multline*}
\chi(U,V) = [(U \mathbin{\mathring{\meet}} V)\lr (U \meet V)] \meet [(U \mathbin{\mathring{\join}} V)\lr (U \join V)] \\
\meet [(U \mathbin{\mathring{\to}} V)\lr (U \to V)] \meet [(U \mathbin{\mathring{\dotdiv}} V) \lr (U \dotdiv V)].
\end{multline*}
Then $\chi(U,V) = X$. 
Similarly, if $\chi^+(U,V)$ is given by
\begin{multline*}
\chi^+(U,V) = [(U \mathbin{\mathring{\meet}} V)\lr (U \meet V)] \meet [(U \mathbin{\mathring{\join}} V)\lr (U \join V)] \\
\meet [(U \mathbin{\mathring{\to}} V)\lr (U \to V)] \meet [(\mathring{\dpc} U \lr \dpc U)], 
\end{multline*}
then
\[
\chi^+(U,V) = \begin{cases}X &\text{if $\top^\mathbf{X} \in \mathring{\dpc}U,$}\\X \comp \down \top^\mathbf{X}&\text{otherwise.}\end{cases}\]
\end{lemma}

With these calculations established, we can now evaluate the term-diagram.
%We will include one extra assumption: 
For the remainder of this section, we will assume all double-pointed ordered sets satisfy $\bot^\mathbf{X} \leq \top^\mathbf{X}$.
This is not a problematic assumption, since we can always find such a pair of elements in any finite connected ordered set with two or more elements.

\begin{definition}
Let $\mathbf{X}$ be a finite connected double-pointed ordered set, assume that $\bot^\mathbf{X} \leq \top^\mathbf{\X}$, and let $\mathbf{A} = \bu(\X)$.
For each $i \in \omega$ and each $a \in A$, let $a_i = a \times \{i\}$, and then let $U_n(a)$ denote the element of $\mathbf{X}^{(n)}$ given by
\[
U_n(a) = \bigcup_{i \leq n}a_i.
\]
By assuming $\bot^\mathbf{X} \leq \top^\mathbf{X}$, we ensure that $U(a)$ is an up-set, for all $a \in A$.
\end{definition}

\begin{lemma}\label{big-calculation}
	Let $\mathbf{X}$ be a finite connected double-pointed ordered set and let ${n \in \omega}$.
The map $U\colon \bu(\mathbf{X}) \to \bu(\mathbf{X}^{(n)})$ given by $a \mapsto U_n(a)$ is a double Heyting algebra homomorphism.
\end{lemma}
\begin{proof}
We show that the map $h \colon \mathbf{X}^{(n)} \to \mathbf{X}$ given by $(x,i) \mapsto x$ is a double Heyting morphism whose dual is $U$.
Demanding that $\bot^\mathbf{X} \leq \top^\mathbf{X}$ ensures that $\down h(x) = h(\down x)$ and $\up h(x) = h(\up x)$.
Moreover, for each $a \in \uu(X)$, we have 
\[
h^{-1}(a) = \{(x,i) \in \mathbf{X}^{(n)} \mid h((x,i)) \in a\} = \{(x,i) \in \mathbf{X}^{(n)} \mid x \in a\} = U_n(a),
\]
which proves that $h$ is the dual of $U$.
\end{proof}

It follows immediately that the map $U$ is also an H\textsuperscript{+}-algebra homomorphism.
Now let $\mathbf{A}$ be a finite non-Boolean simple \h{} and let $\mathbf{X}$ be a double-pointed ordered set such that $\mathbf{A} \cong \bu(\mathbf{X})$. 
From Corollary~\ref{fence-corollary2}, there exists a finite connected double-pointed ordered set $\mathbf{F}$ such that $\mathbf{A} \nleq \bu(\mathbf{X}^{(i)}\arrow \mathbf{F})$, for all $i \geq 1$.
For each $i \in \omega$, let $\mathbf{C}_i = \bu(\mathbf{X}^{(i+2)}\arrow \mathbf{F})$.
Then $\mathbf{A} \nleq \mathbf{C}_i$.
The use of $i+2$ is necessary for Lemma~\ref{diagram-final} to work for \h{s} -- for double Heyting algebras, $i+1$ would suffice.
All that remains is to prove the following:
\[
\forall i \in \omega\colon\;\mathbf{C}_i \nvDash \delta^i\Delta_{\mathbf{A}} = 0.
\]

Let $\Delta^\mathsf{DH}_{\mathbf{A}}$ denote the term-diagram of $\mathbf{A}$ as a double Heyting algebra and let $\Delta^\mathsf{H^+}_{\mathbf{A}}$
denote the term-diagram of $\mathbf{A}$ as an \h{}.
Then,
\begin{align*}
\Delta^\mathsf{DH}_{\mathbf{A}}(\ov{x}) &= \bigwedge\{[x_{a \meet b}\lr (x_a \meet x_b)] \meet [x_{a \join b}\lr (x_a \join x_b)] \\
&\hspace*{2cm}\meet [x_{a \to b}\lr (x_a \to x_b)] \meet [x_{a \dotdiv b} \lr (x_a \dotdiv x_b)]\\
&\hspace*{5cm} \meet [x_0 \lr 0] \meet [x_1 \lr 1] \mid a,b \in A\},\\
\Delta^{\mathsf{H}^+}_{\mathbf{A}}(\ov{x}) &= \bigwedge\{[x_{a \meet b}\lr (x_a \meet x_b)] \meet [x_{a \join b}\lr (x_a \join x_b)] \\
&\hspace*{2cm}\meet [x_{a \to b}\lr (x_a \to x_b)] \meet [x_{\dpc a} \lr \dpc x_a]\\
&\hspace*{5cm} \meet [x_0 \lr 0] \meet [x_1 \lr 1] \mid a,b \in A\}.
\end{align*}

Notice that the next lemma does not rely on any particular choice of $\mathbf{Y}$.

\begin{lemma}\label{diagram-final}
Let $\mathbf{X}$ and $\mathbf{Y}$ be finite connected double-pointed ordered sets and let $\mathbf{A} = \bu(\mathbf{X})$.
For each $n \in \omega$, let $\mathbf{C}_n = \bu(\mathbf{X}^{(n+2)} \arrow \mathbf{Y})$.
Then $\mathbf{C}_n \nvDash \delta^n \Delta_{\mathbf{A}} = 0$.
\end{lemma}
\begin{proof}
Let $n \in \omega$ and, for convenience, let $\mathbf{Z} = \mathbf{X}^{(n+2)}$, so that $\mathbf{C}_n = \bu(\mathbf{Z} \arrow \mathbf{Y})$. 
Observe that because $\up Z = Z$ in $\mathbf{Z} \arrow \mathbf{Y}$, we have $\mathcal{U}(\mathbf{Z}) \subseteq \mathcal{U}(\mathbf{Z} \arrow \mathbf{Y}) = C_n$, and so $U_{n+2}(a) \in C_n$, for each $a \in A$.
Henceforth, we will omit $n+2$ from the subscript of $U$.
Map the variable $x_a$ into $C_n$ by $x_a \mapsto U(a)$.
As we did earlier, for each $* \in \{\join,\meet,\to,\dotdiv,\dpc\}$, let $\mathring{*}$ be shorthand for $*^{\bu(\mathbf{Z})}$.
Lemma~\ref{big-calculation} then tells us that $x_{a*b} = U({a*b}) = U(a)\mathbin{\mathring{*}}U(b)$ and $U(\dpc a) = \mathring{\dpc} U(a)$, for all $a,b \in A$.
We also have $U(0) = \varnothing$ and $U(1) = Z$.
Each $U(a)$ is a subset of $Z$, so Lemma~\ref{pre-yankov} applies.
Define $\chi$ and $\chi^+$ as in Lemma~\ref{pre-yankov}. 
By the definition of $\Delta_\mathbf{A}$, and evaluating it in $\bu(\mathbf{Z} \arrow \mathbf{Y})$, we then obtain
\begin{align*}
\Delta^\mathsf{DH}_{\mathbf{A}}(\ov{x}) &= \bigwedge\{\chi(U(a),U(b)) \meet \neg U(0) \meet U(1) \mid a,b \in A\} = Z,\\
\Delta^\mathsf{H^+}_{\mathbf{A}}(\ov{x}) 
&= \bigwedge\{\chi^+(U(a),U(b)) \meet \neg U(0) \meet U(1) \mid a,b \in A\} = Z \comp \down \top^\mathbf{Z},
\end{align*}
where the latter equality holds by choosing $U(a) = U(1)$.
In each case we have 
${{X}^{(n+1)} \subseteq \Delta_\mathbf{A}(\ov{x})}$.
Now write ${\mathbf{W} = \mathbf{Z}\arrow \mathbf{Y}}$.
In $\mathbf{C}_n$, we have 
\[ 
\delta^n {X}^{(n+1)} = {W}\comp (\down\up)^n({W} \comp {X}^{(n+1)}) = {W} \comp (\down\up)^n({X}_{n+2} \cup {Y}),
\]
and this is equal to $\varnothing$ if and only if 
$(\down\up)^n({X}_{n+2} \cup {Y}) = {W}$.
But, by construction, the leftmost part ${X}_1$ is not a subset of $(\down\up)^n({X}_{n+2} \cup {Y})$.
So $\delta^n{X}^{(n+1)} \neq \varnothing$. 
Then since $\delta$ is order-preserving, we have $\delta^n\Delta_{\mathbf{A}}(\ov{x}) \neq \varnothing$.
Hence, ${\mathbf{C}_n \nvDash \delta^n\Delta_{\mathbf{A}} = 0}$. 
\end{proof}

We will 
say that a variety $\CV$ of \h{s} or double Heyting algebras \emph{contains all fences} if $\bu(\F) \in \CV$, for every fence $\F$, and will
say that $\CV$ is \emph{finitarily closed under $\arrow$} provided that, for all double-pointed ordered sets
$\mathbf{X}$ and $\mathbf{Y}$, if $\bu(\mathbf{X})$ and $\bu(\mathbf{Y})$ are in $\CV$, then $\bu(\mathbf{X}\arrow\mathbf{Y})$ is in $\CV$ as well.
Let $\CV$ be a variety of double Heyting algebras or \h{s}
that is finitarily closed under $\arrow$ and contains all fences.
For every finite subdirectly irreducible algebra $\mathbf{A} \in \CV$ such that $|A| > 3$, we now have
\[
\forall i \in \omega\;\exists \mathbf{B}_i \in \CV\colon \mathbf{B}_i~\text{is simple,}~\mathbf{A} \nleq \mathbf{B}_i~\text{and}~\mathbf{B}_i \nvDash \delta^i\Delta_{\mathbf{A}} = 0.
\]
Recall that $\mathbf{2}$ is trivially a splitting algebra, and Theorem~\ref{3-subalgebra2} ensures that $\mathbf{3}$ is splitting.
Thus, by the Non-splitting Lemma~\ref{no-splitting}, we have proved the following theorem.

\begin{theorem}\label{thm:main}
Let $\CV$ be a variety of \h{s} or double Heyting algebras.
If $\CV$ is finitarily closed under $\arrow$ and contains all fences, then the
only finite splitting algebras in $\CV$ are $\mathbf{2}$ and $\mathbf{3}$. 
Moreover, if $\CV$ is generated by its finite members, then $\mathbf{2}$ and
$\mathbf{3}$ are the only splitting algebras in $\CV$.
\end{theorem}

It follows from the discussion in Section~\ref{sec:fep} that every
splitting algebra in each of $\mathsf{DH}$, $\mathsf{H}^+$, and $\mathsf{RDP}$
is finite. The following corollary is
% corollaries are
immediate. For $\CV = \mathsf{DH}$ it is due to Wolter~\cite{wolter}.

%tk squeezed 3 corollaries into one. 

\begin{corollary}\label{cor:DH-etal}
Let $\CV\in \{\mathsf{DH}, \mathsf{H}^+,\mathsf{RDP}\}$.   
The only splitting algebras in $\CV$ are $\mathbf{2}$ and $\mathbf{3}$.
\end{corollary}

% \begin{corollary}[Wolter~\cite{wolter}]\label{cor:DH}
% The only splitting algebras in $\mathsf{DH}$ are $\mathbf{2}$ and $\mathbf{3}$.
% \end{corollary}

% \begin{corollary}\label{cor:H+}
% The only splitting algebras in $\mathsf{H}^+$ are $\mathbf{2}$ and $\mathbf{3}$.
% \end{corollary}

% \begin{corollary}\label{cor:RDP}
% 	The only splitting algebras in $\mathsf{RDP}$ are $\mathbf{2}$ and $\mathbf{3}$.
% \end{corollary}

%%%%%%%%%%%%%%%%%%%%%%%%%%%%%%%%%%%%%%%%%%%%%%%%
\bibliographystyle{amsplain}
%%%%%%%%%%%%%%%%%%%%%%%%

%%%%%%%%%%%%%%%%%%%%%%%%%%%%%%%%%%%%%%%%%%%%%%%%

\end{document}